\documentclass{cmslatex}
\usepackage{latexsym, amssymb, enumerate, amsmath}

\usepackage[dvips]{graphicx}
\usepackage{placeins}

\renewcommand{\theequation}{\arabic{section}.\arabic{equation}}

\newtheorem{thm}{Theorem}[section]

\newtheorem{rem}[thm]{Remark}

\begin{document}

\title{Beyond pressureless gas dynamics :  Quadrature-based velocity moment models}


\author{Christophe  Chalons
\thanks {Universit\'e Paris Diderot-Paris 7 \& Laboratoire J.-L. Lions, U.M.R.~7598
UMPC, Bo\^ite courrier 187, 75252 Paris Cedex 05, France and Laboratoire EM2C - UPR CNRS 288, Ecole Centrale Paris, Grande Voie des Vignes, 92295 Chatenay-Malabry Cedex (chalons@math.jussieu.fr). This research 
was partially supported by  an ANR Young Investigator Award (French Agence Nationale de la Recherche), contract ANR-08-JCJC-0132-01 - INTOCS -  2009-2013}
\and Damien Kah
\thanks {Laboratoire EM2C - UPR CNRS 288, Ecole Centrale Paris, Grande Voie des Vignes, 92295 Chatenay-Malabry Cedex and IFP Energies Nouvelles, (damienk@stanford.edu). D. Kah was funded during his Ph.D. Thesis by an  IFPEN/EM2C CIFRE Ph.D. Grant.}
\and Marc Massot
\thanks {Laboratoire EM2C - UPR CNRS 288, Ecole Centrale Paris, Grande Voie des Vignes, 92295 Chatenay-Malabry Cedex (marc.massot@ecp.fr -  Corresponding author. Present address : mmassot@stanford.edu - Center for Turbulence 
Research, Stanford University). This research was supported by an ANR Young Inverstigator Award for M. Massot, ANR-05-JCJC-0013 - j\'eDYS  - 2005-2009. We would like to thank the Center for Turbulence Research at Stanford University and its director, Parviz Moin, for their kind invitation during the Summer Program 2010 where the present work was completed and Laboratory EM2C for supporting the visit of C. Chalons at Stanford as well as  Damien Kah's Master Thesis during which the study was initiated.}}
\maketitle

\begin{abstract}
Following the seminal work of F. Bouchut on zero pressure gas dynamics, extensively used for
gas particle-flows, the present contribution  investigates quadrature-based velocity moments models for kinetic equations in the framework of the infinite Knudsen number limit, that is, for dilute clouds of small particles where the collision or coalescence probability asymptotically approaches zero. Such models define a hierarchy based on the number of moments and associated quadrature nodes, the first level of which leads to pressureless gas dynamics.
We focus in particular on the four moment model where the flux closure is provided 
by a two-node quadrature in the velocity phase space and provides the right framework for studying both smooth and singular solutions.
The link with both the kinetic underlying equation as well as with zero pressure gas dynamics, i.e. the dynamics at the frontier of the moment space of order four,  is provided.  We define the notion of measure solutions and characterize
the mathematical structure of the resulting system of four PDEs. 
We exhibit a family of entropies and entropy fluxes and define the notion of entropic solution. We study the Riemann problem and provide entropic solutions in particular cases. This leads to  a rigorous link with the possibility of the system of macroscopic PDEs to allow particle trajectory crossing (PTC) in the framework of smooth solutions. 
Generalized $\delta$-shock solutions resulting from Riemann problem are also investigated. Finally, using a kinetic scheme proposed in the literature in several areas, we validate such a numerical approach and propose a dedicated extension at the frontier of the moment space
in the framework of both regular and singular solutions. This is a key issue for application fields where such an approach is extensively used.
\end{abstract}


\section{Introduction}\label{intro}

The physics of particles and droplets in a carrier gaseous flow field are described in 
many applications (fluidized beds, spray dynamics, alumina particles in rocket boosters, $\dots$)
by a number density function (NDF) satisfying a kinetic equation introduced by \cite{Williams58}. Solving such a kinetic equation relies either on a sample of discrete numerical parcels of particles through a Lagrangian--Monte-Carlo approach or on a moment approach resulting in a Eulerian system of conservation laws on velocity moments eventually conditioned on size. In the latter case investigated in the present contribution, the main difficulty for particle flows with high Knudsen numbers (i.e.\ weakly collisional flows), where the velocity distribution can be very far from equilibrium, is the closure of the convective transport at the macroscopic level. One way to proceed is to use quadrature-based moment methods where the higher-order moments required for closure are evaluated from the lower-order transported moments using  quadratures in the form of a sum of Dirac delta functions in velocity phase space (see \cite{cqmom10,kah10phd} and the references therein).

 Such an approach also allows for a well-behaved kinetic numerical scheme in the spirit of Bouchut \cite{Bouchut03} (See references from \cite{Laurent02,sympo08,SdcThesis,MassotVki,ctrSdc}  to \cite{kah10,freret10ctr,cqmom10}) where the fluxes in a cell-centered finite-volume formulation are directly evaluated from the knowledge of the quadrature abscissas and weights with guaranteed realizability conditions and singularity treatment. Such a quadrature approach and the related numerical methods have been shown to be able to capture particle trajectory crossing (PTC) in a Direct Numerical Simulation (DNS) context, where the distribution in the exact kinetic equation remains at all times in the form of a sum of Dirac delta functions. Such methods can be extended to partially high-order numerical schemes (\cite{vwpf10}). 
 
  In another component of the literature devoted to multiphase semiclassical limits of the Schr\"odinger equation \cite{jin03,gosse03,gosse05}, the series of Wigner measures obtained from the Wigner transform for studying the semiclassical limits can be shown to converge towards a measure solution of the Liouville equation. Such an equation naturally unfolds the caustics and can generate the proper multiphase solutions globally in time. Two approaches have been used to solve this equation with a moment approach, either the {\em Heaviside closure} \cite{brenier98} as it is called in \cite{jin03,gosse03}, or, the one which is related to the present work, the {\em delta closure} (see \cite{jin03,gosse03} and references therein). It
 leads to a weakly hyperbolic systems of conservation laws by taking moments of a Liouville equation exactly identical to the Williams-Boltzmann equation studies in gas-particle flows previously mentioned. Such approaches naturally degenerate towards the pressureless gas system of equation in the context of monokinetic velocity distributions \cite{MassotVki,kah10,kah10phd,runborg00}.
 
 Numerical algorithms in order to simulate such systems of conservations laws with the related {\em delta closure} or {\em quadrature-based closure} have been proposed in \cite{jin03,gosse03} and \cite{dfv08} independently, from the work for \cite{Bouchut03,Laurent02} using naturally kinetic scheme with finite volume methods. However, many issues are still to be tackled in order to reach fully high order numerical schemes that preserve the vector of moments inside or at the frontier of the moment space, thus leading to several possibilities of degeneration from a given number of abscissas to a lower one. In fact, such models are meant to capture a given level of complexity in the phase space which is fixed in advance by the number of moments and related quadrature nodes. In some particular situations, for perfectly controlled dynamics, it can be guaranteed that the solutions will remain smooth and consist in free boundary value (contact discontinuities) problem associated with switches between various numbers of quadrature abscissas. However, in most cases the numerical schemes have to tackle the possibility of singular solutions when the dynamics complexity goes beyond the one allowed by the model.
 In such cases the solution  of the resulting system of PDEs is the viscosity solution and does not reproduce the exact dynamics in phase space and measure solutions are expected, for which we need a precise framework.
 More specifically, even if for the pressureless gas system, 
 \cite{Bouchut94} had set the correct mathematical background to define general entropic solutions, such a work had not yet been performed for higher order moment methods in the cited publications and no rigorous link has been provided between these and zero pressure gas dynamics. This is the purpose of the present contribution for both theoretical and numerical points of view.
 
 The paper is organized as follows. First we introduce the {\em quadrature-based} or {\em delta closure} velocity moment models for 
kinetic equations and focus on the four moment model in order to
generalize to what can be done to
higher orders but which would be difficult to expose due to algebra complications.  The behavior at the frontier of the moment space is characterized as well as the mathematical structure of the system of conservation laws. We then define entropy conditions and provide, 
for smooth solution, the one-to-one kinetic-macroscopic relation.   We then tackle the Riemann problem and define entropy measure 
solutions. Three examples of piecewise linear and singular solutions are then provided for which we rigorously identify the entropic character of 
the solution and which are then reproduced numerically. 

\section{Quadrature-based velocity moment models for kinetic equations}

Consider the solution $f=f(t,x,v)$ of the free transport kinetic equation 
\begin{equation} \label{eq:cinetique}
\partial_t f + v \partial_x f = 0, \quad t> 0, \, x \in \mathbb{R}, \, v \in \mathbb{R},\qquad   f(0,x,v) = f_0(x,v)\\[-1.5ex]
\end{equation}
The exact solution is given by 
$f(t,x,v) = f(0,x-vt,v) = f_0(x-vt,v)$.
Defining the $i$-order moment
$M_i = \int_{v} f(t,x,v) v^i dv$,  $i=1,...,N, \quad N \in \mathbb{N}$,
the associated governing equations are easily obtained 
from (\ref{eq:cinetique}) after multiplication by $v^i$ and integration over $v$, 
and write
\begin{equation*}
\partial_t M_i + \partial_x M_{i+1} = 0, \quad i \geq 0. \\[-1.5ex]
\end{equation*}
For the sake of simplicity, but without any restriction, we will focus our attention hereafter 
on the four-moment model
\begin{equation} \label{eq:modele_bipic}
\left\{
\begin{array}{l}
\partial_t M_0 + \partial_x M_1 = 0, \\
\partial_t M_1 + \partial_x M_2 = 0, \\
\partial_t M_2 + \partial_x M_3 = 0, \\
\partial_t M_3 + \partial_x \overline{M_4} = 0.
\end{array}
\right. 
\end{equation}
It will be convenient to write (\ref{eq:modele_bipic}) under the following abstract form 
\begin{equation} \label{eq:modele_bipic_short}
\partial_t {\bf M} + \partial_x {\bf F}({\bf M}) = 0,
\end{equation}
with ${\bf M} = (M_0,M_1,M_2,M_3)^t$ and ${\bf F}({\bf M}) = (M_1,M_2,M_3,\overline{M_4})^t$. \\

\subsection{Quadrature inside the moment space}

This model is closed provided that $\overline{M_4}$ is defined as a function of ${\bf M}$.
In quadrature-based moment methods, the starting point to define this closure relation 
consists in representing the velocity distribution of $f(t,x,v)$ by a set of 
two Dirac delta functions, that is a two-node quadrature~: 
\begin{equation}
f(t,x,v) = \rho_1(t,x) \delta (v - v_1(x,t)) + \rho_2(t,x) \delta (v - v_2(x,t)),
\end{equation}
where the weights $\rho_1(t,x) > 0$, $\rho_2(t,x) > 0$ and the velocity abscissas $v_1(t,x)$, $v_2(t,x)$ 
are expected to be uniquely determined from the knowledge of ${\bf M}(x,t)$. 
Dropping the $(x,t)$-dependance to avoid cumbersome notations, such a function $f$ 
has exact moments of order $i=0,...,4$ given by $\rho_1 v_1^i + \rho_2 v_2^i$. The next step then naturally 
consists in setting 
\begin{equation} \label{eq:M4}
\overline{M_4} = \rho_1 v_1^4 + \rho_2 v_2^4
\end{equation}
where $\rho_1$, $\rho_2$ and $v_1$, $v_2$ are defined from ${\bf M}$ by the following nonlinear 
system~:
\begin{equation} \label{eq:quadrature}
\left\{
\begin{array}{l}
M_0 = \rho_1 + \rho_2, \\
M_1 = \rho_1 v_1 + \rho_2 v_2, \\
M_2 = \rho_1 v_1^2 + \rho_2 v_2^2, \\
M_3 = \rho_1 v_1^3 + \rho_2 v_2^3. \\
\end{array}
\right. 
\end{equation}
At last, it remains to prove that this system is well-posed, which is the matter of the next proposition. 
We refer to \cite{jin03,gosse03,dfv08} for the proof. \\

\begin{proposition}\label{prop_quadrature}
System (\ref{eq:modele_bipic_short})-(\ref{eq:M4})-(\ref{eq:quadrature}) is well-defined 
on the convex phase space $\Omega$, also called the moment space, given by
$$
\Omega = \{{\bf M} = (M_0,M_1,M_2,M_3)^t, M_0 > 0, M_0 M_2 - M_1^2 > 0\}.
$$
Moreover, setting ${\bf U} = (\rho_1,\rho_2,\rho_1 v_1,\rho_2 v_2)^t$, the function 
${\bf U} = {\bf U}({\bf M})$ is one-to-one and onto as soon as we set for instance $v_1 > v_2$.
Moreover we have $0 < \rho_1 < M_0$ and $0 < \rho_2 < M_0$. \\
\end{proposition}
\ \\
\noindent Proposition \ref{prop_quadrature} can be extended to the more general case of a $2k$-moment models, $k>1$. 
The velocity distribution
is represented in this situation by a set of $k$ Dirac delta functions, leading to $M_i = \sum_{j=1}^{k} \rho_j v_j^i$,
$i=0,...,2k-1$, and $\overline{M_{2k}} = \sum_{j=1}^{k} \rho_j v_j^{2k}$.  

\subsection{Hyperbolic structure inside the moment space}

The two-moment model, corresponding to $k=1$ (one-node quadrature) writes 
$$
\left\{
\begin{array}{l}
\partial_t \rho + \partial_x \rho v = 0, \\
\partial_t \rho v + \partial_x \rho v^2 = 0,
\end{array}
\right. 
$$
which is the well-known pressureless gas dynamics system. Recall that this model is  weakly hyperbolic 
(the jacobian matrix is not diagonalizable) with $v$ as unique eigenvalue, the characteristic field 
being linearly degenerate. Since there can be areas in the solution where a single quadrature node is sufficient at the frontier of the moment space in order to describe the dynamics, the solution in such zones will satisfy the previous  system of two conservation laws. However we will first work inside the moment space and leave the behavior at the frontier for the next subsection.

Actually, we will observe in the course of the next section that the four-moment model (\ref{eq:modele_bipic_short}) 
is equivalent for smooth solutions (only) to two decoupled pressureless gas dynamics systems 
associated with $(\rho_1,\rho_1 v_1)$ and $(\rho_2,\rho_2 v_2)$ respectively. Then (\ref{eq:modele_bipic_short}) 
is expected to admit two eigenvalues $v_1$ and $v_2$ and to be weakly hyperbolic with linearly degenerate 
characteristic fields, as stated in the following proposition. \\

\begin{proposition} (\cite{jin03,gosse03}) \label{prop_hyper}
System (\ref{eq:modele_bipic_short})-(\ref{eq:M4})-(\ref{eq:quadrature}) is weakly hyperbolic 
on $\Omega$ and admits the two eigenvalues $v_1$ and $v_2$, $v_1 \neq v_2$. The associated characteristic fields 
are linearly degenerate. \\
\end{proposition}
\begin{proof} For the sake of completeness, we propose here a direct proof of the eigenvalues 
$v_1$ and $v_2$ of (\ref{eq:modele_bipic_short})-(\ref{eq:M4})-(\ref{eq:quadrature}).
By (\ref{eq:quadrature}), we first easily get 
$$
\left\{
\begin{array}{l}
M_0 = \rho_1 + \rho_2, \\
M_1 - v_1 M_0 = \rho_2 (v_2-v_1), \\
M_2 - v_1 M_1 = \rho_2 v_2 (v_2 - v_1), \\
M_3 - v_1 M_2 = \rho_2 v_2^2(v_2-v_1), \\
\end{array}
\right. 
$$
and then, setting $\sigma_0 = v_1 v_2$ and $\sigma_1 = -(v_1+v_2)$,
$$
\left(
\begin{array}{cc}
M_0 & M_1 \\
M_1 & M_2 \\
\end{array}
\right) \,
\left(
\begin{array}{c}
\sigma_0 \\
\sigma_1 \\
\end{array}
\right) \, =
- 
\left(
\begin{array}{c}
M_2 \\
M_3 \\
\end{array}
\right). 
$$
This system is invertible in the phase space $\Omega$ 
($M_0 M_2 - M_1^2 \neq 0$) and uniquely defines $\sigma_0$ and $\sigma_1$ with respect to ${\bf M}$~:
\begin{equation} \label{eq:sigma01}
\left(
\begin{array}{c}
\sigma_0 \\
\sigma_1 \\
\end{array}
\right) = 
\frac{1}{M_0 M_2 - M_1^2}
\left(
\begin{array}{c}
M_1 M_3 - M_2^2 \\
M_1 M_2 - M_0 M_3 \\
\end{array}
\right).
\end{equation}
Then, we have 
\begin{equation} \label{eq:M4_sigma01M}
\begin{array}{rcl}
 {\overline{M_4}} &=& \rho_1 v_1^4 + \rho_2 v_2^4 \\
 &=&\rho_1 v_1^3 v_1 + \rho_2 v_2^3 v_2 \\  
 &=&(\rho_1 v_1^3 + \rho_2 v_2^3) (v_1 + v_2) - (\rho_1 v_1^2 + \rho_2 v_2^2) v_1 v_2 \\
 &=&-M_2 \sigma_0 - M_3 \sigma_1, \\
\end{array}
\end{equation}
which finally gives ${\overline{M_4}}$ with respect to ${\bf M}$. 
The Jacobian matrix ${\bf J} = \nabla_{{\bf M}} {\bf F}$ is given by 
$$
{\bf J} = \left(
\begin{array}{cccc}
0 & 1 & 0 & 0 \\
0 & 0 & 1 & 0 \\
0 & 0 & 0 & 1 \\
a & b & c & d \\
\end{array}
\right) \quad \mbox{with} \quad
\left\{
\begin{array}{l}
a = \partial_{M_0} {\overline{M_4}}, \\
b = \partial_{M_1} {\overline{M_4}}, \\
c = \partial_{M_2} {\overline{M_4}}, \\
d = \partial_{M_3} {\overline{M_4}}. \\
\end{array}
\right.
$$
Using (\ref{eq:sigma01}) and (\ref{eq:M4_sigma01M}), the calculations of the last row coefficients eventually lead 
to 
$$
{\bf J} = \left(
\begin{array}{cccc}
0 & 1 & 0 & 0 \\
0 & 0 & 1 & 0 \\
0 & 0 & 0 & 1 \\
-\sigma_0^2 & -2\sigma_0 \sigma_1 & -2\sigma_0-\sigma_1^2 & -2\sigma_1 \\
\end{array}
\right).
$$
Finally, the characteristic polynomial $p(\lambda)$ of ${\bf J}$ is easily shown to equal  
\begin{equation*}
p(\lambda) = (\lambda - v_1)^2 (\lambda - v_2)^2.\\[-1.5ex]
\end{equation*}
This concludes the proof. 
\end{proof}

Propositions \ref{prop_quadrature} and \ref{prop_hyper} show that System (\ref{eq:modele_bipic_short})-(\ref{eq:M4})-(\ref{eq:quadrature}) 
is well-defined and weakly hyperbolic only on 
$\Omega$, which gives in particular $v_1 \neq v_2$ in the interior of the moment space. 
At a first sight, this might appear to be 
restrictive in the sense that one of the main objectives of the model is to allow particle trajectory 
crossing, that is in particular to deal with initial data consisting of two colliding particle packets such that 
$v_1=v_2$ at each point initially  (see for instance Section 7). Thus, in the last part of the present section, we characterize the behavior at the frontier 
$\Gamma $ 
of the moment space when $M_0>0$~: $\Gamma =  \{{\bf M} = (M_0,M_1,M_2,M_3)^t, M_0 > 0, M_0 M_2 - M_1^2 = 0\}$.

\subsection{Behavior at the frontier of the moment space}
\label{border}

As mentioned previously, it is rather natural to envision the coexistence, in a single smooth moment solution, of zones where the number of quadrature nodes are different. More specifically, we will examine the coexistence of zones where only one quadrature node is needed ($v_1 = v_2$), that is where 
 ${\bf M} = (M_0,M_1,M_2,M_3)^t$ with $M_0 > 0$ and $M_0 M_2 - M_1^2 = 0$, and zones where $M_0 M_2 - M_1^2>0$ which are inside the 
moment space $\Omega$, whereas the vector of moments are smooth everywhere.

After easy calculations in terms of $\rho_1$, $\rho_2$, $v_1$ and $v_2$, the latter equality $M_0 M_2 - M_1^2 = 0$ writes
$\rho_1 \rho_2 (v_1-v_2)^2=0$, 
so that if the vector $(\rho_1,\rho_2,\rho_1 v_1,\rho_2 v_2)^t$ exists, this actually corresponds to the 
case $v:=v_1=v_2$ (still under the assumption $\rho_1 \neq 0$ and $\rho_2 \neq 0$) or to the case where one of the weights is zero. We also note that in both cases $M_k=M_0 v^k$, whatever $k$ in this case, so that the whole set of moments should be provided once $M_0$ and $M_1$ are given, in close connection  to the case of pressureless gas dynamics.
There are in fact two possibilities~: \\
- either ${\bf M} = (M_0,M_1,M_2,M_3)^t$ is such that $M_0 M_3 - M_1 M_2 \neq 0$~: in this 
case (\ref{eq:quadrature}) cannot be solved and the vector $(\rho_1,\rho_2,\rho_1 v_1,\rho_2 v_2)^t$ does not exist, \\
- or ${\bf M} = (M_0,M_1,M_2,M_3)^t$ is also such that $M_0 M_3 - M_1 M_2 = 0$~: in this case 
(\ref{eq:quadrature}) can be solved and we have $v=v_1=v_2=M_1/M_0$, together with $\rho_1$ and $\rho_2$ 
defined by the one-parameter equation $\rho_1+\rho_2=M_0$. As we will see just below, the 
choice $\rho_1=\rho_2=M_0/2$ is the most natural one when we have to deal with an isolated point at the frontier of the moment space.\\
\ \\
In order to justify the choice $\rho_1=\rho_2=M_0/2$, we first observe that both conditions
\vspace{-0.2cm}
$$
\left\{
\begin{array}{l}
 M_0 M_2 - M_1^2 = 0, \\
 M_0 M_3 - M_1 M_2 = 0,\\[-1ex]
\end{array}
\right.
$$
are equivalent to conditions
$$
\left\{
\begin{array}{l}
 e = 0, \\
 q = 0,
\end{array}
\right., \quad \text{where} \quad
\left\{
\begin{array}{l}
e = M_0 M_2 - M_1^2, \\
 q = (M_3M_0^2-M_1^3)-3M_1(M_0 M_2 - M_1^2),
\end{array}
\right.
$$
and we consider $\rho_1$, $\rho_2$, $v_1$ and 
$v_2$ as functions of $(M_0, M_1, q, e)$ with $M_0>0$, and $e>0$\footnote{The definitions of $e$ and $q$ naturally 
comes out after noticing that setting 
$\overline{\rho}_1 = \frac{\rho_1}{M_0},$
$\overline{\rho}_2 = \frac{\rho_2}{M_0}$, 
$\overline{v}_1 = v_1-\frac{M_1}{M_0}$, 
$\overline{v}_2 = v_2-\frac{M_1}{M_0}$,
solving (\ref{eq:quadrature}) is equivalent to solving
$$
\left\{
\begin{array}{l}
1 = \overline{\rho}_1 + \overline{\rho}_2, \\
0 = \overline{\rho}_1  \overline{v}_1 + \overline{\rho}_2  \overline{v}_2, \\
e = \overline{\rho}_1 \overline{v}_1^2 + \overline{\rho}_2  \overline{v}_2^2, \\
q = \overline{\rho}_1 \overline{v}_1^3 + \overline{\rho}_2 \overline{v}_2^3,
\end{array}
\right.
$$
with $e = (M_0 M_2-M_1^2)/M_0^2$ and $q = ((M_3M_0^2-M_1^3)-3M_1(M_0M_2-M_1^2))/M_0^3$.
}. We
then propose to study the asymptotic behavior of these functions when $e \to 0^+$, considering that $M_0 > 0$, $M_1$ and $q$ 
are fixed. 
Note that  $\Gamma =  \{(M_0,M_1,e,q)^t, M_0 > 0, e = 0\}$.
We get the following result. \\
\begin{lemma} \label{lemma23}
Let be given $M_0 > 0$, $M_1$ and $q$. Then we have 
$$
\lim_{e\to0^+} \rho_2 = 
\left\{
\begin{array}{rcl}
 M_0 & \mbox{if} & q > 0, \\
 0 & \mbox{if} & q < 0, \\
 \frac{M_0}{2} & \mbox{if} & q = 0, \\
\end{array}
\right.
\quad \quad 
\lim_{e\to0^+} \rho_1 = 
\left\{
\begin{array}{rcl}
 0 & \mbox{if} & q > 0, \\
 M_0 & \mbox{if} & q < 0, \\
 \frac{M_0}{2} & \mbox{if} & q = 0, \\
\end{array}
\right.
$$
$$
\lim_{e\to0^+} v_2 = 
\left\{
\begin{array}{rcl}
 \frac{M_1}{M_0} & \mbox{if} & q > 0, \\
 -\infty & \mbox{if} & q < 0, \\
 \frac{M_1}{M_0} & \mbox{if} & q = 0, \\
\end{array}
\right.
\quad \quad 
\lim_{e\to0^+} v_1 = 
\left\{
\begin{array}{rcl}
 +\infty & \mbox{if} & q > 0, \\
 \frac{M_1}{M_0} & \mbox{if} & q < 0, \\
 \frac{M_1}{M_0} & \mbox{if} & q = 0, \\
\end{array}
\right.
$$
$$
\lim_{e\to0^+} \rho_2 v_2 = 
\left\{
\begin{array}{rcl}
 M_1 & \mbox{if} & q > 0, \\
 0 & \mbox{if} & q < 0, \\
 \frac{M_1}{2} & \mbox{if} & q = 0, \\
\end{array}
\right.
\quad \quad 
\lim_{e\to0^+} \rho_1 v_1 = 
\left\{
\begin{array}{rcl}
 0 & \mbox{if} & q > 0, \\
 M_1 & \mbox{if} & q < 0, \\
 \frac{M_1}{2} & \mbox{if} & q = 0. \\
\end{array}
\right.
$$
\end{lemma}
\begin{proof} The admissible change of variables $(M_0,M_1,M_2,M_3) \to (M_0,M_1,e,q)$ allows to write after easy calculations
$$
v_1 = \frac{M_1}{M_0} + \frac{q +\sqrt{q^2+4e^3} }{2M_0e}, \quad
\rho_2 = \frac{M_0\, e(v_1M_0-M_1)}{\sqrt{q^2+4e^3}},
$$
and if $q\neq0$
$$
v_1 = \frac{M_1}{M_0} + \frac{q}{M_0\,e}\,\frac{\big( 1+sign(q)\sqrt{1+4e^3/q^2} \big)}{2}, \quad
\rho_2 = \frac{M_0\,e}{q}\,\left(\frac{v_1M_0-M_1}{sign(q)\sqrt{1+ 4e^3/q^2}}\right),
$$
where we have set
\vspace{-0.5cm}
$$
sign(q) = 
\left\{
\begin{array}{rcl}
 1 & \mbox{if} & q > 0, \\
 -1 & \mbox{if} & q < 0.
\end{array}
\right.
$$
It is then an easy matter to get the expected results distinguishing between the three cases $q<0$, $q>0$ 
and $q=0$. It is then clear by a continuity argument that the proposed choice $\rho_1=\rho_2=M_0/2$ when $e=q=0$ 
is actually natural.
\end{proof}

An important consequence of this lemma is that in the half plane $e>0$, the region close to the frontier $\Gamma$ {\it for a non-zero $q$} corresponds to abscissas going to infinity with arbitrary small weights. Moreover,  when the velocity distributions at the kinetic level have compact support in the initial distribution, such a property will be preserved in the dynamics of the system 
and we want to be able to switch continuously from two-node to one-node quadrature without pathological behavior on  abscissas and weights.

Let us provide a first example where such a behavior is present. We consider a path in the moment space parametrized by the variable $x$, such that $\rho_1=x^3$,  $\rho_2=1$, $v_1=1/x$ and $v_2=0$. As $x$ approaches zero, the smooth moment vector ${\bf M} = (1+x^3,x^2,x,1)^t$ has a very regular limit at the frontier of the moment space along the lines presented before with an unbounded abscissa. Indeed we have here $e=x$ and $q=1-x^3$ approaches the fixed non-zero value of $1$.
 Note that if we replace the first weight by $\rho_1=x^4$, we still have an unbounded abscissa even if we converge toward the 
point $(0,0)$ in the $(e,q)$ plane ($e=x^2, q=x(1-x^4)$).

We thus have to find a framework in a subset of the plane $(e,q)$ such that the limits are better behaved. A natural choice presented above is the line $q=0$ but it is too restrictive. In order to naturally introduce 
the relevant subset of $\Omega$, let us consider the other example with  $\rho_1=\alpha x^\beta$,  
$\rho_2=1$, $v_1=\gamma x^\delta$ and $v_2=0$, with $\alpha>0$, $\gamma>0$, $\beta\ge0$ and $\delta\ge0$. As $x$ approaches zero, the moment vector ${\bf M} = (1+\alpha x^\beta,\alpha\gamma x^{\beta+\delta},\alpha\gamma^2 x^{\beta+2\delta},\alpha\gamma^3 x^{\beta+3\delta})^t$ 
reaches the frontier $\Gamma$ of the moment space\footnote{The corresponding values of $e$ and $q$ are $e=\alpha \gamma^2 x^{\beta+2\delta}$ and $q=\alpha\gamma^3 x^{\beta + 3\delta}(1-\alpha x^\beta)$.}. Two cases are interesting; firstly when $\beta=0$, we reach the point $(0,0)$ in the $(e,q)$ plane asymptotically along the line $q/e^{3/2}=(1-\alpha)/\alpha^{1/2}$ and no weight is approaching zero, whereas the two abscissas are joining (see formulas in the proof above). 
Secondly, when $\delta=0$, one of the weights is reaching zero, whereas the two abscissas remain different at a distance of $\gamma$ at the limit and we reach the point $(0,0)$ in the $(e,q)$ plane asymptotically along the line $q/(M_0\,e)=\gamma$ 
at the limit $x \to 0$ (see again formulas in the proof above). 
We will prove in the following proposition that the proper framework is a symmetric cone in the $(e,q)$ plane centered at the point $(0,0)$ corresponding the $\vert q/(M_0\,e) \vert \le \eta$, where $\eta$ is a measure of the maximal distance allowed between the two abscissas.\\

\begin{definition}[Regular path]
We define a regular path parametrized by $x$ in the moment space, ${\bf M}_x = (M_{0\,x}, M_{1\,x}, M_{2\,x} M_{3\,x})$ which admits a limit as $x$ goes to zero and is at least ${\mathcal C}^1$ up to 
the limit  ${\bf M}_0$. Moreover, we define its reduced second and third order moments $e_x = M_{0\,x}M_{2\,x} -M_{1\,x}^2$ and $q_x = (M_{3\,x}M_{0\,x}-M_{1\,x}^3) - 3M_{1\,x}
(M_{0\,x}M_{2\,x}-M_{1\,x}^2)$.
Its limit further satisfies $e_0=e_{x=0}=0$ and we assume $e_{x>0} >0$,  $M_{0\,x}>\nu>0$, $\vert q_x/(M_{0\,x}\,e_x)\vert \le \eta$, where $\eta>0$. 
\end{definition}
\\
\begin{proposition}\label{cone}
We then have the following properties~:
\begin{itemize}
\item $\lim_{x\to 0} q_x =q_0= 0$.
\item the weights and abscissas admit limits $\rho_{i\,0} = \lim_{x\to 0} \rho_{i\,x}$, $v_{i\,0} = \lim_{x\to 0} v_{i\,x}$. If we assume that $\rho_{i\,0}>0$ for both $i$, then $v_{1\,0} =v_{2\,0}$, or, if one weight approaches zero, such as   $\rho_{1\,0}=0$ then we have $\vert v_{1\,0}-v_{2\,0}\vert \le \eta$ and $\eta$ is then a bound on the distance between the two abscissas.
\item ${\bf M}_0 = (M_{0\,0}, M_{1\,0}, M_{1\,0}^2/M_{0\,0},M_{1\,0}^3/M_{0\,0}^2)^t$.
\end{itemize}
\end{proposition}

\begin{proof}
It is first clear that $\lim_{x\to 0} q_x =q_0= 0$ since 
$\vert q_x \vert \le \eta M_{0\,x}\,e_x$ and $e_0=0$. Then, easy calculations give 
\begin{equation*}
v_1-v_2 = \sqrt{\frac{q^2}{M_0^2 e^2} + \frac{4e}{M_0^2}}\\[-1ex]
\end{equation*}
so that denoting $l=\lim_{x \to 0} \vert q_x\vert /(M_{0x}e_x) \geq 0$, we clearly have $ v_{1x}-v_{2x} \to l$ when 
$x \to 0$ and $\eta$ represents an upper bound for $ v_{1\,0} - v_{2\,0}$. Let us now distinguish between 
the cases $l > 0$ and $l=0$. We first note the following expression for $q/(M_0\, e)$~:
$$
\frac{q}{M_0\,e} = (\omega_1-\omega_2)(v_2-v_1), \quad \omega_i=\rho_i/M_0, \quad \omega_1+\omega_2=1.
$$
If $l > 0$, one can write  
\vspace{-0.2cm}
$$
\vert \rho_{1\,x}- \rho_{2\,x} \vert = \vert \frac{q_x}{M_{0\,x}\,e_x} \vert \times \frac{1}{\vert v_{1\,x}-v_{2\,x} \vert} \times M_{0\,x}, 
$$
and this quantity clearly tends to $M_{0\,0}$ as $x$ goes to zero. Which means that one weight approaches zero, 
$\rho_{1\,0}=0$ or $\rho_{2\,0}=0$, and the other 
one $M_{0\,0}$ (with $\rho_{1\,x}+\rho_{2\,x}=M_{0\,x}$). \\
If $l = 0$, it is clear by the following formula for $v_1$ (see the proof above) 
$$
v_1 = \frac{M_1}{M_0} + \frac{q +\sqrt{q^2+4e^3} }{2M_0e}
$$
that $v_{1\,0} =v_{2\,0} = M_{1\,0} / M_{0\,0}$. Using now the definition of $\rho_1$ and $\rho_2$ (see again 
the proof above), one easily get 
$$
\omega_2 = \frac{\rho_2}{M_0} = \frac{\sqrt{1+4e^3/q^2} + sign(q)}{2\sqrt{1+4e^3/q^2}}, \quad
\omega_1 = \frac{\rho_1}{M_0} = \frac{\sqrt{1+4e^3/q^2} - sign(q)}{2\sqrt{1+4e^3/q^2}},
$$
so that both weights have limits depending on the limit of $q/e^{3/2}$. \\
Clearly, ${\bf M}_0 = (M_{0\,0}, M_{1\,0}, M_{1\,0}^2/M_{0\,0},M_{1\,0}^3/M_{0\,0}^2)^t$, 
which completes the proof.
\end{proof}
%
%
%
%

A very important consequence of the previous proposition is the fact that along smooth paths inside the proposed cone which reach the point $(0,0)$ in the $(e,q)$ plane, the flux introduced in equation \ref{eq:modele_bipic_short} is regular up to the frontier of the moment space, even if the mapping of ${\bf M}$ onto ${\bf U}$  is not\footnote{The  mapping of ${\bf M}$ onto ${\bf U}$ will never be smooth at point $(0,0)$ in the $(e,q)$ plane since the limit of ${\bf U}$ will depend on the limit of $q/(M_0 e)$, whereas the limit value of ${\bf M}$ in the proposed cone is always fixed.}

\begin{proposition}\label{regular_flux}
For any regular path in the moment space satisfying the assumptions of the previous proposition, that is living in the proper cone in the $(e,q)$ plane and reaching smoothly the point $(0,0)$, the flux ${\bf F}({\bf M})$ is continuous  up to the frontier of the moment space and ${\mathcal C}^1$ at $(0,0)$ in any direction inside the proposed cone.
\end{proposition}

\begin{proof}
The proof is rather straightforward when one has noticed the two equations, the first of which is the expression (\ref{eq:M4_sigma01M}) of $\overline{M_4} = -\sigma_0\,M_2-\sigma_1\,M_3$ as a function of $M_2$, $M_3$,  $\sigma_0$ and $\sigma_1$, and the second is the expression of $\sigma_0$ and $\sigma_1$~:
$$
\left(
\begin{array}{c}
\sigma_0 \vphantom{\displaystyle \frac{q}{M_0\,e}}\\ 
\sigma_1  \vphantom{\Biggl(}\\
\end{array}
\right)
=
\left(
\begin{array}{c}
\displaystyle \frac{q}{M_0\,e}\frac{M_1}{M_0} + \left(\frac{M_1}{M_0}\right)^2- \frac{e}{M_0^2} \\ 
\displaystyle -\frac{q}{M_0\,e} - 2\,\frac{M_1}{M_0}  \vphantom{\Biggl(}
\end{array}
\right).
$$
Finally, following the same lines for the evaluation of the Jacobian matrix of the flux (see matrix $J$ in section 2.2), it becomes clear that the flux is continuous and continuously differentiable in any direction inside the proposed cone. Besides, it can be easily seen that the expression of the flux as a function of $(M_0,M_1,e,q)^t$ becomes
\begin{equation*}
\begin{split}
\frac{\overline{M_4}}{M_0}= -\left( \frac{q}{M_0\,e}\,\frac{M_1}{M_0} + \left(\frac{M_1}{M_0}\right)^2 -  \frac{e}{M_0^2}\right)\left( \frac{e}{M_0^2} + \left(\frac{M_1}{M_0}\right)^2\right) +\\ 
\left( \frac{q}{M_0\,e} + 2\,\frac{M_1}{M_0} \right) \left( \frac{q}{M_0^3} +  \left(\frac{M_1}{M_0}\right)^3 + 3 \frac{M_1}{M_0}\, \frac{e}{M_0^2}\right).
\end{split}
\end{equation*}
%
%
$\overline{M_4}/M_0$ tends to $(M_{10}/M_{00})^4$ when $e$ goes to $0^+$ which concludes the proof.
\end{proof}

\begin{rem}
Let us emphasize that in the various configurations we have proposed when the convergence toward the frontier of the moment space does not lie inside the cone $\vert q_x/(M_{0\,x}\,e_x)\vert \le \eta$ in the $(e,q)$ plane, the flux can dramatically loose regularity. It can either have a limit without being differentiable or even not have a limit at all. The impact of the previous proposition thus becomes clear and sets the proper framework for solutions which will reach the frontier of the moment space.
\end{rem}

\begin{rem}
In the case $M_0 > 0$, at the frontier of the moment space within the previous proposed framework, we have 
$e=0$, $q=0$; the model is made of the two 
unknowns $M_0$ and $M_1$ and then degenerates to the usual pressureless gas dynamics which is weakly 
hyperbolic with a single eigenvalue $v=M_1/M_0$. It should be noticed that for smooth solutions, the last two equations of system \ref{eq:modele_bipic} on $M_2 = M_1^2/M_0$ and $M_3= M_1^3/M_0^2$ are still satisfied with  $\overline{M_4} = M_1^4/M_0^3$ coherent with the previous limit obtained for the flux. As a consequence, we can notice, that at least for smooth solutions, the system of partial differential equations (\ref{eq:modele_bipic}) can describe the dynamics inside and at the frontier of the moment space.
\end{rem}

\section{Entropy conditions}

In this section, we will work in the interior of the moment space and we exhibit natural entropy inequalities for the following
small viscosity system associated with (\ref{eq:modele_bipic})~:
\begin{equation} \label{eq:modele_bipic_visqueux}
\left\{
\begin{array}{l}
\partial_t M_0 + \partial_x M_1 = \varepsilon \partial_{xx} M_0, \\
\partial_t M_1 + \partial_x M_2 = \varepsilon \partial_{xx} M_1, \\
\partial_t M_2 + \partial_x M_3 = \varepsilon \partial_{xx} M_2, \\
\partial_t M_3 + \partial_x M_4 = \varepsilon \partial_{xx} M_3, 
\end{array}
\right. 
\end{equation}
which gives in condensed form
\begin{equation} \label{eq:modele_bipic_visqueux_short}
\partial_t {\bf M} + \partial_x {\bf F}({\bf M}) = \varepsilon \partial_{xx} {\bf M}.
\end{equation}
Throughout this section, we will consider smooth solutions only. We thus have
$$
\partial_t {\bf M} + {\bf J} \partial_x {\bf M} = \varepsilon \partial_{xx} {\bf M} \quad
\mbox{with} \quad {\bf J} = \nabla_{{\bf M}} {\bf F}.
$$
Setting ${\bf A} = \nabla_{{\bf U}} {\bf M}$, we then get
\begin{equation} \label{eq:edp_U}
\partial_t {\bf U} + {\bf A}^{-1} {\bf J} {\bf A} \partial_x {\bf U} = \varepsilon 
{\bf A}^{-1} \partial_{x} ({\bf A} \partial_{x} {\bf U}).
\end{equation}
Our objective is to prove that
\vspace{-0.3cm}
\begin{equation} \label{eq:ineq_entropy_visqueux}
 \partial_t \eta + \partial_x q \leq \varepsilon 
\partial_{xx} \eta,
\end{equation}
for a natural choice of couple $(\eta,q)$ given by
\begin{equation} \label{eq:couple_entropy_flux}
\left\{
\begin{array}{l}
\eta = \rho_1 S(v_1) + \rho_2 S(v_2), \\
q = \rho_1 v_1 S(v_1) + \rho_2 v_2 S(v_2).
\end{array}
\right.
\end{equation}
Here $S$ denotes a convex function from $\mathbb{R}$ to $\mathbb{R}$, and we will especially consider 
the case where $S(v) = v^{2 \alpha}$, $\alpha \geq 0$. Of course, the densities $\rho_1$, $\rho_2$ and 
velocities $v_1$, $v_2$ involved in (\ref{eq:couple_entropy_flux}) are naturally defined by means of the one-to-one and onto function ${\bf U} = {\bf U}({\bf M})$. In the following and with a little abuse in the notations, we will consider 
without distinction $\eta$ and $q$ as functions of ${\bf M}$ or ${\bf U}$.\\
\ \\
We first observe 
$$
\begin{array}{rcl}
 \partial_t \eta + \partial_x q &=& \partial_t \eta({\bf U}) + \partial_x q({\bf U}) \\
 &=& \nabla_{{\bf U}} \eta \partial_t {\bf U} + \nabla_{{\bf U}} q \partial_x {\bf U} \\
 &=& \nabla_{{\bf U}} \eta \{ - {\bf A}^{-1} {\bf J} {\bf A} \partial_x {\bf U} 
+ \varepsilon {\bf A}^{-1} \partial_x ({\bf A} \partial_x {\bf U})\} +
\nabla_{{\bf U}} q \partial_x {\bf U}.
\end{array}
$$
The following two lemmas, the proofs of which are left to the reader, will be useful in order to estimate the entropy dissipation rate ${\bf D}$ defined 
by 
$$
{\bf D} = \nabla_{{\bf U}} \eta \{ - {\bf A}^{-1} {\bf J} {\bf A} \partial_x {\bf U} 
+ \varepsilon {\bf A}^{-1} \partial_x ({\bf A} \partial_x {\bf U})\} +
\nabla_{{\bf U}} q \partial_x {\bf U}.
$$

\begin{lemma}
 The matrices ${\bf J}$ and ${\bf A}$ and ${\bf A}^{-1} {\bf J} {\bf A}$ are given by
$$
{\bf J} = 
\left(
\begin{array}{cccc}
 0 & 1 & 0 & 0 \\
 0 & 0 & 1 & 0 \\
 0 & 0 & 0 & 1 \\
 -v_1^2 v_2^2 & 2 v_1 v_2 (v_1+v_2) & -2v_1v_2-(v_1+v_2)^2 & 2(v_1+v_2) \\
\end{array}
\right),
$$ 
$$
{\bf A} = 
\left(
\begin{array}{cccc}
 1 & 1 & 0 & 0 \\
 0 & 0 & 1 & 1 \\
 -v_1^2 & -v_2^2 & 2v_1 & 2v_2 \\
 -2v_1^3 & -2 v_2^3  & 3v_1^2 & 3v_2^2 \\
\end{array}
\right), \quad {\bf A}^{-1} {\bf J} {\bf A}= 
\left(
\begin{array}{cccc}
 0 & 0 & 1 & 0 \\
 0 & 0 & 0 & 1 \\
 -v_1^2 & 0 & 2v_1 & 0 \\
 0 & -v_2^2  & 0 & 2v_2 \\
\end{array}
\right).
$$ 
\end{lemma}

\begin{lemma} \label{lemme:2}
 The gradients $\nabla_{{\bf U}} \eta$ and $\nabla_{{\bf U}} q$ are given by
$$
\nabla_{{\bf U}} \eta = 
\left(
\begin{array}{c}
 S(v_1)-v_1 S^{'}(v_1) \\
 S(v_2)-v_2 S^{'}(v_2) \\
 S^{'}(v_1) \\
 S^{'}(v_2)  \\
\end{array}
\right)^t, \quad 
\nabla_{{\bf U}} q = 
\left(
\begin{array}{c}
 -v_1^2 S^{'}(v_1) \\
 -v_2^2 S^{'}(v_2) \\
 S(v_1)+v_1 S^{'}(v_1) \\
 S(v_1)+v_1 S^{'}(v_1)  \\
\end{array}
\right)^t,
$$ 
and we have
$$
\nabla_{{\bf U}} q = \nabla_{{\bf U}} \eta {\bf A}^{-1} {\bf J} {\bf A}. 
$$ 
\end{lemma}
%

\noindent Before going on, let us make the following two remarks. We first note that thanks to the first 
lemma, (\ref{eq:edp_U})
with $\varepsilon=0$ inside the moment space gives~:
$$
\left\{
\begin{array}{l}
\partial_t \rho_1 + \partial_x \rho_1 v_1 = 0, \\
\partial_t \rho_2 + \partial_x \rho_2 v_2 = 0, \\
\partial_t \rho_1 v_1 - v_1^2 \partial_x \rho_1 + 2 v_1 \partial_{x} \rho_1 v_1 = 0, \\
\partial_t \rho_2 v_2 - v_2^2 \partial_x \rho_2 + 2 v_2 \partial_{x} \rho_2 v_2 = 0, \\
\end{array}
\right. 
$$
which is equivalent to 
\vspace{-0.2cm}
\begin{equation}\label{eq:2gsp}
\left\{
\begin{array}{l} 
\partial_t \rho_1 + \partial_x \rho_1 v_1 = 0, \\
\partial_t \rho_1 v_1 + \partial_x \rho_1 v_1^2  = 0, \\
\partial_t \rho_2 + \partial_x \rho_2 v_2 = 0, \\
\partial_t \rho_2 v_2 + \partial_x \rho_2 v_2^2 = 0. \\
\end{array}
\right. 
\end{equation}
Besides, at the frontier of the moment space, we obtain the pressureless gas dynamics on a single quadrature node.
We then observe that for smooth solutions, thanks to the remark at the end of the previous section, the system (\ref{eq:modele_bipic_short}) is nothing but either two 
decoupled or one single pressureless gas dynamics system of equations. We then observe that still with $\varepsilon=0$, 
${\bf D} =  0$, by lemma \ref{lemme:2}, in both cases~:
\begin{equation}\label{eq:entropieeqzero}
\partial_t \eta + \partial_x q = 0,
\end{equation}
that is for smooth solution both inside and at the frontier of the moment space.

Let us go back to the case $\varepsilon > 0$. We thus have the following equality,
$$
{\bf D} = \varepsilon \nabla_{{\bf U}} \eta {\bf A}^{-1} \partial_x ({\bf A} \partial_x {\bf U}),
$$
from which it is natural to isolate $\varepsilon \partial_{xx} \eta$~:
$$
\begin{array}{rcl}
{\bf D} &=& \varepsilon \nabla_{{\bf U}} \eta {\bf A}^{-1} \partial_x ({\bf A} \partial_x {\bf U}) \\
        &=& \varepsilon \nabla_{{\bf U}} \eta \partial_{xx} {\bf U} + 
            \varepsilon \nabla_{{\bf U}} \eta {\bf A}^{-1} \partial_x {\bf A} \partial_x {\bf U} \\
        &=& \varepsilon \partial_{x} (\nabla_{{\bf U}} \eta \partial_{x} {\bf U}) 
            - \varepsilon \partial_{x} (\nabla_{{\bf U}} \eta) \partial_{x} {\bf U} + 
            \varepsilon \nabla_{{\bf U}} \eta {\bf A}^{-1} \partial_x {\bf A} \partial_x {\bf U} \\
        &=& \varepsilon \partial_{xx} \eta + 
            \varepsilon \big( \nabla_{{\bf U}} \eta {\bf A}^{-1} \partial_x {\bf A} \partial_x {\bf U} 
             - \partial_{x} (\nabla_{{\bf U}} \eta) \partial_{x} {\bf U} \big). 
             \\
\end{array}
$$
By lemma \ref{lemme:2} giving $\nabla_{{\bf U}} \eta$, we easily get 
$$
\partial_{x} (\nabla_{{\bf U}} \eta) \partial_{x} {\bf U} = 
\rho_1 (\partial_x v_1)^2 S^{''}(v_1) + \rho_2 (\partial_x v_2)^2 S^{''}(v_2).
$$
It is now a matter to calculate 
$\nabla_{{\bf U}} \eta {\bf A}^{-1} \partial_x {\bf A} \partial_x {\bf U} = 
{\bf A}^{-t} (\nabla_{{\bf U}} \eta)^t \partial_x {\bf A} \partial_x {\bf U}$. We first observe that 
$$
\partial_x {\bf A} \partial_x {\bf U} = 
\left(
\begin{array}{c}
0 \\
0 \\
2 \rho_1 (\partial_x v_1)^2 + 2 \rho_2 (\partial_x v_2)^2 \\
6 \rho_1 v_1 (\partial_x v_1)^2 + 6 \rho_2 v_2 (\partial_x v_2)^2
\end{array}
\right)
$$
so that only the last two components of ${\bf A}^{-t} (\nabla_{{\bf U}} \eta)^t$ are actually needed.
Finally, easy calculations not reported here lead to
$$
{\bf A}^{-t} (\nabla_{{\bf U}} \eta)^t \partial_x {\bf A} \partial_x {\bf U} = 
- \frac{1}{(v_1-v_2)^4} ( 2 \rho_1 (\partial_x v_1)^2 X_1 + 2 \rho_2 (\partial_x v_2)^2 X_2),
$$
where we have set
$$
\left\{
\begin{array}{l}
 X_1=(v_1-v_2)^2 \Big( 3\big(S(v_1)-S(v_2)\big) - 2(v_1-v_2)S^{'}(v_1) - (v_1-v_2)S^{'}(v_2)  \Big), \\
 X_2=-(v_1-v_2)^2 \Big( 3\big(S(v_1)-S(v_2)\big) - (v_1-v_2)S^{'}(v_1) - 2(v_1-v_2)S^{'}(v_2)  \Big).
\end{array}
\right.
$$
The entropy inequality (\ref{eq:ineq_entropy_visqueux}) is then valid if and only if 
$$
\nabla_{{\bf U}} \eta {\bf A}^{-1} \partial_x {\bf A} \partial_x {\bf U} 
             - \partial_{x} (\nabla_{{\bf U}} \eta) \partial_{x} {\bf U} \leq 0,
$$
that is, setting $S_i=S(v_i)$, $S^{'}_i=S^{'}(v_i)$ and $S^{''}_i=S^{''}(v_i)$, $i=1,2$, 
\begin{equation}
\begin{array}{rcl}
&(v_1-v_2)^2  \rho_1 (\partial_x v_1)^2 \Big( 6\big(S_1-S_2\big) - 4(v_1-v_2)S^{'}_1 - 2(v_1-v_2)S^{'}_2 +
(v_1-v_2)^2 S^{''}_1 \Big)&\\
&+&\\
&(v_1-v_2)^2  \rho_2 (\partial_x v_2)^2 \Big( -6\big(S_1-S_2\big) + 2(v_1-v_2)S^{'}_1 + 4(v_1-v_2)S^{'}_2 +
(v_1-v_2)^2 S^{''}_2 \Big) &\\
&\geq 0 &.
\end{array}
\end{equation}
A sufficient condition is given by
\begin{equation} \label{eq:ineg_cond_entropy}
\left\{
\begin{array}{c}
6\big(S_1-S_2\big) - 4(v_1-v_2)S^{'}_1 - 2(v_1-v_2)S^{'}_2 +
(v_1-v_2)^2 S^{''}_1 \geq 0,\\
-6\big(S_1-S_2\big) + 2(v_1-v_2)S^{'}_1 + 4(v_1-v_2)S^{'}_2 +
(v_1-v_2)^2 S^{''}_2 \geq 0.
\end{array}
\right.
\end{equation}
Let us focus for instance on the first inequality (the second one is treated in a similar way), 
and let us consider the left-hand side as a function 
of $v_2$, for any given $v_1$~:
$$
{\mathcal F}_1(v_2) = 6(S_1-S_2) - 4(v_1-v_2)S^{'}_1 - 2(v_1-v_2)S^{'}_2 +
(v_1-v_2)^2 S^{''}_1.
$$ 
Differentiation yields
$$
\begin{array}{l}
{\mathcal F}_1^{'}(v_2) = 4(S^{'}_1-S^{'}_2) - 2(v_1-v_2)(S^{''}_1 + S^{''}_2), \\
{\mathcal F}_1^{''}(v_2) = 2(S^{''}_1-S^{''}_2) - 2(v_1-v_2)S^{'''}_2, \\
{\mathcal F}_1^{'''}(v_2) = 2(v_2-v_1)S^{''''}_2.
\end{array}
$$ 
It is then clear that 
${\mathcal F}_1(v_1)={\mathcal F}_1^{'}(v_1)={\mathcal F}_1^{''}(v_1)={\mathcal F}_1^{'''}(v_1)=0$. Then, provided that 
$S^{''''}(v) \geq 0$, $\forall v$, we easily get by a chain argument based on the sign of the derivative and 
the monotonicity property that ${\mathcal F}_1(v_2) \geq 0$, $\forall v_1,v_2$. We have thus proved the 
following proposition~: \\
\begin{proposition}
 Smooth solutions of (\ref{eq:modele_bipic_visqueux_short}) satisfy the entropy inequality 
(\ref{eq:ineq_entropy_visqueux}) for any entropy entropy-flux pair $(\eta,q)$ defined 
by (\ref{eq:couple_entropy_flux}) provided that $v \to S(v)$ is a smooth function from $\mathbb{R}$ to $\mathbb{R}$ 
with nonnegative fourth-order derivative. In particular, the natural choice $S(v) = v^{2 \alpha}$ 
with $\alpha \geq 2$ is suitable. \\
\end{proposition}

\begin{rem}
Any third-order polynomial may of course be added to the leading term of $S$, without changing 
the sign of the fourth-order derivative. However, if we focus on (strictly) convex functions $v \to S(v)$ 
in order to get a (strictly) convex entropy $\eta = \eta({\bf U})$, only first-order polynomials 
may be added without changing the convexity property. 
\end{rem}

\begin{rem}
If we consider $S(v)=1,v,v^2,v^3$, it is easily checked that (\ref{eq:ineg_cond_entropy}) holds true 
with two equalities. In agreement with (\ref{eq:modele_bipic_visqueux}), these choices that lead to 
the pairs $(\eta,q) = (M_i,M_{i+1})$, $i=0,...,3$, are admissible. 
\end{rem} 

\begin{rem}
In the case $M_0 > 0$, $e=0$, $q=0$ with the additional conditions associated with the connection between the interior and the frontier of the moment space presented in subsection \ref{border}, we clearly have $M_k=M_1^k/M_0^{k-1}$. In this case, the entropy pairs clearly work also in such a case, for smooth solutions, and admits a smooth behavior at the frontier of the moment space in the cone we have defined previously.
\end{rem}

\section{Kinetic-macroscopic relation for smooth solutions}

For smooth solutions, we established in the previous section that the four-moment model (\ref{eq:modele_bipic}) is equivalent 
to the following two decoupled pressureless gas dynamics model 
\begin{equation} \label{eq:model_twopgd}
\left\{
\begin{array}{l}
\partial_t \rho_1 + \partial_x \rho_1 v_1 = 0, \\
\partial_t \rho_1 v_1 + \partial_x \rho_1 v_1^2  = 0, \\
\partial_t \rho_2 + \partial_x \rho_2 v_2 = 0, \\
\partial_t \rho_2 v_2 + \partial_x \rho_2 v_2^2 = 0, \\
\end{array}
\right. 
\end{equation}
where $\rho_1$, $\rho_2$, $v_1$ and $v_2$ are defined by the nonlinear system (\ref{eq:quadrature}). The aim of this section 
is to prove a rigorous equivalence result, still for smooth solutions, between this macroscopic model and the free transport 
kinetic formulation (\ref{eq:cinetique}) when the velocity distribution is given by a set of two Dirac delta functions. 
This result is nothing but a generalization of the one given in \cite{Bouchut94} for the usual pressureless gas dynamics model. \\

\begin{proposition}
Let $T>0$ and $\rho_i(t,x)$, $v_i(t,x)$ in $\mathcal{C}^1(]0,T[\times \mathbb{R})$ for $i=1,2$. Let us define 
\vspace{-0.5cm}
\begin{equation*}
f(t,x,v) = \sum_{i=1}^2 \rho_i(t,x) \delta (v - v_i(t,x)).
\end{equation*}
Then, $\rho_i$ and $v_i$ solve (\ref{eq:modele_bipic}), or equivalently (\ref{eq:model_twopgd}), in $]0,T[\times \mathbb{R}$ 
if and only if 
\begin{equation} \label{eq:equ_cine_distrib_1}
\partial_t f + v \partial_x f = 0, \quad \mbox{in} \quad ]0,T[ \times \mathbb{R} \times \mathbb{R}
\end{equation}
in the distributional sense, i.e. if and only if $\forall \phi \,\, \in \,\, \mathcal{C}^{\infty}_c(]0,T[\times \mathbb{R})$
and $\chi \,\, \in \,\, \mathcal{C}^{\infty}_c(\mathbb{R})$ 
\begin{equation} \label{eq:equ_cine_distrib_2}
\int_0^T \int_{\mathbb{R}} \sum_{i=1}^2 \rho_i(t,x) 
\Big( \partial_t \phi(t,x) + v_i(t,x) \partial_x \phi(t,x) \Big) 
\chi\big(v_i(x,t)\big)=0. \\
\end{equation}
\end{proposition}

\begin{proof}
Let us first assume that (\ref{eq:equ_cine_distrib_2}) holds true. Since the velocity functions $v_i$ are in particular locally 
bounded, one can successively choose $\chi \,\, \in \,\, \mathcal{C}^{\infty}_c(\mathbb{R})$ such that 
$\chi(v)=v^k$, $k=0,...,3$ for all $v=v_i(t,x)$ and then get  
$$
\int_0^T \int_{\mathbb{R}} \sum_{i=1}^2 \rho_i(t,x) 
\Big( \partial_t \phi(t,x) + v_i(t,x) \partial_x \phi(t,x) \Big) 
\big(v_i(t,x)\big)^k=0
$$
for all $\phi \,\, \in \,\, \mathcal{C}^{\infty}_c(]0,T[\times \mathbb{R})$. Invoking the closure relation 
(\ref{eq:quadrature}), this gives the four-moment model
(\ref{eq:modele_bipic}), or equivalently (\ref{eq:model_twopgd}), as $\rho_i(t,x)$ and $v_i(t,x)$ are smooth functions. \\
Conversely, let us assume that the partial differential equations of (\ref{eq:model_twopgd}) are satisfied. Using the mass
conservation equations, it is then usual to show that for $i=1,2$
$$
\rho_i (\partial_t v_i + v_i \partial_x v_i) = 0,
$$ 
and then multiplying by $\chi'$ for any smooth function $\chi$, 
$$
\partial_t \rho_i \chi(v_i) + \partial_x \rho_i \chi(v_i) v_i = 0.
$$ 
Summing over $i=1,2$ and integrating past a test function $\phi \,\, \in \,\, \mathcal{C}^{\infty}_c(]0,T[\times \mathbb{R})$ 
gives the expected result (\ref{eq:equ_cine_distrib_2}). This concludes the proof.
\end{proof}

This Proposition allows, as a corollary, to introduce a particular type of solution which will be denoted piecewise {\it free boundary} $\mathcal C^1$ solutions. Such solutions correspond to a discontinuous connection from the interior of the moment space to the frontier through a contact discontinuity for which the Rankine Hugoniot solutions are trivially satisfied as well as the entropy conservation equation (\ref{eq:entropieeqzero}). 

\begin{corollary}\label{freebound}
We consider the following distribution at the kinetic level $f(t,x,v) = \sum_{i=1}^2 \rho_i(t,x) \delta (v - v_i(t,x))$, where $\rho_1(t,x)>0$ and $v_1(t,x)$ are taken as constants (or sufficiently smooth in some time interval), whereas $\rho_2(t,x)$ is zero except in a compact connected subset $K_0$ at time $t=0$ of $\mathbb{R}$, where $\rho_2(0,x)>0$ and $v_2(0,x)$ are two constants (or sufficiently smooth in some time interval) such that $v_2\neq v_1$. The resulting solution at the moment level exhibits two discontinuities at the frontier of the compact set $K_t$ which is the translation of set $K_0$ at velocity $v_2$; however the system \ref{eq:model_twopgd} as well as  the entropy conservation equation (\ref{eq:entropieeqzero}) are satisfied in the weak sense, that is the equations are satisfied in the usual sense where the solution is smooth and  Rankine-Hugoniot conditions are satified at discontinuity points.
\end{corollary}

\begin{proof}
Clearly, the moment solution will satisfy the system of conservation equations everywhere except at the frontier of the $K_t$ set.
The Rankine-Hugoniot jump conditions are trivially satisfied at the discontinuities where the mass flux associated to the first abscissa is $\rho_1(v_1-v_2)$ in the referential of the discontinuity and leads to zero jump conditions for the part of the flux associated to the first abscissa by continuity, whereas the mass flux associated to the second abscissa is zero, as usual in contact discontinuities, which also allows to conclude. The same path allows to conclude that for any entropy-flux pair, the conservation equation is satisfied in the weak sense.
\end{proof}

Let us underline the fact that in the region where the second weight is zero outside the compact set $K_t$, we have used so far the convention that in such a region where a single quadrature node is to be found, the two weights are equal and the two abscissas are equal. The results proposed in the previous corollary are of course independent of such a choice since the point at the frontier of the moment space is isolated. Besides, such a corollary can be extended to as many quadrature nodes as needed as long as the number of nodes allows to describe the dynamics at the kinetic level at any point and time. Finally, the collision of two particle packets presented in subsection \ref{col_two} satisfies the assumptions of Corollary \ref{freebound} and will be an entropic solution.

\section{Riemann problems and entropic measure solutions}

In this section, we focus on the Riemann problem, which is associated with the inital condition for two constant states ${\bf M}_L$ and ${\bf M}_R$ in $\Omega$.
\begin{equation} \label{eq:donnee_initiale}
 {\bf M}(x,0)=
\left\{
\begin{array}{lll}
 {\bf M}_L & \mbox{if} & x<0,\\
 {\bf M}_R & \mbox{if} & x>0,\\
\end{array}
\right.
\end{equation}

\ \\
The solution of 
(\ref{eq:modele_bipic_short})-(\ref{eq:donnee_initiale}) is sought in the form
\begin{equation} \label{eq:solution_mesure}
 {\bf M}(x,t)=
\left\{
\begin{array}{lll}
 {\bf M}_L & \mbox{if} & x<\sigma_L t,\\
 {\bf M}_L^{\delta}(t) \delta(x-\sigma_L t) & \mbox{if} & x=\sigma_L t,\\
 {\bf M}_{\star} & \mbox{if} & \sigma_L t < x < \sigma_R t,\\
 {\bf M}_R^{\delta}(t) \delta(x-\sigma_R t) & \mbox{if} & x=\sigma_R t,\\
 {\bf M}_R & \mbox{if} & x>\sigma_R t,\\
\end{array}
\right.
\end{equation}
which corresponds to the juxtaposition of two Dirac delta functions with mass ${\bf M}_{\beta}^{\delta}$ and position
$x-\sigma_{\beta} t$, $\beta=L,R$, and separated by a 
constant state ${\bf M}_{\star}$ in $\Omega$. Here, $\sigma_{\beta}$ denotes a real number, we have $m_{\beta}(t) \geq 0$, $m_{\beta}(0) = 0$ and $\beta=L,R$, and 
${\bf M}_{\beta}^{\delta}(t)$ is defined by
\begin{equation}
{\bf M}_{\beta}^{\delta}(t) =
\left(
\begin{array}{l}
m_{\beta}(t) \\
m_{\beta}(t) \sigma_{\beta} \\
m_{\beta}(t) \sigma_{\beta}^2\\
m_{\beta}(t) \sigma_{\beta}^3
\end{array}
\right)
\end{equation}

We introduce the following natural definitions of (entropy) measure solutions. \\
\begin{definition}
Let $\eta = \rho_1 S(v_1) + \rho_2 S(v_2)$ and 
$q = \rho_1 v_1 S(v_1) + \rho_2 v_2 S(v_2)$ with $S(v) = v^{2\alpha}$, $\alpha \in \mathbb{N}$. 
We say that (\ref{eq:solution_mesure}) is a measure solution of (\ref{eq:modele_bipic_short})-(\ref{eq:donnee_initiale}) 
if and only if $\partial_t \eta + \partial_x q = 0$ in ${\mathcal{D}}^{'}(]0,\infty[ \times \mathbb{R})$ for $S(v) = v^{2\alpha}$ 
with $\alpha=0,\frac{1}{2},1,\frac{3}{2}$, that is for $(\eta,q)=(M_i,M_{i+1})$, $i=0,1,2,3$. 
We say that (\ref{eq:solution_mesure}) is an entropic measure solution of (\ref{eq:modele_bipic_short})-(\ref{eq:donnee_initiale}) 
if and only if $\partial_t \eta + \partial_x q \le 0$ in ${\mathcal{D}}^{'}(]0,\infty[ \times \mathbb{R})$ for $S(v) = v^{2\alpha}$ 
with $\alpha=0,\frac{1}{2},1,\frac{3}{2}$ and $\alpha \geq 2$. \\
\end{definition}

\noindent We can prove the following equivalence result. \\

\begin{theorem}
The solution given by (\ref{eq:solution_mesure}) is a measure solution of (\ref{eq:modele_bipic_short})-(\ref{eq:donnee_initiale}) 
iff
\begin{equation} \label{eq:rh_generalisee}
\left\{
\begin{array}{l}
 \sigma_L ({\bf M}_L - {\bf M}_{\star}) t - 
\big({\bf F}({\bf M}_L) - {\bf F}({\bf M}_{\star})\big) t + {\bf M}_L^{\delta}(t) = 0 \\
\ \\
 \sigma_R ({\bf M}_{\star} - {\bf M}_R) t - 
\big({\bf F}({\bf M}_{\star}) - {\bf F}({\bf M}_R)\big) t + {\bf M}_R^{\delta}(t) = 0.
\end{array}
\right.
\end{equation}
The solution given by (\ref{eq:solution_mesure}) is a entropic measure solution of (\ref{eq:modele_bipic_short})-(\ref{eq:donnee_initiale}) 
if and only if in addition, for all 
$S(v) = v^{2\alpha}$ with $\alpha \geq 2$~:
\begin{equation} \label{eq:entropy_generalisee}
\left\{
\begin{array}{l}
 \sigma_L \big(\eta({\bf M}_L) - \eta({\bf M}_{\star})\big) t - 
\big(q({\bf M}_L) - q({\bf M}_{\star})\big) t + \eta({\bf M}_L^{\delta}(t)) \leq 0 \\
\ \\
 \sigma_R \big(\eta({\bf M}_{\star}) - \eta({\bf M}_R)\big) t - 
\big(q({\bf M}_{\star}) - q({\bf M}_R)\big) t + \eta({\bf M}_R^{\delta}(t)) \leq 0,
\end{array}
\right.
\end{equation} 
\end{theorem}

\begin{proof}
We first recall that $\partial_t \eta + \partial_x q = 0$ in ${\mathcal{D}}^{'}(]0,\infty[ \times \mathbb{R})$ means that 
for any smooth function with compact support $\varphi \in C^{\infty}_c(]0,\infty[ \times \mathbb{R})$, we have
\vspace{-0.3cm}
\begin{equation*}
\begin{array}{rcl}
&<\partial_t \eta + \partial_x q,\varphi>=-<\eta,\partial_t \varphi>-<q,\partial_x \varphi> = &\\[2ex]
&- \int_0^{\infty} \int_{-\infty}^{+\infty} \eta(x,t) \partial_t \varphi(x,t) dxdt 
- \int_0^{\infty} \int_{-\infty}^{+\infty} q(x,t) \partial_x \varphi(x,t) dxdt = 0,&
\end{array}
\end{equation*}
where by definition (\ref{eq:solution_mesure})
$$
 \eta(x,t)=
\left\{
\begin{array}{lll}
 \eta({\bf M}_L), & x<\sigma_L t,\\
 \eta({\bf M}_L^{\delta}(t)) \delta(x-\sigma_L t), & x=\sigma_L t,\\
 \eta({\bf M}_{\star}),& \!\!\!\!\!\!\!\!\!\!\!\!\!\!\!\!\!\!\!\!\sigma_L < \frac{x}{t} < \sigma_R,\\
 \eta({\bf M}_R^{\delta}(t)) \delta(x-\sigma_R t),   & x=\sigma_R t,\\
 \eta({\bf M}_R), & x>\sigma_R t,\\
\end{array}
\right.
\quad
q(x,t)=
\left\{
\begin{array}{lll}
 q({\bf M}_L),& x<\sigma_L t,\\
 q({\bf M}_L^{\delta}(t)) \delta(x-\sigma_L t), & x=\sigma_L t,\\
 q({\bf M}_{\star}),  & \!\!\!\!\!\!\!\!\!\!\!\!\!\!\!\!\!\!\!\! \sigma_L < \frac{x}{t}  < \sigma_R,\\
 q({\bf M}_R^{\delta}(t)) \delta(x-\sigma_R t), & x=\sigma_R t,\\
 q({\bf M}_R),  & x>\sigma_R t.\\
\end{array}
\right.
$$
Here we note that $q({\bf M}_{\beta}^{\delta}(t)) = \sigma_{\beta} \eta({\bf M}_{\beta}^{\delta}(t))$ with 
$\eta({\bf M}_{\beta}^{\delta}(t)) = m_{\beta}(t) S(\sigma_{\beta})$, $\beta=L,R$.
For all $\varphi \in C^{\infty}_c(]0,\infty[ \times \mathbb{R})$ we thus have
$$
\begin{array}{l}
\displaystyle <\partial_t \eta + \partial_x q,\varphi> = \\ 
\ \\
\displaystyle
- \int_0^{\infty} dt \int_{-\infty}^{\sigma_L t} \eta({\bf M}_L) \partial_t \varphi(x,t) dx 
- \int_0^{\infty} \eta({\bf M}_L^{\delta}(t)) \partial_t \varphi(\sigma_L t,t) dt \\
\displaystyle 
- \int_0^{\infty} dt \int_{\sigma_L t}^{\sigma_R t} \eta({\bf M}_{\star}) \partial_t \varphi(x,t) dx 
- \int_0^{\infty} \eta({\bf M}_R^{\delta}(t)) \partial_t \varphi(\sigma_R t,t) dt \\
\displaystyle 
- \int_0^{\infty} dt \int_{\sigma_R t}^{\infty} \eta({\bf M}_R) \partial_t \varphi(x,t) dx
- \int_0^{\infty} dt \int_{-\infty}^{\sigma_L t} q({\bf M}_L) \partial_x \varphi(x,t) dx \\
\displaystyle 
- \int_0^{\infty} q({\bf M}_L^{\delta}(t)) \partial_x \varphi(\sigma_L t,t) dt 
- \int_0^{\infty} dt \int_{\sigma_L t}^{\sigma_R t} q({\bf M}_{\star}) \partial_x \varphi(x,t) dx \\ 
\displaystyle 
- \int_0^{\infty} q({\bf M}_R^{\delta}(t)) \partial_x\varphi(\sigma_R t,t) dt
- \int_0^{\infty} dt \int_{\sigma_R t}^{\infty} q({\bf M}_R) \partial_x \varphi(x,t) dx \\
\end{array}
$$
that is, using in particular $q({\bf M}_{\beta}^{\delta}(t)) = \sigma_{\beta} \eta({\bf M}_{\beta}^{\delta}(t))$, 
$$
\begin{array}{l}
\displaystyle <\partial_t \eta + \partial_x q,\varphi> = \\ 
\ \\
\displaystyle
= - \int_0^{\infty} dt \eta({\bf M}_L) \big( \frac{d}{dt} \int_{-\infty}^{\sigma_L t} \varphi(x,t) dx - \sigma_L \varphi(t,\sigma_Lt) \big) 
- \int_0^{\infty} q({\bf M}_L) \varphi(\sigma_L t,t) dt \\
\displaystyle 
- \int_0^{\infty} dt \eta({\bf M}_{\star}) \big( \frac{d}{dt} \int_{\sigma_L t}^{\sigma_R t} 
\varphi(x,t) dx - \sigma_R \varphi(t,\sigma_R t) + \sigma_L \varphi(t,\sigma_L t) \big) \\
\displaystyle 
- \int_0^{\infty} q({\bf M}_{\star}) \big( \varphi(\sigma_R t,t) - \varphi(\sigma_L t,t) \big) dt \\
\displaystyle 
- \int_0^{\infty} dt \eta({\bf M}_R) \big( \frac{d}{dt} \int_{\sigma_R t}^{\infty} 
\varphi(x,t) dx + \sigma_R \varphi(t,\sigma_R t) \big) 
+ \int_0^{\infty} q({\bf M}_R) \varphi(\sigma_R t,t) dt \\
\displaystyle 
- \int_0^{\infty} \eta({\bf M}_L^{\delta}(t)) \frac{d}{dt} [\varphi(\sigma_L t,t)] dt 
- \int_0^{\infty} \eta({\bf M}_R^{\delta}(t)) \frac{d}{dt} [\varphi(\sigma_R t,t)] dt \\
\end{array}
$$
$$
\begin{array}{l}
\displaystyle 
=  \int_0^{\infty} \varphi(\sigma_L t,t) \big( \sigma_L \eta({\bf M}_L) - q({\bf M}_L) \big) dt 
+ \int_0^{\infty} \varphi(\sigma_R t,t) \big( \sigma_R \eta({\bf M}_{\star}) - q({\bf M}_{\star}) \big) dt \\
\displaystyle 
- \int_0^{\infty} \varphi(\sigma_L t,t) \big( \sigma_L \eta({\bf M}_{\star}) - q({\bf M}_{\star}) \big) dt
- \int_0^{\infty} \varphi(\sigma_R t,t) \big( \sigma_R \eta({\bf M}_R) - q({\bf M}_R) \big) dt \\
\displaystyle 
- \int_0^{\infty} \eta({\bf M}_L^{\delta}(t)) \frac{d}{dt} [\varphi(\sigma_L t,t)] dt 
- \int_0^{\infty} \eta({\bf M}_R^{\delta}(t)) \frac{d}{dt} [\varphi(\sigma_R t,t)] dt  \\
\ \\
\displaystyle 
= \int_0^{\infty} \varphi(\sigma_L t,t) \big( \sigma_L (\eta({\bf M}_L) - \eta({\bf M}_{\star})) - 
(q({\bf M}_L) - q({\bf M}_{\star})) \big) dt \\
\displaystyle 
+ \int_0^{\infty} \varphi(\sigma_R t,t) \big( \sigma_R (\eta({\bf M}_{\star}) - \eta({\bf M}_{R})) - 
(q({\bf M}_{\star}) - q({\bf M}_{R})) \big) dt \\
\displaystyle 
- \int_0^{\infty} \eta({\bf M}_L^{\delta}(t)) \frac{d}{dt} [\varphi(\sigma_L t,t)] dt 
- \int_0^{\infty} \eta({\bf M}_R^{\delta}(t)) \frac{d}{dt} [\varphi(\sigma_R t,t)] dt
\\
\end{array}
$$
$$
\begin{array}{l}
\displaystyle 
= \int_0^{\infty} \varphi(\sigma_L t,t) \frac{d}{dt} 
\big( \sigma_L (\eta({\bf M}_L) - \eta({\bf M}_{\star})) t - 
(q({\bf M}_L) - q({\bf M}_{\star})) t + \eta({\bf M}_L^{\delta}(t)) \big) dt \\
\displaystyle 
+ \int_0^{\infty} \varphi(\sigma_R t,t) \frac{d}{dt} 
\big( \sigma_R (\eta({\bf M}_{\star}) - \eta({\bf M}_R)) t - 
(q({\bf M}_{\star}) - q({\bf M}_{R})) t + \eta({\bf M}_R^{\delta}(t)) \big) dt.
\end{array}
$$
Then, since $m_{\beta}(0) = 0$ and then $\eta({\bf M}_{\beta}^{\delta}(0))=0$, $\beta=L,R$, it is clear that (\ref{eq:solution_mesure}) is a measure solution of (\ref{eq:modele_bipic_short})-(\ref{eq:donnee_initiale}) if and only if 
(\ref{eq:rh_generalisee}) is valid for all $t \geq 0$, and an entropic measure solution if and only if in 
addition (\ref{eq:entropy_generalisee}) holds true
for all $t \geq 0$ and all 
$\eta = \rho_1 S(v_1) + \rho_2 S(v_2)$ and 
$q = \rho_1 v_1 S(v_1) + \rho_2 v_2 S(v_2)$ with $S(v) = v^{2\alpha}$, $\alpha \geq 2$. 
\end{proof}

\begin{rem}
Let us recall that $\eta({\bf M}_{\beta}^{\delta}(t)) = m_{\beta}(t) S(\sigma_{\beta})$. 
Since (\ref{eq:entropy_generalisee}) is made of equalities when $S(v) = 1$ and $S(v) = v$ (we get 
in these cases the first two components of (\ref{eq:rh_generalisee})), the validity 
of (\ref{eq:entropy_generalisee}) for all $S(v) = v^{2 \alpha}$, $\alpha \geq 2$, is equivalent to the validity 
of 
\begin{equation} \label{eq:entropy_generalisee_bis}
\left\{
\begin{array}{l}
 \sigma_L \big(\eta({\bf M}_L) - \eta({\bf M}_{\star})\big) - 
\big(q({\bf M}_L) - q({\bf M}_{\star})\big) \leq 0 \\
\ \\
 \sigma_R \big(\eta({\bf M}_{\star}) - \eta({\bf M}_R)\big) - 
\big(q({\bf M}_{\star}) - q({\bf M}_R)\big) \leq 0,
\end{array}
\right. 
\end{equation}
for all $S(v) = v^{2 \alpha} - \sigma_L^{2 \alpha} \frac{v-\sigma_R}{\sigma_L-\sigma_R} - \sigma_R^{2 \alpha} \frac{v-\sigma_L}{\sigma_R-\sigma_L}$, $\alpha \geq 2$. 
\end{rem}

\section{Examples of entropic solutions} \label{ees}

In this section, we propose three particular entropic solutions. The first one models the collision of two particles 
packets with free boundary $\mathcal C^1$ smooth solution, that is a solution for which an exact link with the kinetic level is preserved and for which the entropy equation is exactly satisfied. In such a situation the four-moment model does not develop $\delta$-shock Dirac delta functions and is actually 
able to properly represent the crossing of the two packets which correspond to the dynamics at the kinetic level. 
The second one models the collision of four particles packets. In this case and as expected since the number of moments is set to four, 
the entropic solutions involves two $\delta$-shock Dirac delta functions singularities. Whereas the first case corresponds to a connection from the interior of the moment space to the frontier through a contact discontinuity,  or free boundary solution, resulting in an isolated point at the frontier, we consider in a third example a smooth connection to the frontier of the moment space, such that the point at the frontier is an accumulation point of a trajectory inside the moment space.
 
\subsection{Collision of two particles packets}\label{col_two}

We consider a Riemann initial data (\ref{eq:donnee_initiale}) 
where ${\bf M}_L = {\bf M}({\bf U}_L)$ and ${\bf M}_R={\bf M}({\bf U}_R)$ are such that 
\vspace{-0.2cm}
$$
{\bf U}_L = \frac{1}{2}
\left(
\begin{array}{c}
 \rho_L \\
 \rho_L \\
 \rho_L v_L \\
 \rho_L v_L
\end{array}
\right)
\quad
\mbox{and}
\quad 
{\bf U}_R = \frac{1}{2}
\left(
\begin{array}{c}
 \rho_R \\
 \rho_R \\
 \rho_R v_R \\
 \rho_R v_R
\end{array}
\right)
$$
for two given densities $\rho_L >0$ and $\rho_R > 0$ and velocities $v_L > 0$ and $v_R < 0$.
We recall that the function ${\bf M} = {\bf M}({\bf U})$ is defined by (\ref{eq:quadrature}).
We define 
\begin{equation} \label{eq:solution_mesure_cas1}
 {\bf M}(x,t)=
\left\{
\begin{array}{lll}
 {\bf M}_L & \mbox{if} & x<v_R t,\\
 {\bf M}_{\star} & \mbox{if} & v_R t < x < v_L t,\\
 {\bf M}_R & \mbox{if} & x>v_L t,\\
\end{array}
\right.
\end{equation}
with ${\bf M}_{\star} = {\bf M}({\bf U}_{\star})$ given by 
${\bf U}_{\star} = \left(\rho_L, 
 \rho_R,  \rho_L v_L, \rho_R v_R \right)^t$.
Our objective here is to prove that the following Dirac delta functions free solution is an entropy 
solution of (\ref{eq:modele_bipic_short})-(\ref{eq:donnee_initiale}). Conditions 
(\ref{eq:rh_generalisee}) and (\ref{eq:entropy_generalisee}) write here
\begin{equation} \label{eq:rh_generalisee_cas1}
\left\{
\begin{array}{l}
 v_R ({\bf M}_L - {\bf M}_{\star}) - 
\big({\bf F}({\bf M}_L) - {\bf F}({\bf M}_{\star})\big) = 0 \\
\ \\[-1ex]
 v_L ({\bf M}_{\star} - {\bf M}_R) - 
\big({\bf F}({\bf M}_{\star}) - {\bf F}({\bf M}_R)\big) = 0.
\end{array}
\right.
\end{equation}
and 
\vspace{-0.3cm}
\begin{equation} \label{eq:entropy_generalisee_cas1}
\left\{
\begin{array}{l}
 v_R \big(\eta({\bf M}_L) - \eta({\bf M}_{\star})\big) - 
\big(q({\bf M}_L) - q({\bf M}_{\star})\big) \leq 0 \\
\ \\[-1ex]
 v_L \big(\eta({\bf M}_{\star}) - \eta({\bf M}_R)\big) - 
\big(q({\bf M}_{\star}) - q({\bf M}_R)\big) \leq 0,
\end{array}
\right.
\end{equation} 
with $\eta({\bf U}) = \rho_1 S(v_1) + \rho_2 S(v_2)$ and 
$q({\bf U}) = \rho_1 v_1 S(v_1) + \rho_2 v_2 S(v_2)$
for all 
$S(v) = v^{2\alpha}$ with $\alpha \geq 2$. We will focus only on the first equality of (\ref{eq:rh_generalisee_cas1}) 
and the first inequality of (\ref{eq:entropy_generalisee_cas1}), the second ones being treated similarly. We clearly have 
$$
{\bf M}_{L} - {\bf M}_{\star} =
\left(
\begin{array}{c}
 \rho_L \\
 \rho_L v_L \\
 \rho_L v_L^2 \\
 \rho_L v_L^3
\end{array}
\right)
-
\left(
\begin{array}{c}
 \rho_L + \rho_R\\
 \rho_L v_L + \rho_R v_R\\
 \rho_L v_L^2 + \rho_R v_R^2\\
 \rho_L v_L^3 + \rho_R v_R^3\\
\end{array}
\right)
=
- \left(
\begin{array}{c}
 \rho_R\\
 \rho_R v_R\\
 \rho_R v_R^2\\
 \rho_R v_R^3\\
\end{array}
\right),
$$
while 
$$
{\bf F}({\bf M}_{L}) - {\bf F}({\bf M}_{\star}) =
\left(
\begin{array}{c}
 \rho_L v_L\\
 \rho_L v_L^2 \\
 \rho_L v_L^3 \\
 \rho_L v_L^4
\end{array}
\right)
-
\left(
\begin{array}{c}
 \rho_L v_L + \rho_R v_R\\
 \rho_L v_L^2 + \rho_R v_R^2\\
 \rho_L v_L^3 + \rho_R v_R^3\\
 \rho_L v_L^4 + \rho_R v_R^4\\
\end{array}
\right)
=
- \left(
\begin{array}{c}
 \rho_R v_R\\
 \rho_R v_R^2\\
 \rho_R v_R^3\\
 \rho_R v_R^4\\
\end{array}
\right).
$$
It is then clear that the first equality of (\ref{eq:rh_generalisee_cas1}) holds true. Let us now check that the proposed Riemann solution fulfills the entropy condition. We clearly have 
\begin{equation*}
\begin{array}{rcl}
 &v_R \big(\eta({\bf M}_L) - \eta({\bf M}_{\star})\big) - 
\big(q({\bf M}_L) - q({\bf M}_{\star})\big) & \\
&=v_R \big( \rho_L S(v_L) - (\rho_L S(v_L) + \rho_R S(v_R)) \big) - 
\big( \rho_L v_L S(v_L) - (\rho_L v_ L S(v_L) + \rho_R v_R S(v_R)) \big) 
 &\\
 &=0&\\[-1ex]
\end{array}
\end{equation*}
which allows to prove that the proposed solution is an entropic smooth solution.

\subsection{Collision of four particles packets}

We consider a Riemann initial data (\ref{eq:donnee_initiale}) 
where ${\bf M}_L = {\bf M}({\bf U}_L)$ and ${\bf M}_R={\bf M}({\bf U}_R)$ are such that 
$$
{\bf U}_L = \frac{1}{2}
\left(
\begin{array}{c}
 \rho \\
 \rho \\
 \rho v_{1} \\
 \rho v_{2}
\end{array}
\right)
\quad
\mbox{and}
\quad 
{\bf U}_R = \frac{1}{2}
\left(
\begin{array}{c}
 \rho \\
 \rho \\
 -\rho v_{2} \\
 -\rho v_{1}
\end{array}
\right)
$$
for a given density $\rho >0$ and two velocities $v_{2} > v_{1} > 0$. 
We have 
$$
{\bf M}_L = \frac{1}{2}
\left(
\begin{array}{c}
 2 \rho \\
 \rho (v_{1} + v_{2})\\
 \rho (v_{1}^2 + v_{2}^2) \\
 \rho (v_{1}^3 + v_{2}^3)
\end{array}
\right)
\quad
\mbox{and}
\quad 
{\bf M}_R = \frac{1}{2}
\left(
\begin{array}{c}
 2 \rho \\
 -\rho (v_{1} + v_{2})\\
 \rho (v_{1}^2 + v_{2}^2) \\
 -\rho (v_{1}^3 + v_{2}^3)
\end{array}
\right)
$$
$$
{\bf F}({\bf M}_L) = \frac{1}{2}
\left(
\begin{array}{c}
 \rho (v_{1} + v_{2})\\
 \rho (v_{1}^2 + v_{2}^2) \\
 \rho (v_{1}^3 + v_{2}^3) \\
 \rho (v_{1}^4 + v_{2}^4)
\end{array}
\right)
\quad \mbox{and} \quad 
{\bf F}({\bf M}_R) = \frac{1}{2}
\left(
\begin{array}{c}
 -\rho (v_{1} + v_{2})\\
 \rho (v_{1}^2 + v_{2}^2) \\
 -\rho (v_{1}^3 + v_{2}^3) \\
 \rho (v_{1}^4 + v_{2}^4)
\end{array}
\right).
$$
We define 
\begin{equation} \label{eq:solution_mesure_cas2}
 {\bf M}(x,t)=
\left\{
\begin{array}{lll}
 {\bf M}_L & \mbox{if} & x<- \sigma t,\\
 {\bf M}_L^{\delta}(t) \delta(x+\sigma t) & \mbox{if} & x=-\sigma t,\\
 {\bf M}_{\star} & \mbox{if} & -\sigma t < x < \sigma t,\\
 {\bf M}_R^{\delta}(t) \delta(x-\sigma t) & \mbox{if} & x=\sigma t,\\
 {\bf M}_R & \mbox{if} & x>\sigma t,\\
\end{array}
\right.
\end{equation}
with $\sigma > 0$, ${\bf M}_{\star} = {\bf M}({\bf U}_{\star})$ given by 
$$
{\bf U}_{\star} = \frac{1}{2}
\left(
\begin{array}{c}
 \rho_{\star} \\
 \rho_{\star} \\
 - \rho_{\star} v_{\star} \\
 \rho_{\star} v_{\star}
\end{array}
\right), \quad 
{\bf M}_{\star} = 
\left(
\begin{array}{c}
 \rho_{\star} \\
 0\\
 \rho_{\star} v_{\star}^2 \\
 0
\end{array}
\right), \quad \rho_{\star} > 0, \quad v_{\star} > 0,
$$
and ${\bf M}_L^{\delta}(t)$, ${\bf M}_R^{\delta}(t)$ given by 
$$
{\bf M}_L^{\delta}(t) = m(t)
\left(
\begin{array}{c}
 1 \\
 -\sigma \\
 \sigma^2 \\
 -\sigma^3
\end{array}
\right), \quad 
{\bf M}_R^{\delta}(t) = m(t)
\left(
\begin{array}{c}
 1 \\
 \sigma \\
 \sigma^2 \\
 \sigma^3
\end{array}
\right), \quad {\bf F}({\bf M}_{\star}) = 
\left(
\begin{array}{c}
 0\\
 \rho_{\star} v_{\star}^2 \\
 0 \\
 \rho_{\star} v_{\star}^4 \\
\end{array}
\right),
$$ 
with $m(t) \geq 0$. 
The generalized Rankine-Hugoniot jump conditions (\ref{eq:rh_generalisee}) write here 
$$
\left\{
\begin{array}{l}
 -\sigma ({\bf M}_L - {\bf M}_{\star}) t - 
\big({\bf F}({\bf M}_L) - {\bf F}({\bf M}_{\star})\big) t + {\bf M}_L^{\delta}(t) = 0 \\
\ \\
 \sigma ({\bf M}_{\star} - {\bf M}_R) t - 
\big({\bf F}({\bf M}_{\star}) - {\bf F}({\bf M}_R)\big) t + {\bf M}_R^{\delta}(t) = 0.
\end{array}
\right.
$$
that is, equivalently
\begin{equation} \label{eq:rh_gene_cas2}
\left\{
\begin{array}{l}
\displaystyle 2 \sigma {\bf M}_{\star} - \sigma ({\bf M}_L + {\bf M}_{R}) + 
({\bf F}({\bf M}_R) - {\bf F}({\bf M}_{L})) + \frac{m(t)}{t} 
\left(
\begin{array}{c}
 2\\
 0 \\
 2 \sigma^2 \\
 0 \\
\end{array}
\right)
 = 0, \\
\ \\
\displaystyle 2 {\bf F}({\bf M}_{\star}) + \sigma ({\bf M}_R - {\bf M}_{L}) - 
({\bf F}({\bf M}_L) + {\bf F}({\bf M}_{L})) - \frac{m(t)}{t} 
\left(
\begin{array}{c}
 0\\
 2 \sigma \\
 0 \\
 2 \sigma^3 \\
\end{array}
\right)
 = 0.
\end{array}
\right.
\end{equation}
This is made of eight equalities, four are trivial (zero equals zero), so that four are left 
to determine the four unknowns $\rho_{\star}$, $v_{\star}$, $\sigma$ and $m(t)/t$. We propose below to numerically 
solve this nonlinear system for a specific set of values for $\rho$, $v_1$ and $v_2$.\\ 

\begin{rem} 
We conjecture existence and uniqueness of a solution to this nonlinear system. Indeed, the initial condition involves four different velocities while the model is able to represent two different velocities only. 
More precisely and by analogy with the usual pressureless gas dynamics model (see \cite{Bouchut03}, \cite{Bouchut94}), 
velocity $v_1$ is to "bump" into velocities $-v_1$ and $-v_2$ to create a first Dirac delta function. By symmetry, another Dirac delta function is expected.
\end{rem}

Regarding the entropy inequalities (\ref{eq:entropy_generalisee}), we first remark that for $S(v) = v^{2\alpha}$,
$$
\begin{array}{l}
\displaystyle \eta({\bf M}_L) = \eta({\bf M}_R) = \frac{\rho}{2}(S(v_1) + S(v_2)) \\
\eta({\bf M}_{\star}) = \rho_{\star} S(v_{\star})_{\vphantom{\big(}} , \\
\eta({\bf M}_L^{\delta}(t)) = \eta({\bf M}_R^{\delta}(t)) = m(t) S(\sigma),
\end{array}
$$
while 
\vspace{-0.5cm}
$$
\begin{array}{l}
\displaystyle q({\bf M}_L) = -q({\bf M}_R) = \frac{\rho}{2}(v_1 S(v_1) + v_2 S(v_2)), \quad
q({\bf M}_{\star}) = 0.\\[-1.5ex]
\end{array}
$$
As an immediate consequence, both inequalities in (\ref{eq:entropy_generalisee}) are equivalent and the entropy condition
is 
\begin{equation} \label{eq:edr_cas2}
\sigma \big( \rho_{\star} S(v_{\star}) - \frac{\rho}{2} (S(v_{1}) + S(v_{2})) \big)
- \frac{\rho}{2} (v_1 S(v_{1}) + v_2 S(v_{2})) + \frac{m(t)}{t} S(\sigma) \leq 0.
\end{equation}
{\bf A concrete example.} We propose to take $\rho = 1$, $v_1 = 0.8$ and $v_2=1.2$. Numerically solving (\ref{eq:rh_gene_cas2}) 
gives $\rho_{\star} = 1.88265$, $v_{\star} = 1.06026$, $\sigma = 0.87983$ and $m(t)/t = 0.22342$. If the left-hand side of (\ref{eq:edr_cas2}), which 
represents the entropy dissipation rate associated with $S(v) = v^{2\alpha}$, is denoted $D(\alpha)$, a simple 
calculation gives for instance $D(2)=-0.27324$, $D(3)=-0.86854$, $D(4)=-1.88678$,... The proposed solution (\ref{eq:solution_mesure_cas2}) is then 
actually an entropy measure 
solution of the four-moment model. 
Such an exact entropic solution will be used in the following to prove the relevance of the numerical scheme proposed hereafter 
with respect to the exact solution 
when singularities are present.

\subsection{Piecewise linear solution connected with the frontier of the moment space}
\label{trans}

As a last example, we introduce a piecewise linear solution which allows to connect zones inside the moment space with zone at the frontier within the proper framework introduced in subsection \ref{border}.
Particles are initiated in the domain $[0, \, 0.5]$. In the domain $[0,\, 0.1]$, a monomodal velocity distribution is reconstructed, with $v_1=v_2=1$,
and $\rho_1=\rho_2=0.5$. On the contrary, a bimodal velocity distribution is reconstructed in the domain $[0.1,\, 0.4]$, with $\rho_1=\rho_2=0.5$ and for the 
abscissas, $v_1= 1 + \frac{x-0.1}{0.3}$ and $v_2=1$. The initial conditions are represented in Fig. \ref{fig:ci_transition}.

 \begin{figure}[h]
  \begin{center}
         \includegraphics[width=0.27\textwidth]{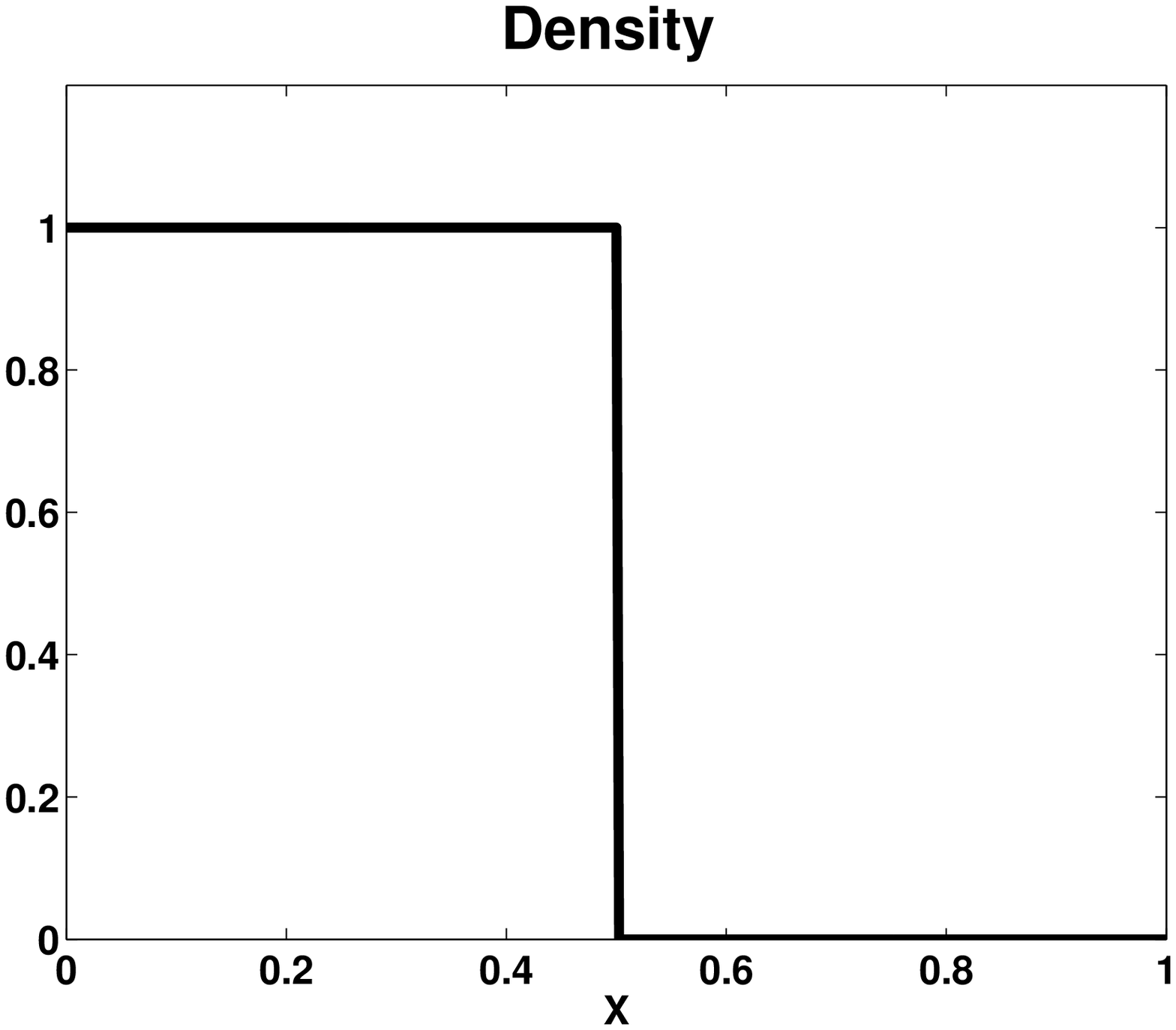}
          \includegraphics[width=0.27\textwidth]{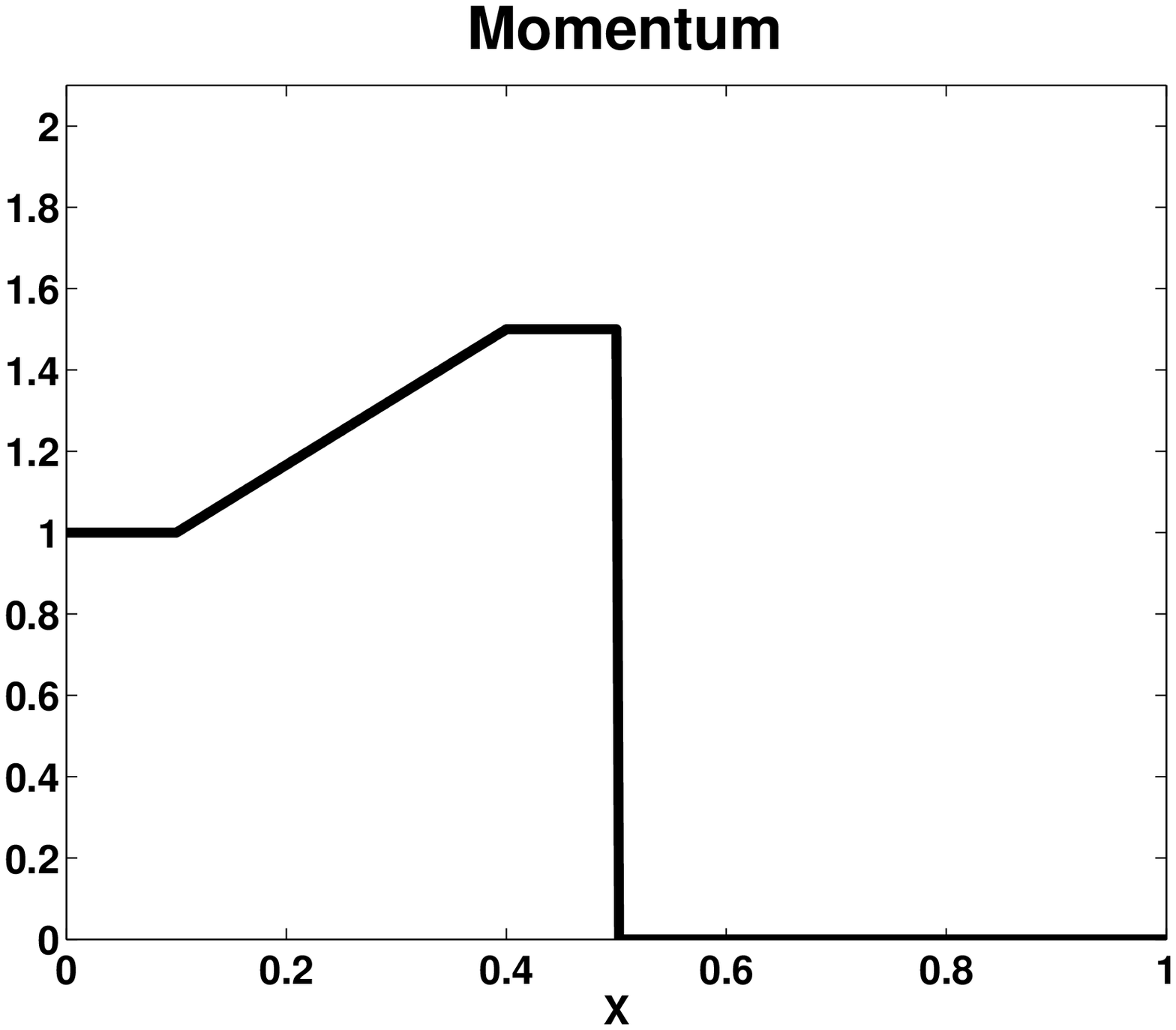} \\
           \includegraphics[width=0.27\textwidth]{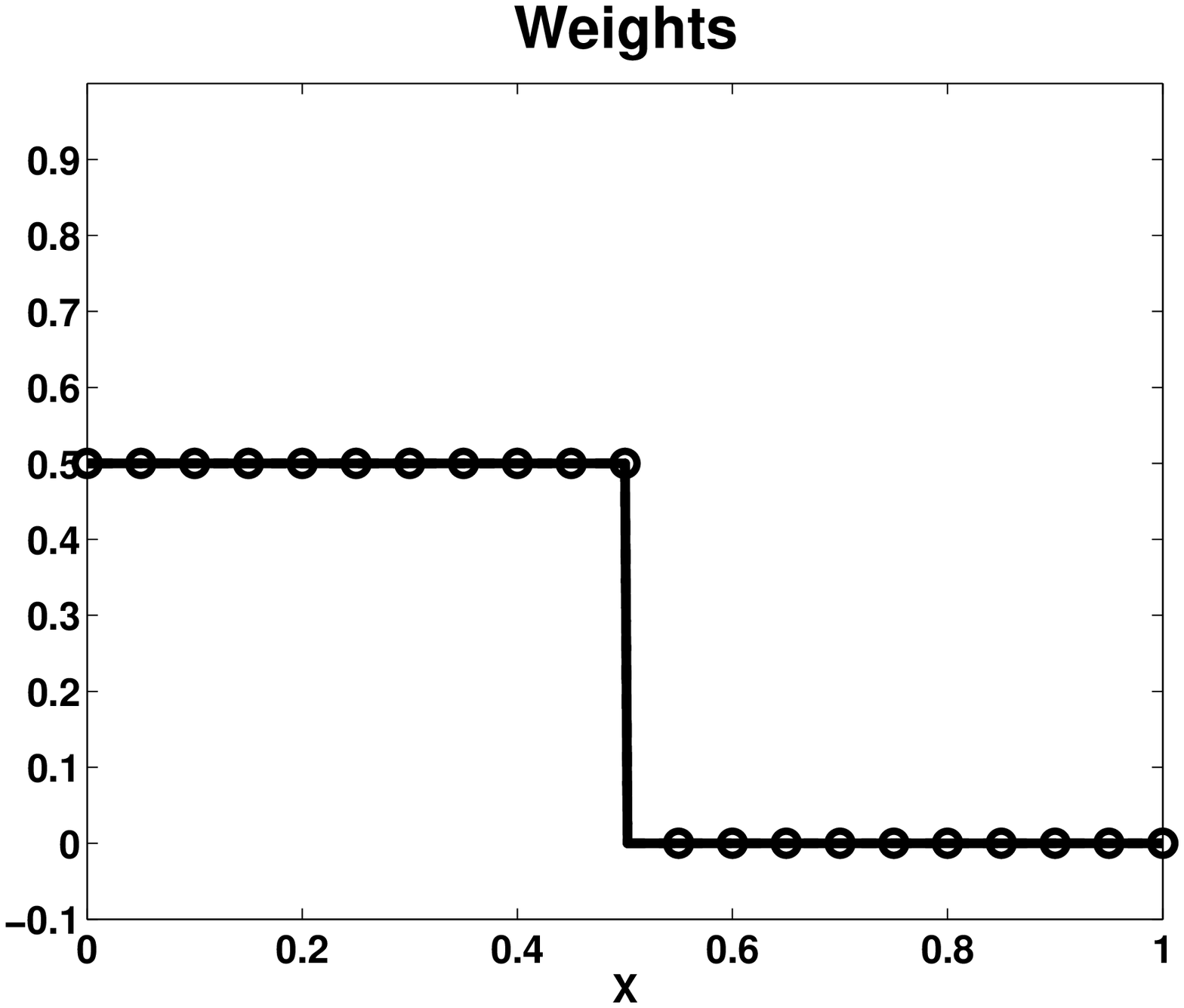}
            \includegraphics[width=0.27\textwidth]{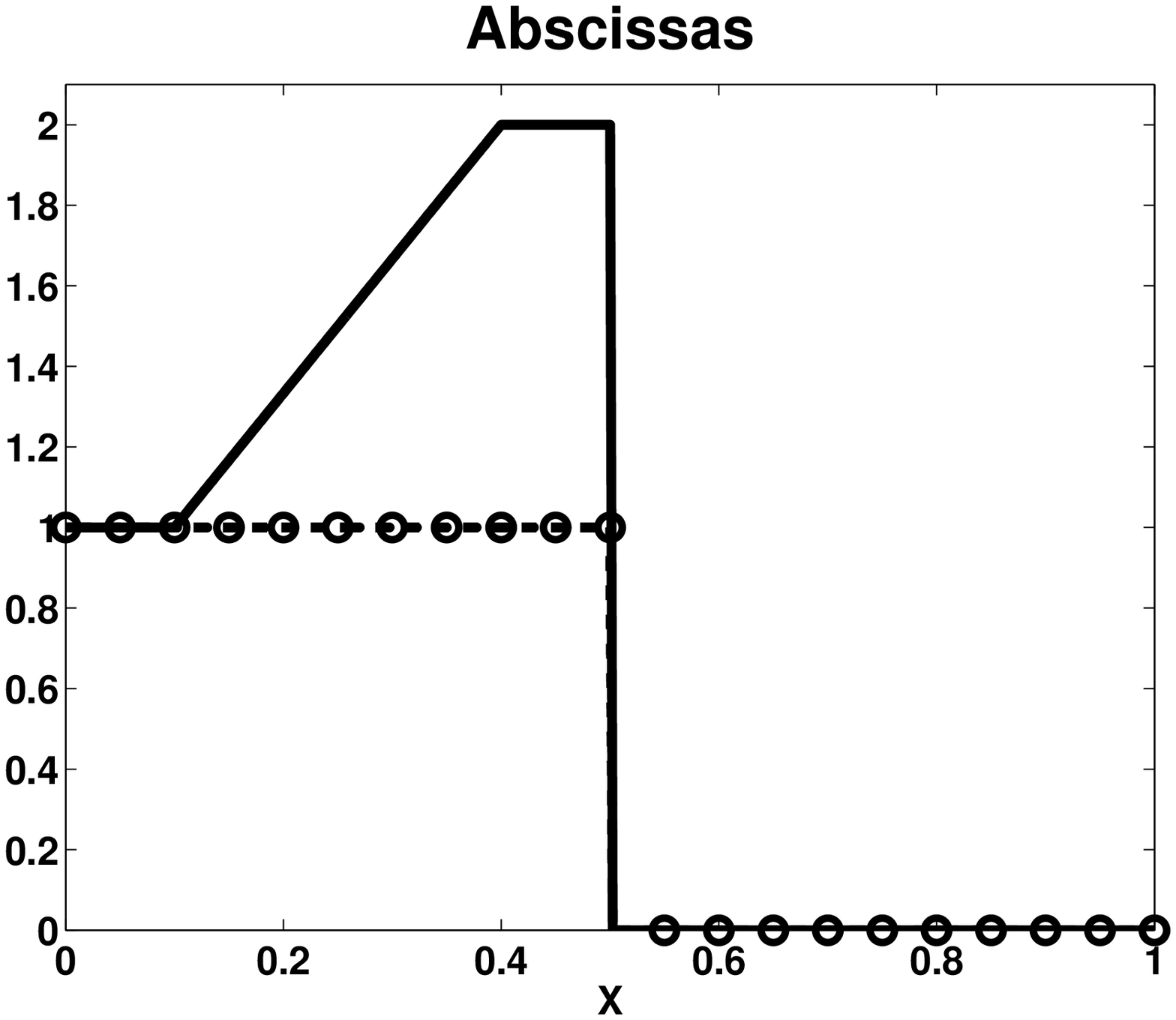}
	\caption{Moment dynamics for a free boundary connecting areas where $e=0$ and $e>0$. Initial conditions. Top-left: $M_0$. Top-right: $M_1 $, Bottom-left: weights. Bottom-right: abscissas. The solid line corresponds to ($\rho_1,v_1$), the dashed line with circles to ($\rho_2,v_2$). }
		\label{fig:ci_transition} 
  \end{center}
\end{figure}

 The ground difference
with the first test case is that the transition between the two zones is smooth, 
and so the numerical strategy to account for this transition is  important.
The analytical solution of this problem, in smooth areas, consists of a decoupled transport of each of the quadrature nodes
as two independant pressureless gas as showed in system (\ref{eq:2gsp}). This comes from the fact that the number of Dirac delta functions
reconstructed from the moments, two in this case, is always sufficient to capture the problem dynamics.
The equivalence between the kinetic and
macroscopic equations is preserved, and the solution in terms of moments satisfies the entropy equation.
Therefore, the solution  is the superposition of a  translation at constant velocity $v_2=1$, and a  transport with the following 
velocity field:
\begin{equation}
 v_1 =  \left \{ \begin{array}{cc}  1 &  x \in [0, \, 0.1] \\
 						1 + \frac{x-0.1}{0.3} & x \in [0.1,\, 0.4] \\
						2 & x \in [0.4,\, 0.5]
		    \end{array} \right.
\end{equation}

 \begin{figure}[h]
  \begin{center}
             \includegraphics[width=0.33\textwidth]{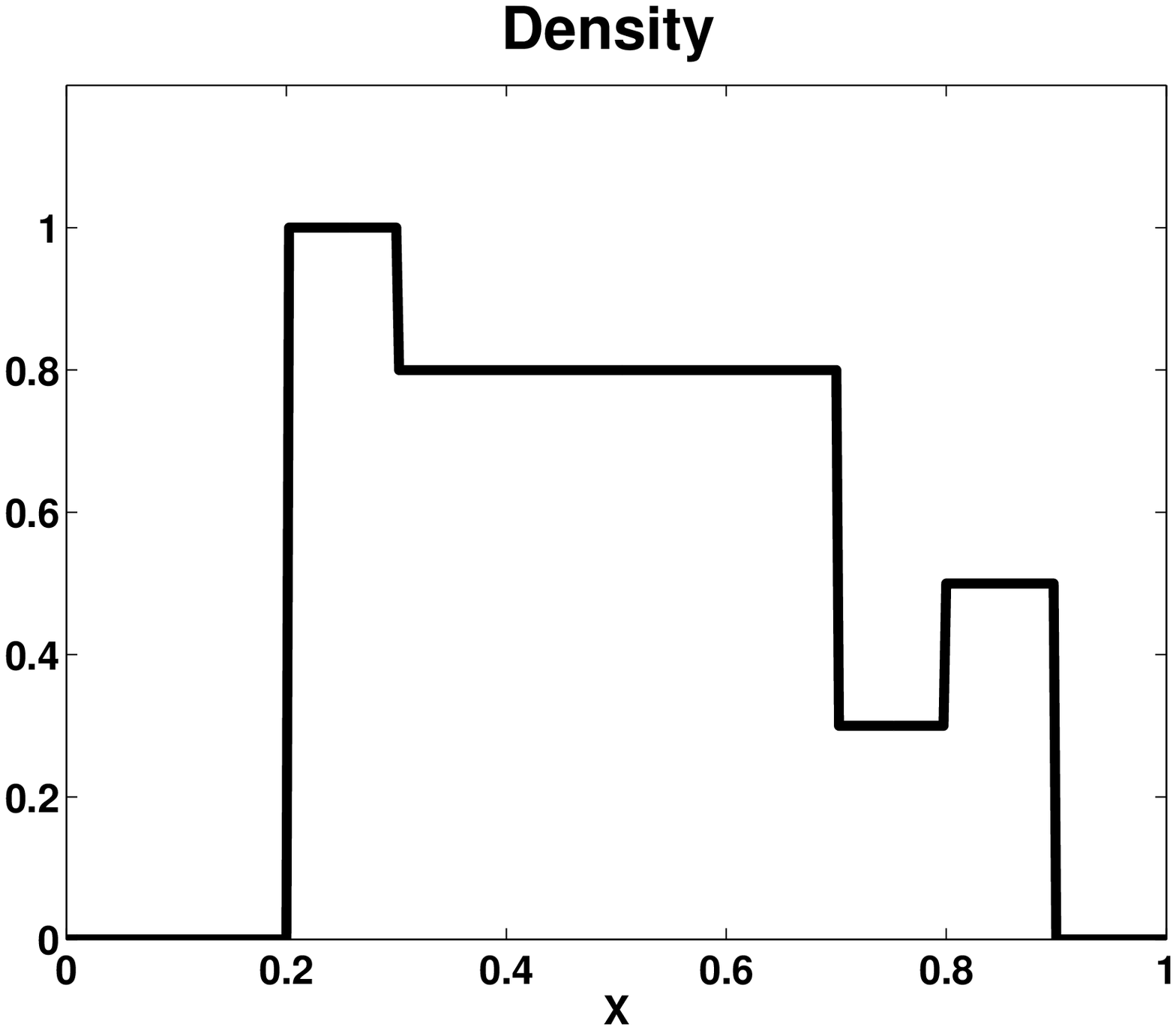}
          \includegraphics[width=0.33\textwidth]{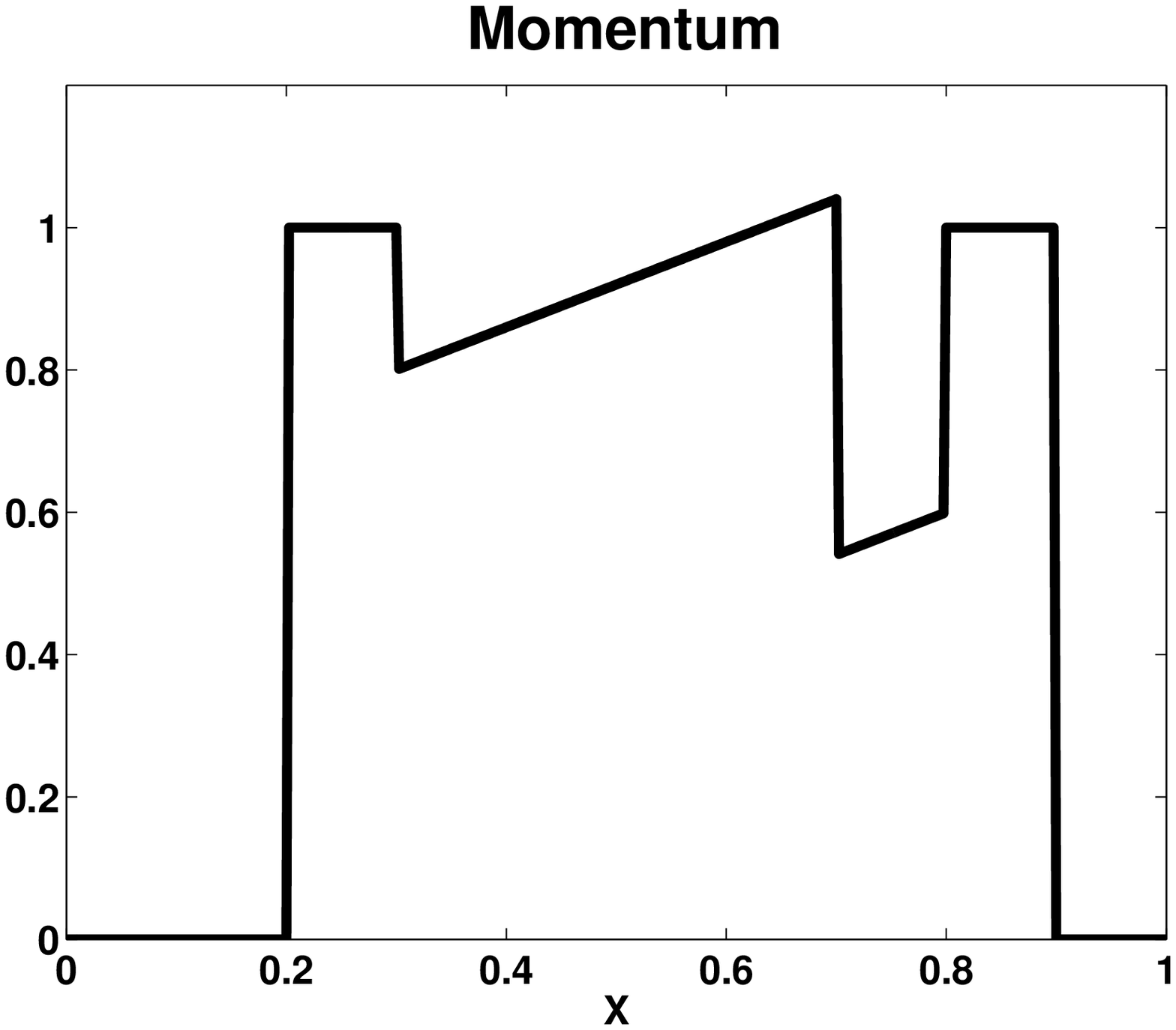}
              \includegraphics[width=0.33\textwidth]{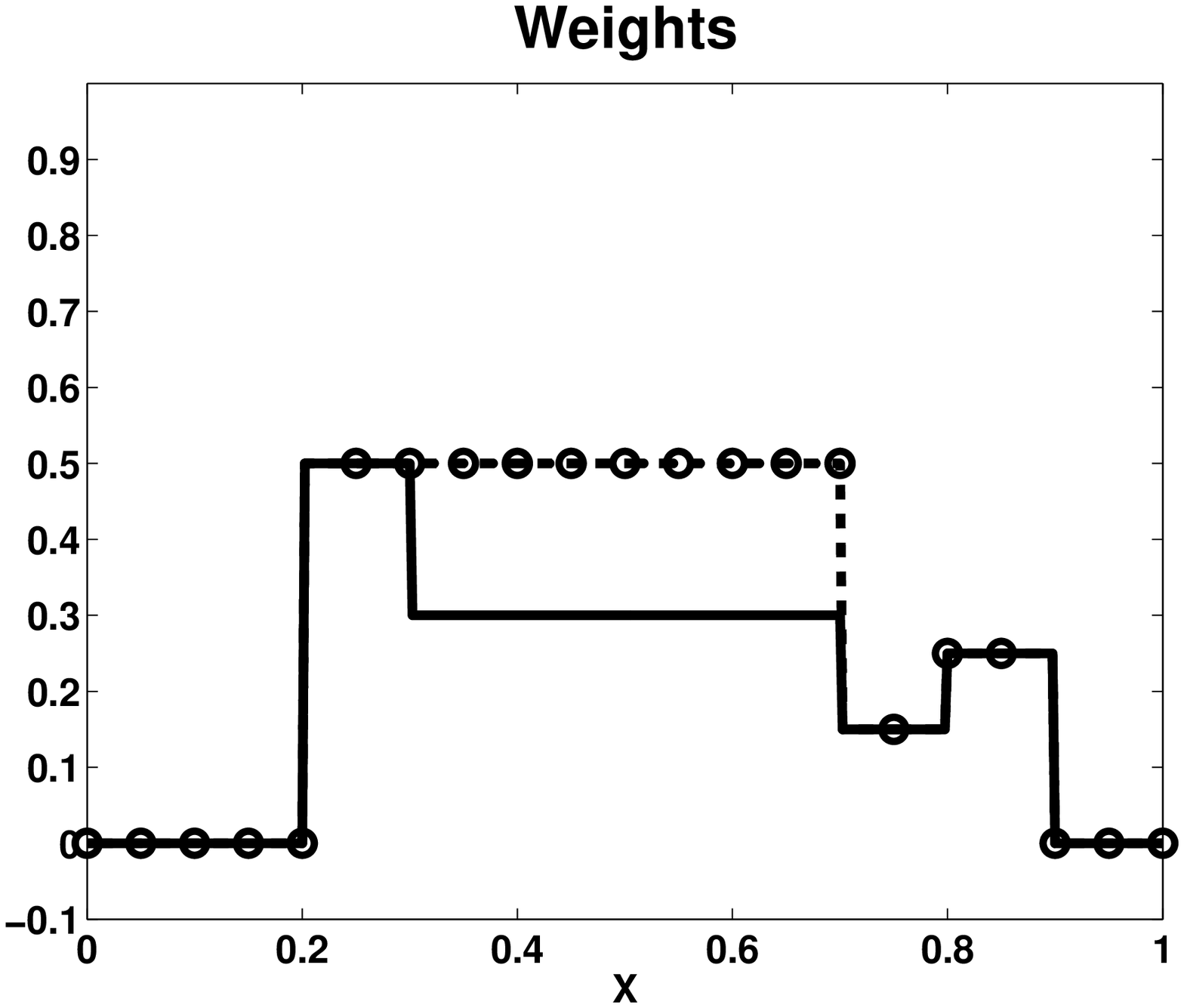}
            \includegraphics[width=0.33\textwidth]{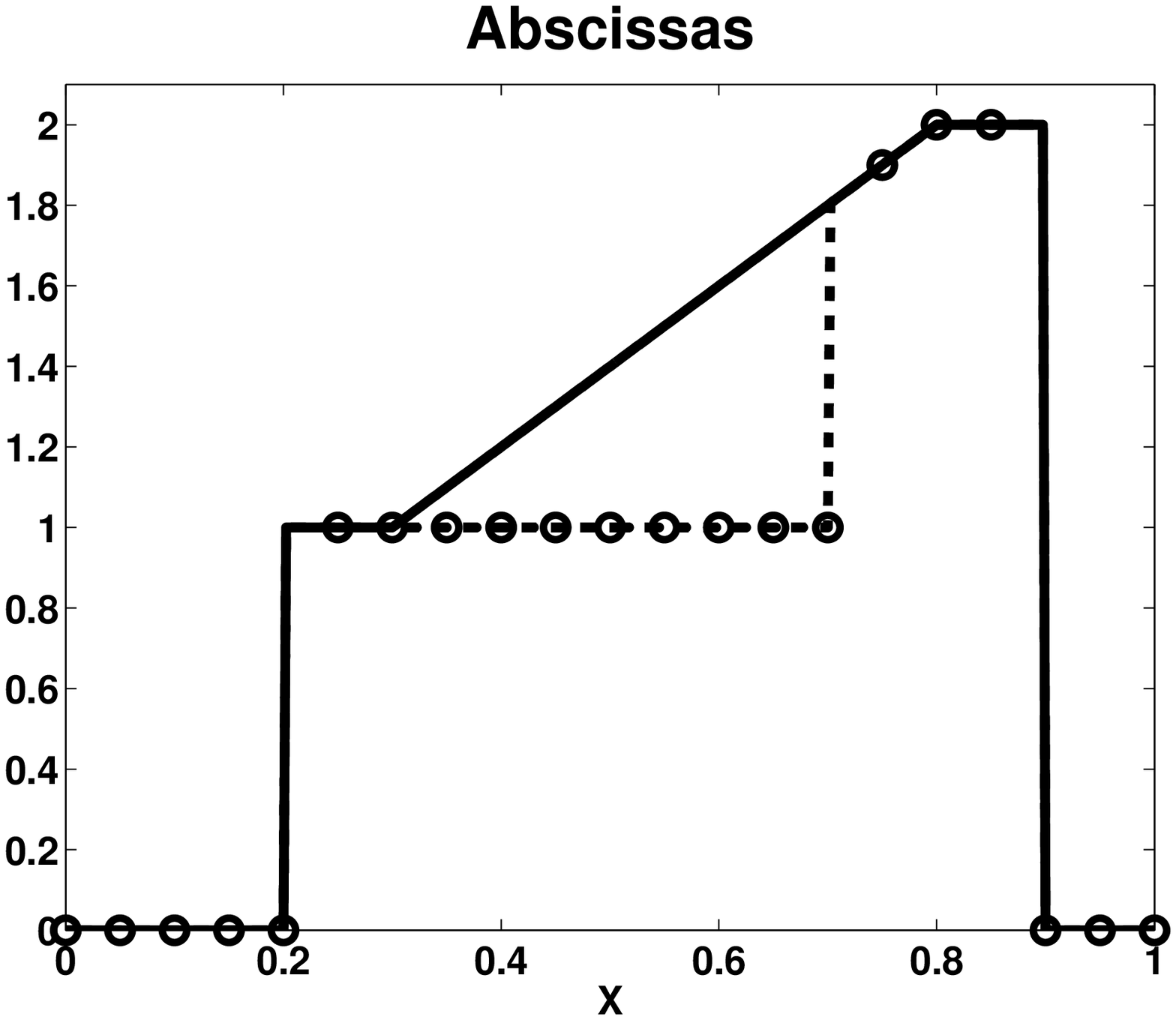}
	\caption{Moment dynamics for a free boundary connecting areas where $e=0$ and $e>0$. Solution at $t=0.2$. Top-left: $M_0$. Top-right: $M_1 $. Bottom-left: weights, Bottom-right:abscissas. The solid line corresponds to ($\rho_1,v_1$), the dashed line with circles to ($\rho_2,v_2$). }
		\label{fig:trans_ana} 
  \end{center}
\end{figure}

 The analytical solution, displayed in Fig.\, \ref{fig:trans_ana},
   has  two fronts at $x=0.2$ and $x=0.9$ moving with a velocity $v=1$ and $v=2$ respectively. The first front 
 corresponds to particles initiated with velocity $v=1$ between $x=0$ and $x=0.1$. The second front corresponds to particles initiated 
 with velocity $v=2$ between $x=0.4$ and $x=0.5$. The square wave between $x=0.8$ and $x=0.9$ is the final location of the particles initiated
 with velocity $v=2$ between $x=0.4$ and $x=0.5$.
The value of $\rho_1$ between $x=0.3$ and $x=0.8$ corresponds to the expansion of the 
density field due to transport with a linear velocity field with positive slope. The value of the density in that area is the solution of the equation
$\partial_t \rho_1 + v_1\partial_x \rho_1 = -\rho_1\partial_x v_1$, which yields $\rho_1=0.3$ at time $t=0.2$ \footnote{In the interval $[0.8, 0.9]$, although $\rho_2$ should be null (the square wave at velocity $v_2=1$ should be bounded 
between the front $x=0.2$ and $x=0.7$) we computed $\rho_2=0.5$ and $v_2=2$. This is consistent 
with the conditions  in Section\, \ref{border} for  moment vectors at the frontier of the moment space.}.

According to the initial conditions, both weights have the same profile, the quantities ${q}/{(M_0\,e)}$ and ${q}/{e^{3/2}}$ are null.
At time $t=0.2$, the weight profiles are different in the interval $[0.3, 0.8]$ corresponding to the expansion of $\rho_1$. Therefore, ${q}/{e^{3/2}}$ 
is non null, as well as $\frac{q}{M_0e}$, since the velocities have also different values, and can be exactly calculated (see Fig. \ref{fig:res_transition_2}).



\section{Numerical simulations via kinetic schemes}

This section is devoted to the discretization of (\ref{eq:modele_bipic_short})-(\ref{eq:M4})-(\ref{eq:quadrature}). 
As already stated, 
we use as a building block a natural first-order kinetic scheme already proposed in the literature \cite{jin03,gosse03,dfv08} and 
briefly recalled 
here for the sake of completeness.

Let us first introduce a time step $\Delta t>0$ and a space step 
$\Delta x>0$ that we assume to be constant for simplicity. We set $\lambda = {\Delta t}/{\Delta x}$ and define
the mesh interfaces $x_{j+1/2}=j \Delta x$ for $j \in \mathbb{Z}$, and the 
intermediate times $t^n=n \Delta t$ for $n \in \mathbb{N}$. In the sequel,
${\bf M}^n_j$ denotes the approximate value of ${\bf M}$ at time 
$t^n$ and on the cell $\mathcal{C}_j = [x_{j-1/2}, x_{j+1/2})$. 
For $n=0$, we set  ${\bf M}^0_j = \frac{1}{\Delta x} \int_{x_{j-1/2}}^{x_{j+1/2}} {\bf M}_0(x) dx, \quad j \in \mathbb{Z}$,
where ${\bf M}_0(x)$ is the initial condition. \\
Let us now assume as given $({\bf M}^n_j)_{j \in \mathbb{Z}}$ the sequence in $\Omega$ of approximate values at time 
$t^n$. In order to advance it to the next time level $t^{n+1}$, the kinetic scheme is decomposed into 
two steps. \\
\ \\ 
{\it First step : transport} ($t^n \to t^{n+1-}$) \\
We first set ${\bf U}^n_j = {\bf U}({\bf M}^n_j)$ and define the function 
$(x,v) \to {f}^n(x,v)$ by
$$
{f}^n(x,v) = (\rho_1)^n_j \delta \big(v - (v_1)^n_j\big) + (\rho_2)^n_j \delta \big(v - (v_2)^n_j\big),\quad
\forall \,\, (x,v) \in \mathcal{C}_j \times \mathbb{R}, \,\,\, j \in \mathbb{Z}.
$$ 
We then solve the transport equation 
\begin{equation} \label{eq:te}
\left\{
\begin{array}{l}
 \partial_t f + v \partial_x f = 0, \quad (x,v) \in \mathbb{R} \times \mathbb{R}, \\
 f(t=0,x,v) = f^n(x,v),
\end{array}
\right.
\end{equation}
the solution of which is given by $f(t,x,v) = f^n(x-vt,v)$. \\
At last, we set ${f}^{n+1-}(x,v) = f^n(x-v \Delta t,v)$. \\
\ \\
{\it Second step : collapse} ($t^{n+1-} \to t^{n+1}$) \\
The first four moments at time $t^{n+1}$ are now naturally defined by setting
$$
(M_i)^{n+1}_j = \frac{1}{\Delta x} \int_{x_{j-1/2}}^{x_{j+1/2}} \int_{-\infty}^{+\infty} v^i {f}^{n+1-}(x,v) dv dx.
$$
Then, we have ${\bf M}^{n+1}_j = \big((M_0)^{n+1}_j,(M_1)^{n+1}_j,(M_2)^{n+1}_j,(M_3)^{n+1}_j\big)^t$ for all 
$j \in \mathbb{Z}$, which completes the algorithm description. \\

\begin{rem} 
It is easy to see that this scheme preserves the moment space $\Omega$, see for instance \cite{dfv08}. \\
\end{rem}

\begin{rem}
Under the natural CFL condition ${\Delta t}\, \max_{j \in \mathbb{Z}}((v_1)^n_j,(v_2)^n_j) \leq CFL \, \Delta x$,
with $CFL \leq 1$, integrating (\ref{eq:te}) over $(t,x,v) \in 
(0,\Delta t) \times \mathcal{C}_j \times \mathbb{R}$ and against $v^i$, $i=0,...,3$ easily leads to the 
equivalent update formula 
$$
{\bf M}^{n+1}_j = {\bf M}^{n}_j - \frac{\Delta t}{\Delta x} \big( {\bf F}^{n}_{j+1/2} - {\bf F}^{n}_{j-1/2} \big), \quad 
j \in \mathbb{Z},
$$
where we have set ${\bf F}^{n}_{j+1/2} = \big((M_1)^{n}_{j+1/2},(M_2)^{n}_{j+1/2},(M_3)^{n}_{j+1/2},(M_4)^{n}_{j+1/2}\big)^t$ and
$(M_i)^{n}_{j+1/2} = (M_i)^{n+}_{j+1/2} + (M_i)^{n-}_{j+1/2}$, and
with 
$$
(M_i)^{n-}_{j+1/2} = (\rho_1)^n_{j+1} \min(0,(v_1)^n_{j+1}) \big((v_1)^n_{j+1}\big)^{i-1} + \min(0,(v_2)^n_{j+1}) (\rho_2)^n_{j+1} \big((v_2)^n_{j+1}\big)^{i-1}, 
$$
\vspace{-0.5cm}
$$
(M_i)^{n+}_{j+1/2} = (\rho_1)^n_j \max(0,(v_1)^n_{j+1}) \big((v_1)^n_j\big)^{i-1} + (\rho_2)^n_j \max(0,(v_2)^n_{j+1}) \big((v_2)^n_j\big)^{i-1}.
$$
\end{rem}

\subsection{Numerical quadrature strategy at the frontier of the moment space}

This paragraph addresses the issue of how to numerically handle the transition between a vector in the interior of the moment space, 
and a vector lying at its frontier. For an isolated point at the frontier of the moment space $\Gamma$, there is no specific problem since we use a single quadrature node and the quadrature is not a problem. The two problems we have to face are related to the preservation of the cone in which we envision to work in section 2 as well as to deal with finite precision algebra in the neighborhood of the point $(0,0)$ in the $(e,q)$ plane.

Consequently we introduce two constants in the numerical quadrature we use. First, for finite precision algebra and in order to avoid numerical errors, we only evaluate the two quadrature nodes when $e/M_0^2 > \epsilon_1$, where $\epsilon_1$ is  a small number related to machine precision.
Under this threshold, the velocity dispersion is considered null ($e=0$), and the set of Dirac delta functions are reconstructed 
as suggested in section $2$:
$\rho_1 = \rho_2 = M_0/2$,  $v_1 = v_2 = M_1/M_0$.


\begin{figure}[h]
  \begin{center}
         \includegraphics[width=0.35\textwidth]{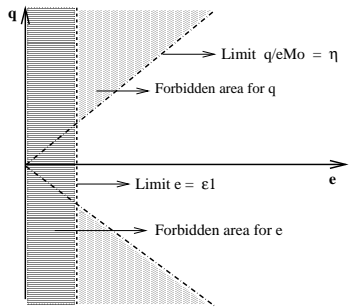}
	\caption{Handling of the moment space border with the admissible cone.  }
		\label{fig:conditions} 
  \end{center}
\end{figure}


Secondly, since we want to deal with compactly supported velocity distribution at the kinetic level, we will introduce another constant, $\eta$ which is a bound for the distance between the two abscissas. In fact we would like to impose to the solution to remain inside the cone in the $(e,q)$ plane : $\frac{\vert q \vert}{M_0e}\le \eta$. It will be shown in the following examples that the cone is in fact automatically preserved by the proposed algorithm and that such a bound does not have to be imposed but is satisfied by the numerical solution.




If $\frac{\vert q \vert}{M_0e}>\eta$, then we set $q/(M_0e) =sign(q)\,\eta$, so that~:  
\begin{equation*}
\frac{q}{M_0e} =sign(q)\,\eta, \quad 
q = sign(q)\,\eta\, M_0\, e, \qquad  M_3 = q + M_1M_2/M_0. 
\end{equation*}
As illustrated in Fig \ref{fig:conditions} and explained in Section\, \ref{border}, the variable $q$ naturally lives in a cone delimited by the straight lines of slope $\pm \eta$ and becomes null when $e\le\epsilon_1$.

Let us remark that limiting $\frac{q}{M_0e}$ does not lead to a limitation on the quantity $\frac{q}{e^{3/2}}$ in such a way as to allow the possibility to reach large ratio in terms of weights (one weight can approach a zero value while the quantity $\frac{q}{e^{3/2}}$ grows indefinitely) but without allowing the abscissas distance to grow beyond of fixed limit naturally inherited from the initial solutions at the kinetic level. We will come back to this point in the following result section.
%
For the simulations we present, the parameters are such that: $\epsilon_1 = 10^{-9}$ and $\eta=2$. 



\subsection{Numerical results}

This section is devoted to  numerical illustrations of the two Riemann problems and one dedicated to the case of a free boundary connecting 
zone where $e=0$ and a zone where $e>0$, discussed in Section~$6$.


In all the figures, we choose to represent $\rho_1$ and $v_1$ by solid lines, and $\rho_2$ and $v_2$ by lines 
with circle markers. In the representation 
of weights and abscissas, values have to be assigned for $v_1$ and $v_2$ : we thus decide that $v_1$ is the 
maximum of the relative values of velocity.

\subsubsection*{Two packet collision}

Figure \ref{fig:two_clouds_ini} represents the initial conditions for the two particle packet case.
Figures \ref{fig:two_clouds_1} and \ref{fig:two_clouds_2} presents the numerical and analytical solutions for the two particle packet case 
with $\rho_L=\rho_R=1$, and $v_L=1$, $v_R=-1$. The computation is run with a $1000$ cell grid on the spatial domain 
$[0,1]$, with $CFL=1$. 
The length of each packet is $0.4$ ($\rho_1=\rho_2=v_1=v_2=0$ for $x \leq 0.1$ and $x \geq 0.9$) and the two packets 
start to collide exactly at time $t=0$. Moving in opposite direction one across the other, with the same opposite speed, they then overlay
and we note in particular that $\rho_1=\rho_2=1$ and $v_1=v_2=0$ in the mixing zone (see for instance the plots 
at time $t=0.1$). As expected, they finally get separated again and we note that a perfect agreement is obtained 
with the exact entropic solution. \\

\begin{figure}[htbp]
  \begin{center}
\includegraphics[width=0.3\textwidth]{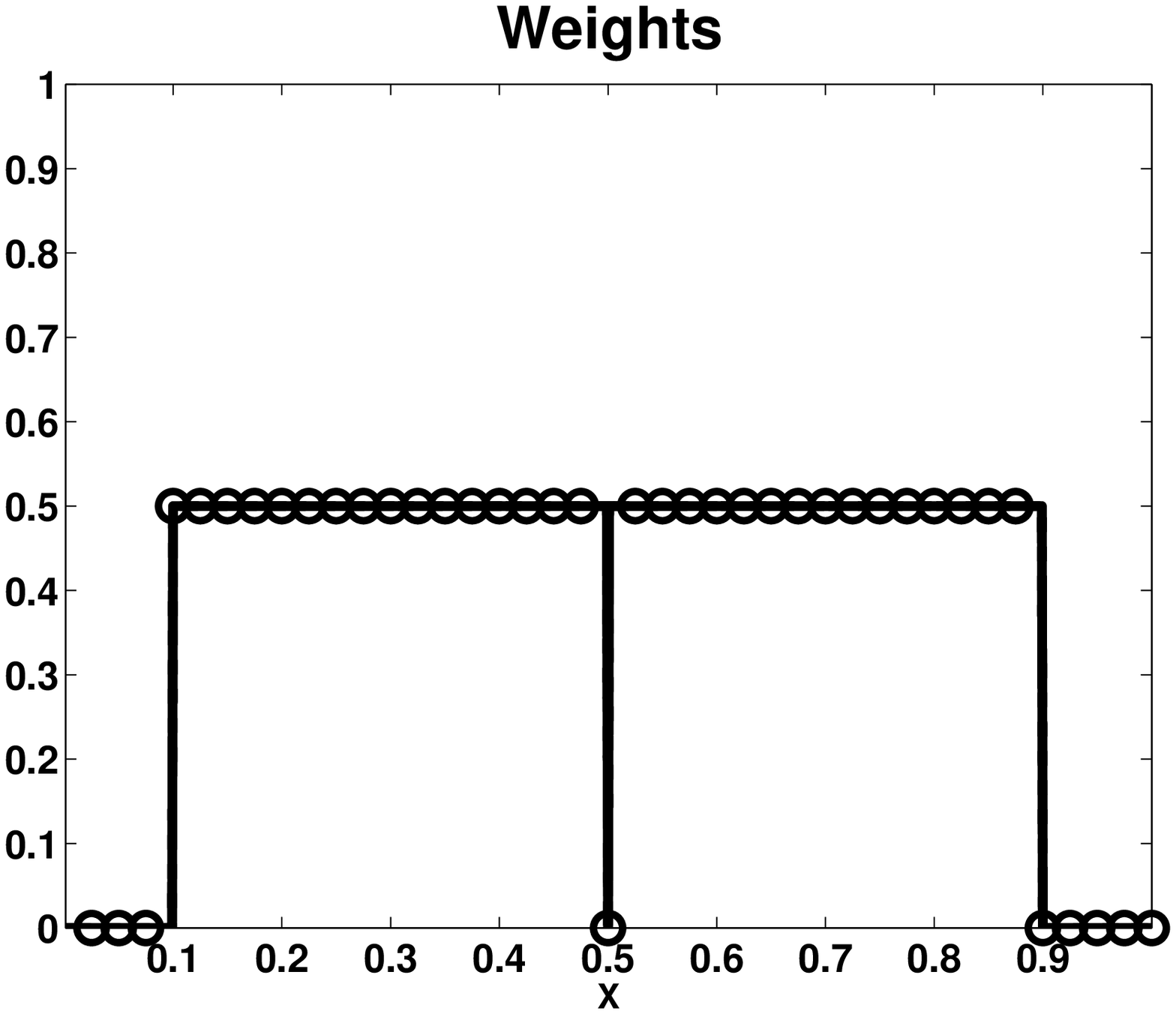}
\includegraphics[width=0.3\textwidth]{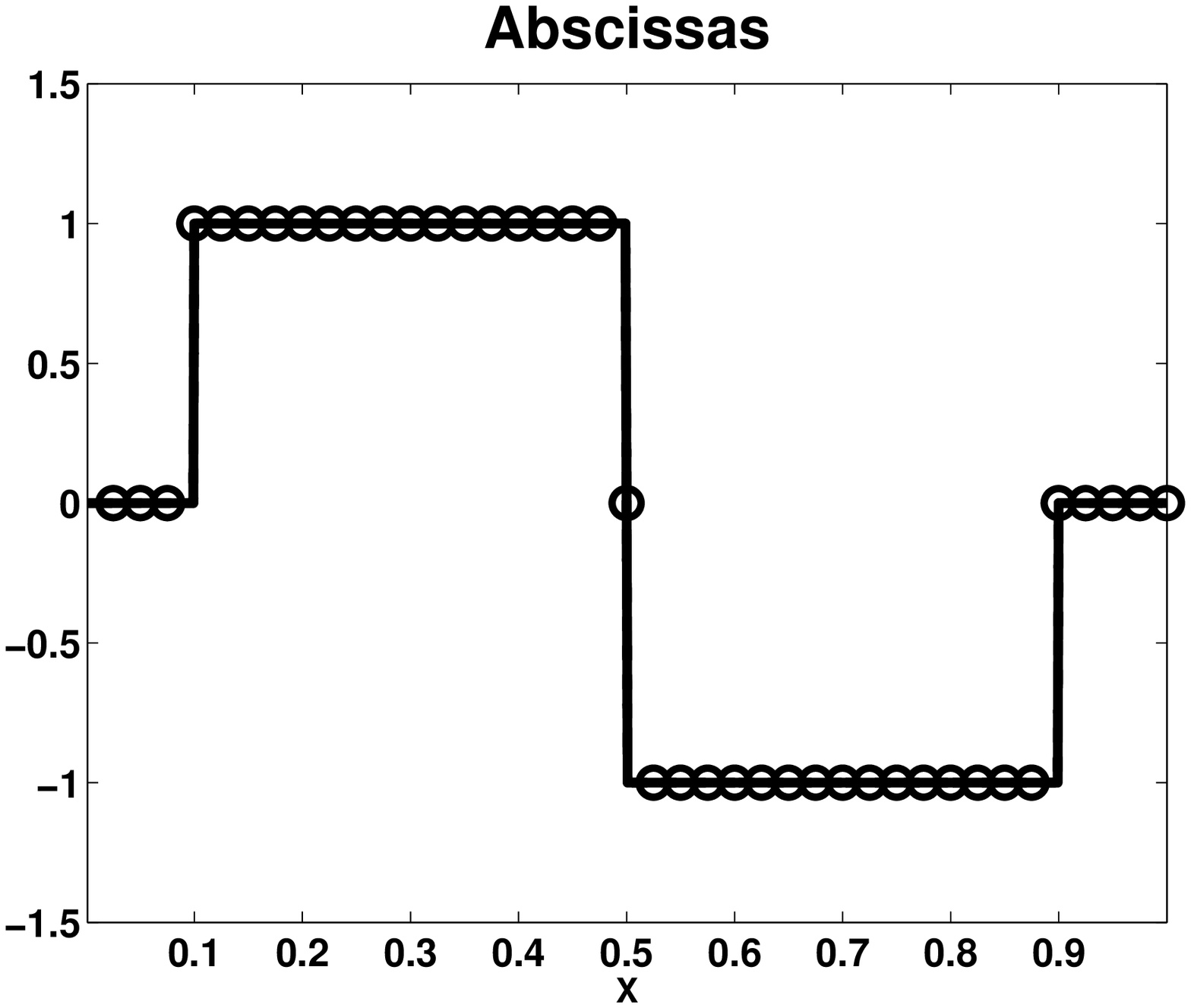}   
			\label{fig:two_clouds_ini} 
			\caption{Initial fields of weights (left) and abscissas (right) for the two particle packet case. The first 
			quadrature node is represented by solid lines whereas the second node is represented by circles.}
  \end{center}
\end{figure}

\begin{figure}[htbp]
  \begin{center}
	  \includegraphics[width=0.33\textwidth]{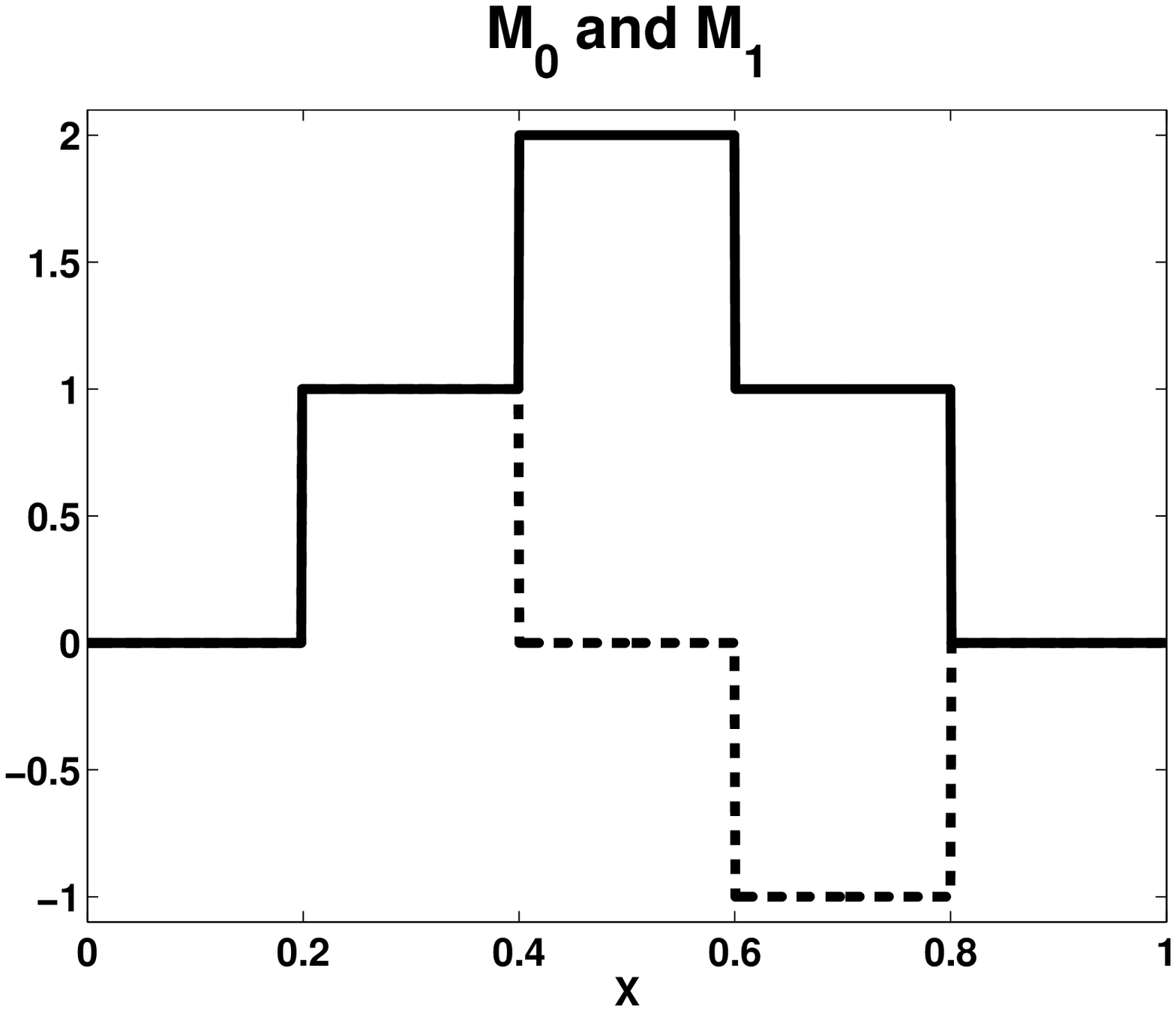}
	\includegraphics[width=0.33\textwidth]{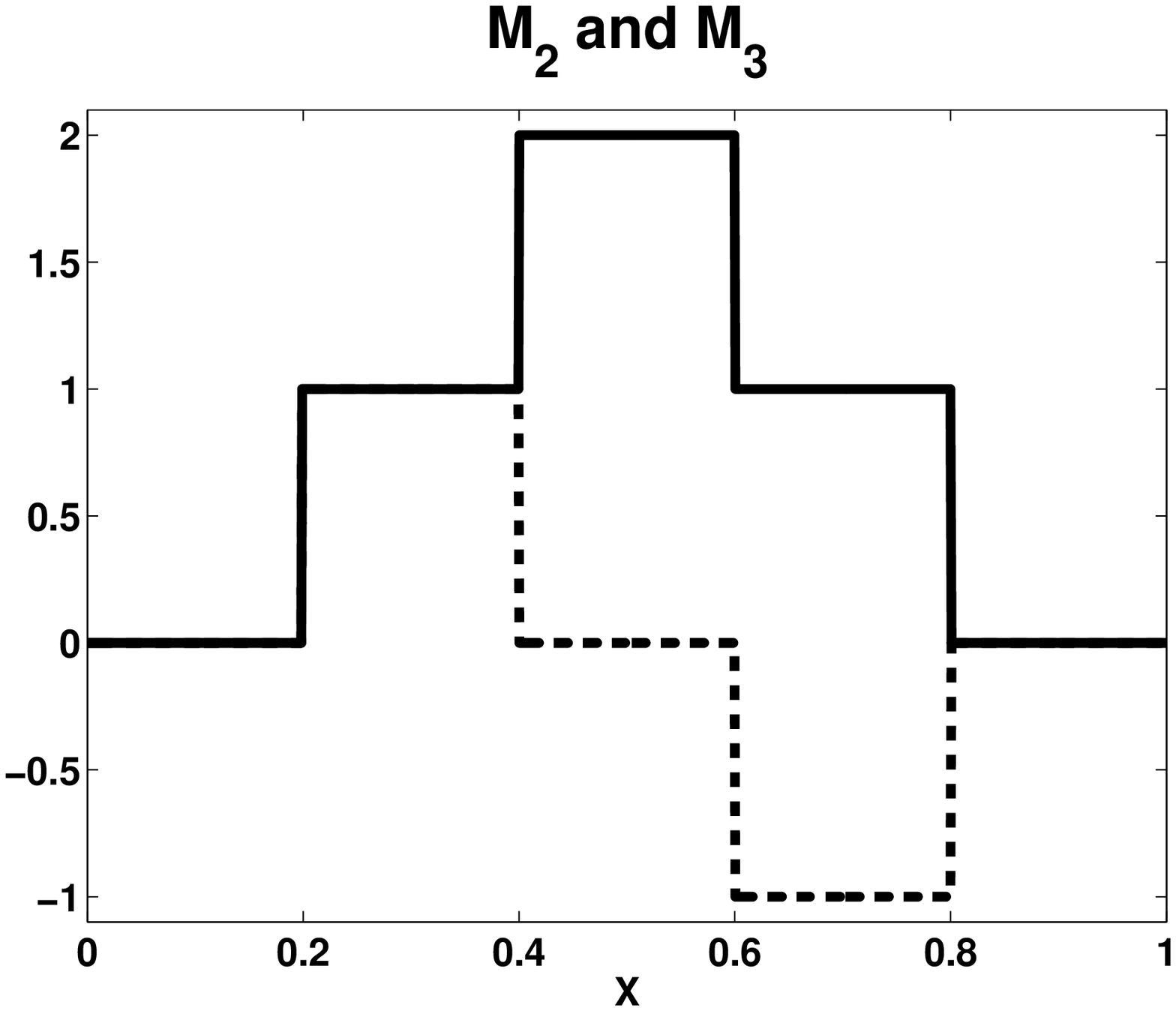}   	
	\includegraphics[width=0.33\textwidth]{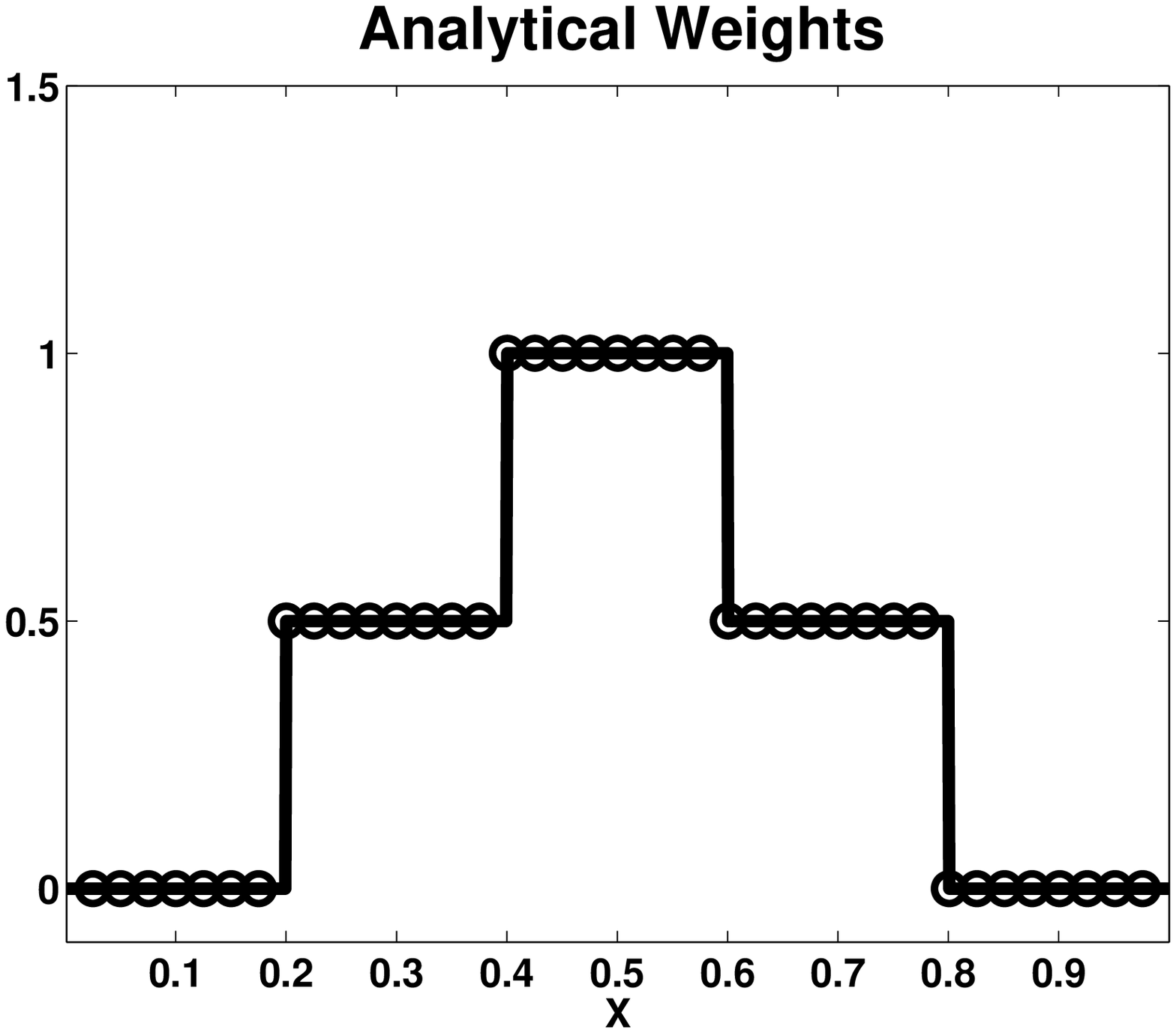}
	\includegraphics[width=0.33\textwidth]{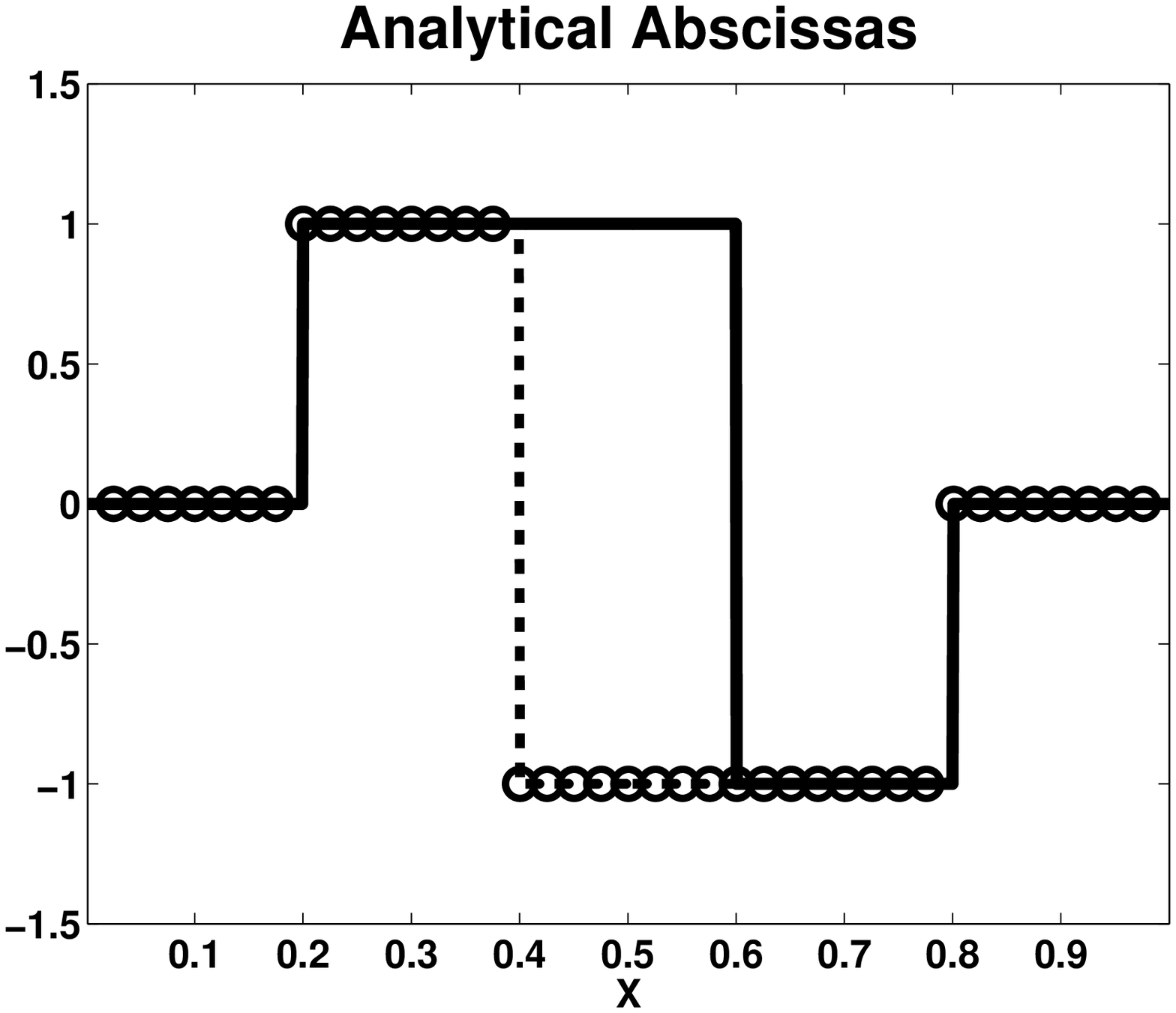}
	\includegraphics[width=0.33\textwidth]{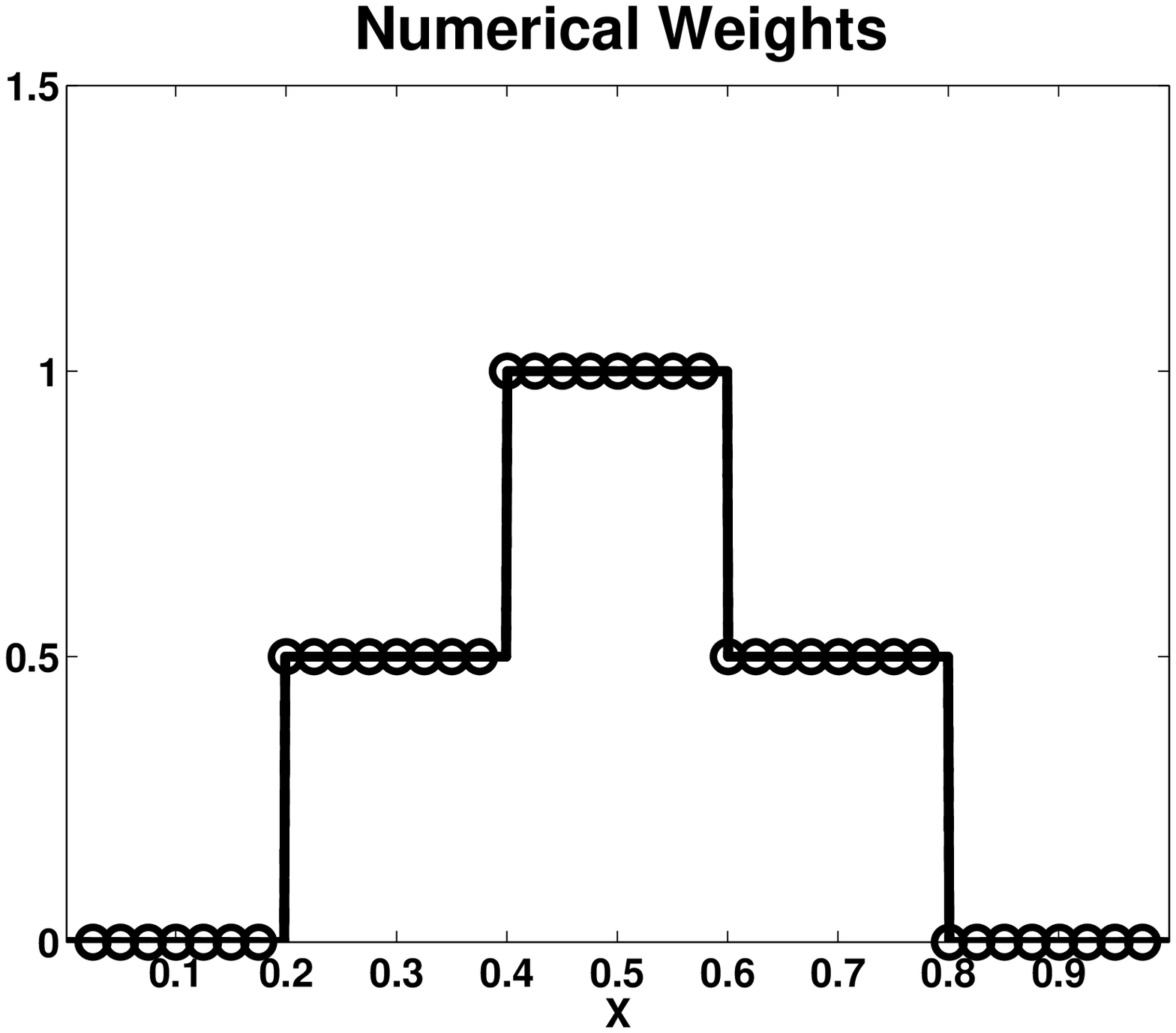}
	\includegraphics[width=0.33\textwidth]{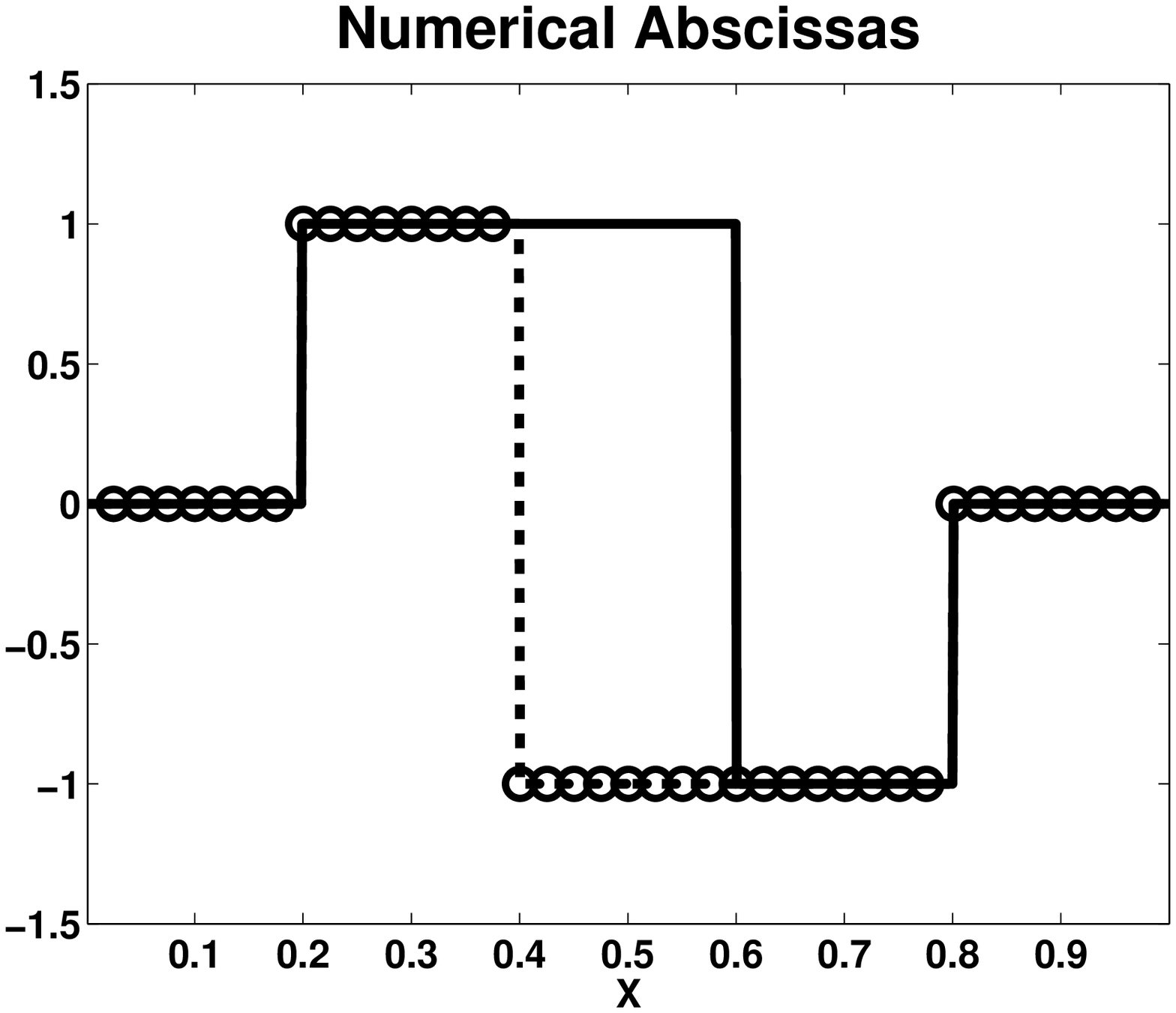}
	
	\caption{ Results for the two packet case at time $t=0.1$. Top: Numerical results for $M_0$ (solid line) and $M_1$ (dashed line) (left) 
and for  $M_2$ (solid line) and $M_3$ (dashed line) (right). 
Middle: Analytical result for weights (left) and abscissas (right).
Bottom: Numerical results for weights (left) and abscissas (right).
 The solid line corresponds to the higher abscissa, the dashed line with circle to 
the lower one.}
		\label{fig:two_clouds_1} 
  \end{center}
\end{figure}

\begin{figure}[htbp]
  \begin{center}
	\includegraphics[width=0.3\textwidth]{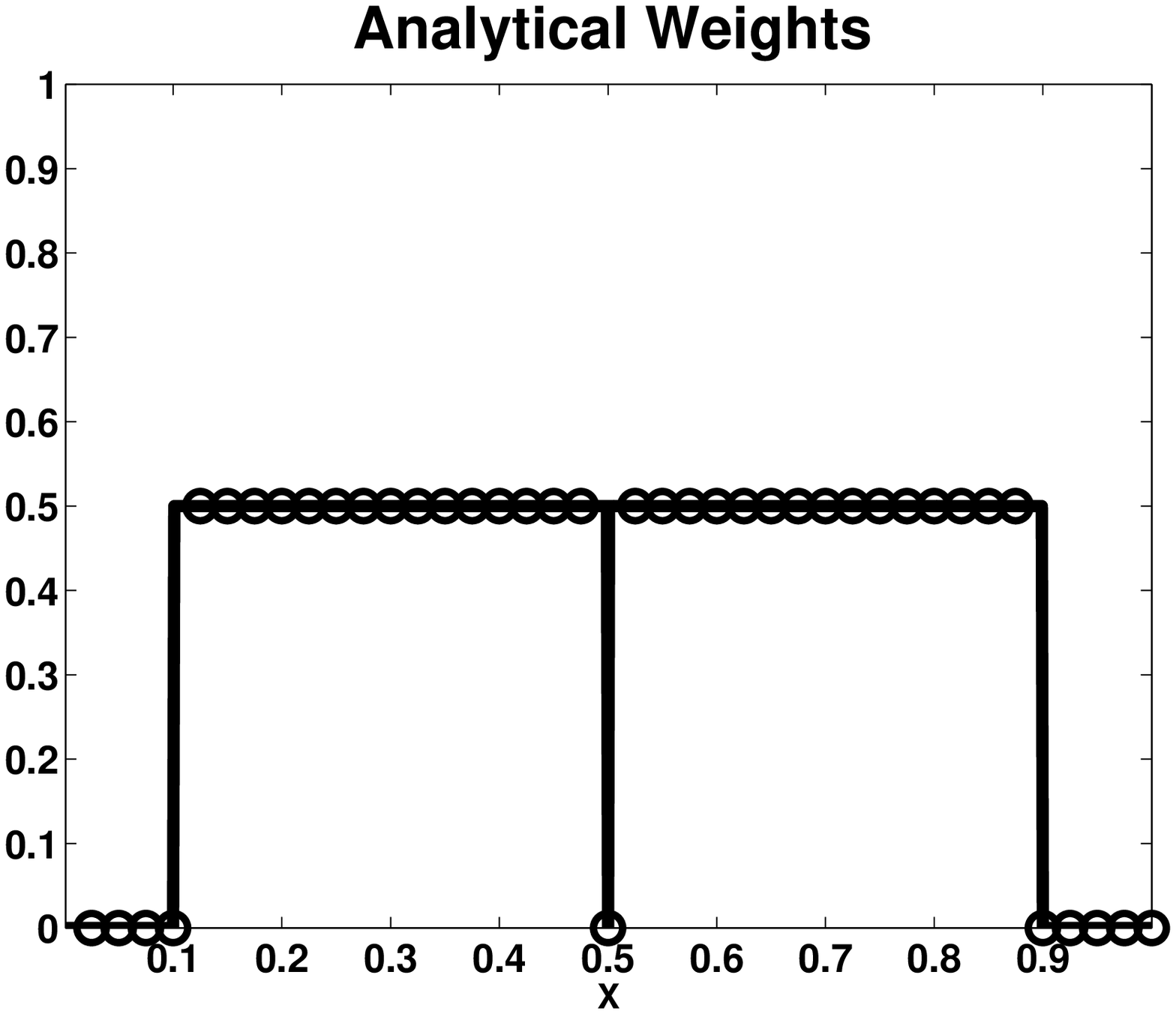}
	\includegraphics[width=0.3\textwidth]{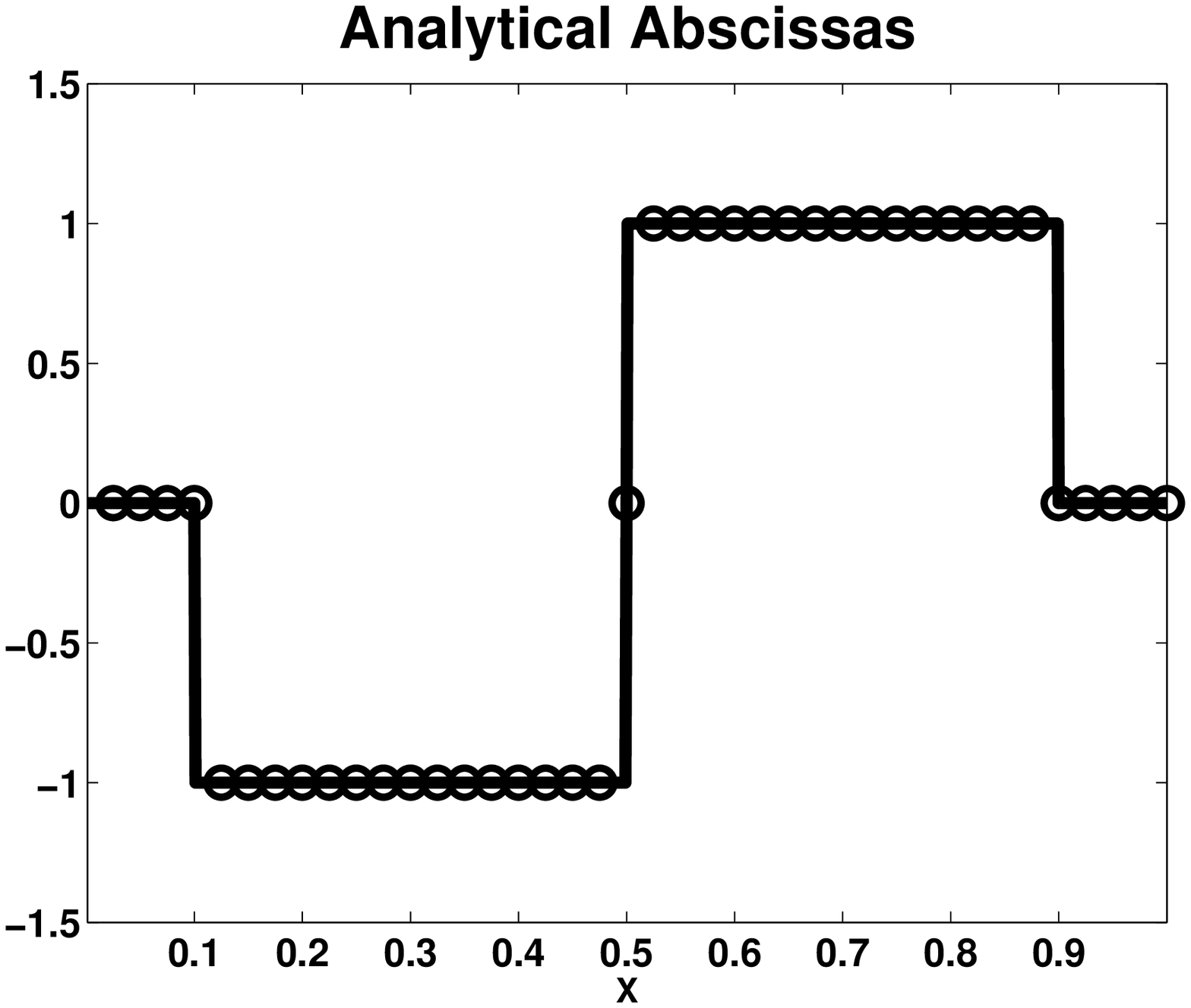}	
	
	\includegraphics[width=0.3\textwidth]{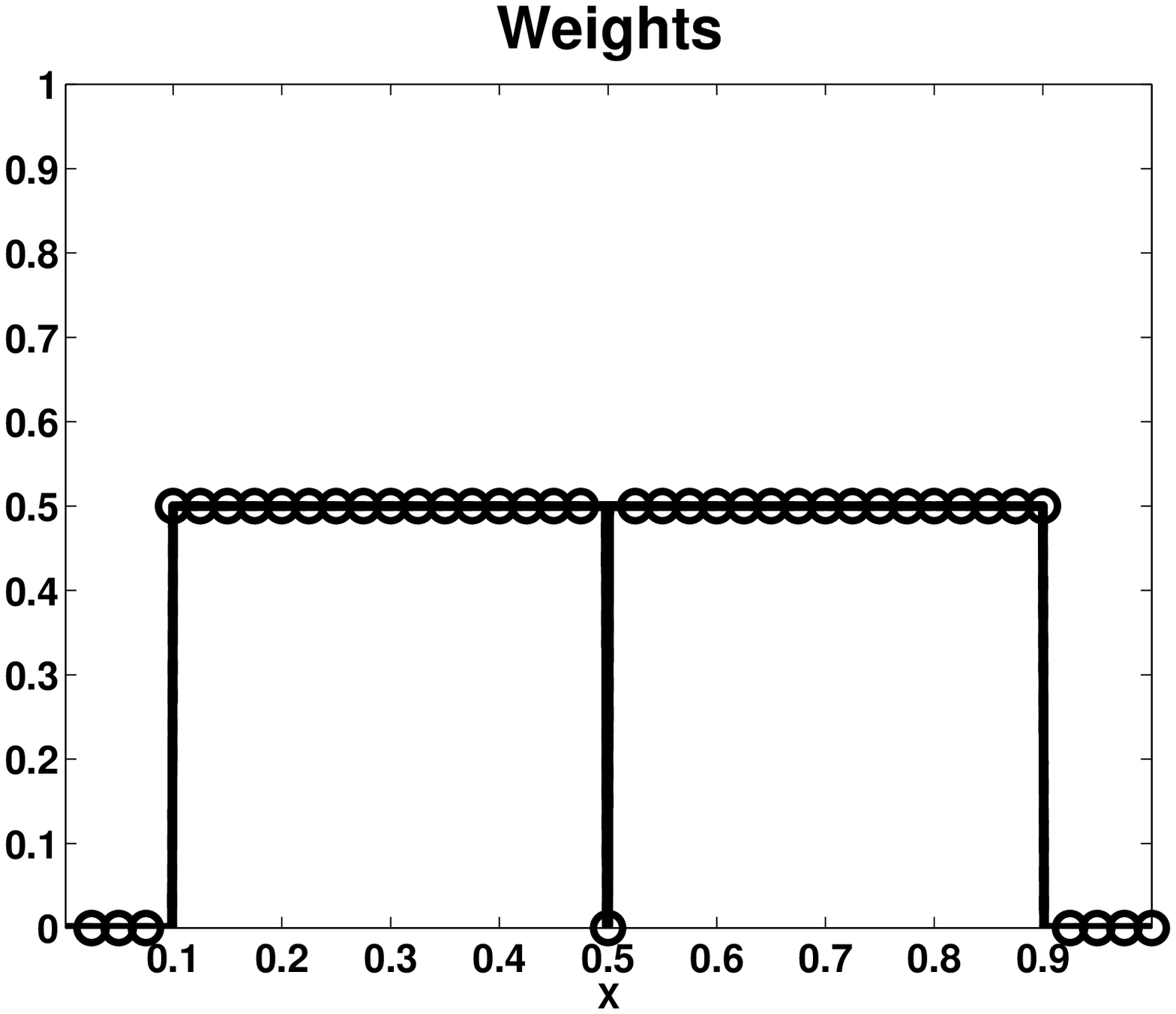}
	\includegraphics[width=0.3\textwidth]{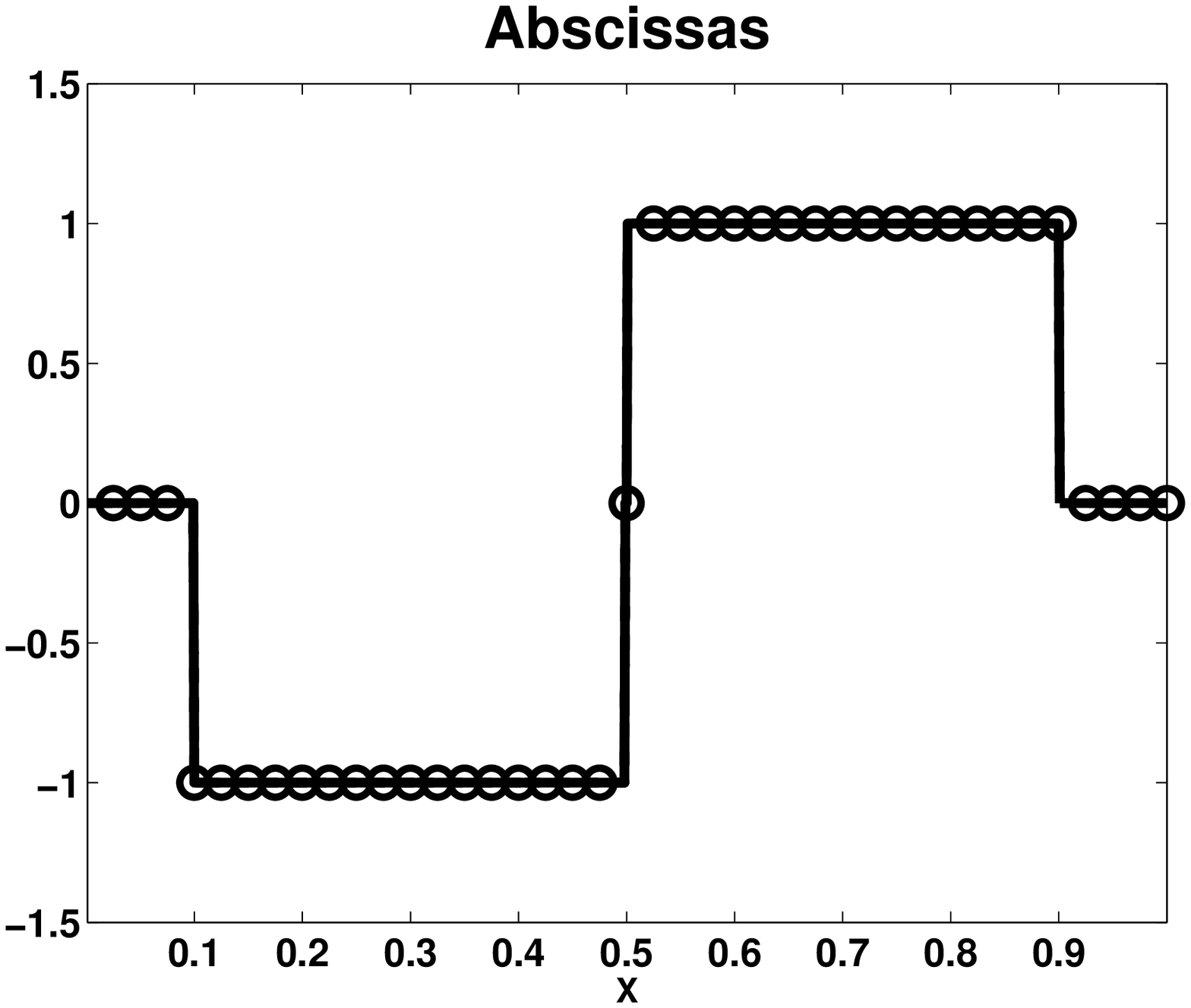}
	
		\caption{ Results for the two packet case at time $t=0.4$. 
Top: Analytical result for weights (left) and abscissas (right).
Bottom: Numerical results for weights (left) and abscissas (right).
 The solid line corresponds to the higher abscissa, the dashed line with circle to 
the lower one.}
		\label{fig:two_clouds_2} 
  \end{center}
\end{figure}

\subsubsection*{Four packet collision}

Figures \ref{fig:four_moments} presents the initial conditions. Figures\, \ref{fig:four_clouds}  and \ref{fig:four_clouds_1bis} present the numerical and analytical solutions 
respectively for the moments and the weights and abscissas.
The computation is run with a $1000$ cell grid on the spatial domain 
$[0,1]$, at $CFL=1$. Here, we observe the presence of two Dirac delta functions as already discussed in 
Section \ref{ees}.  The agreement between exact and numerical solutions for the moments is very good, showing that
the numerical solution converges to the analytical one.
The disparities encountered between the analytical and numerical solution in the case of the weights and abscissas 
are due to the fact that the mapping ${\bf U}({\bf M})$ is discontinuous at the moment space border. In area of 
numerical diffusion ($x\approx 0.22$ and $x\approx 0.88$), when the dispersion $e$ is under the threshold $\epsilon_1$,
weights and abscissas are reconstructed as explained in section $2$.

The wave propagating 
at velocity $1.2$ and separating the constant states ($v_1=0.8, v_2=0$) and ($v_1=1.2, \,v_2=0.8$) is 
steep and coincide with the analytical wave whereas the wave propagating at velocity $0.8$ is actually smooth since the CFL number is based on the highest value 
of velocity, which is $1.2$ in this studied case. The same explanation holds for the symmetric jump at location $x=0.78$.
Meanwhile, because of the conservation of the velocity 
moments, the numerical velocity jump (at $x=0.154$) happens before the analytical velocity jump (at $x=0.18$). 
One can here notice that the quadrature method provides the expected value of velocity in the numerical 
diffusion zones. The same explanation holds for the different velocity jump locations between the analytical 
and numerical solution at $x=0.82$ and $x=0.845$. The same phenomenon is responsible for the disparities 
between the analytical and numerical solutions at the $\delta$-shocks locations, at $x=0.4$ and $x=0.6$. 

Figure\, \ref{fig:four_clouds_2} displays the final profile of the quantities $\frac{q}{M_0e}$ and $\frac{q}{e^{3/2}}$.
Thus, $\frac{q}{M_0e}$
has significative values in areas where the abscissa distance as well as the weight difference are important, whereas $\frac{q}{e^{3/2}}$ reaches high value 
in areas where the weight ratio is important. 
Since in the domain, except for the singularities, the weights have the same value,
 both quantities are equal to zero. At the singularities, $\frac{q}{e^{3/2}}$ is roughly proportional to the square root of the weight ratio,
 and $\frac{q}{M_0e}$ is bounded, accounting for the fact that the velocity field is bounded.

 \begin{figure}[htbp]
  \begin{center}
         \includegraphics[width=0.27\textwidth]{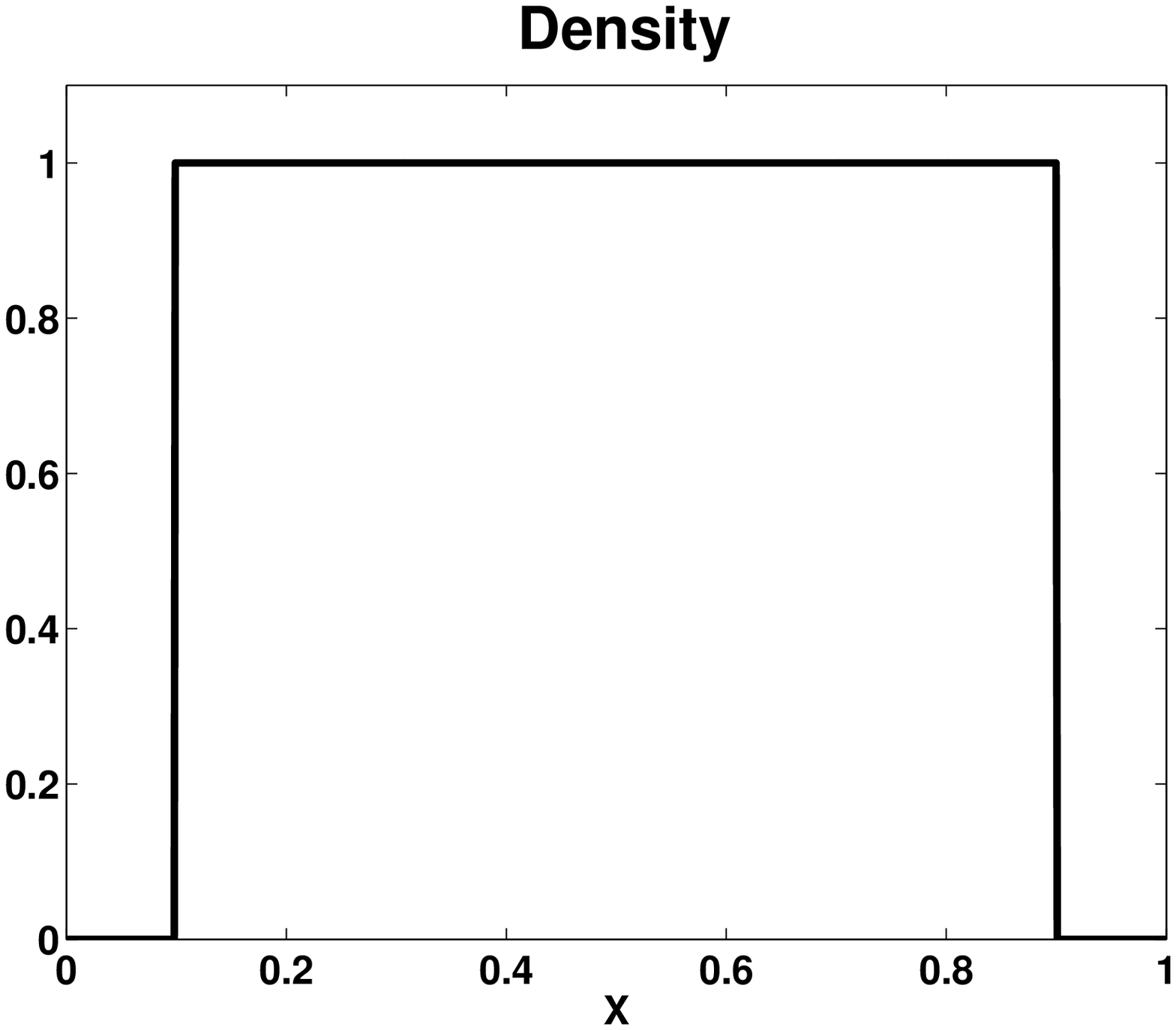}
	\includegraphics[width=0.27\textwidth]{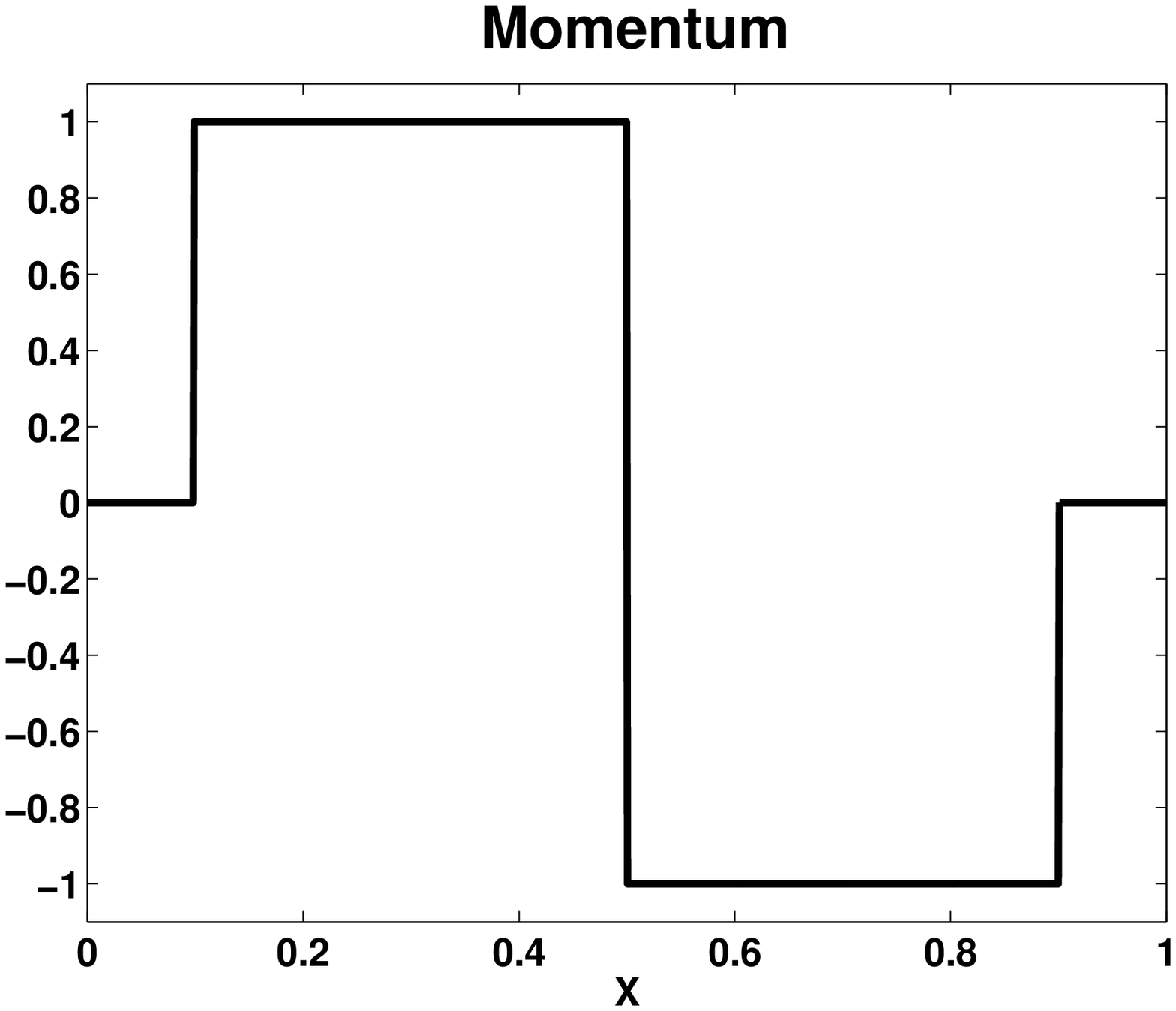}    \\
	
	 \includegraphics[width=0.27\textwidth]{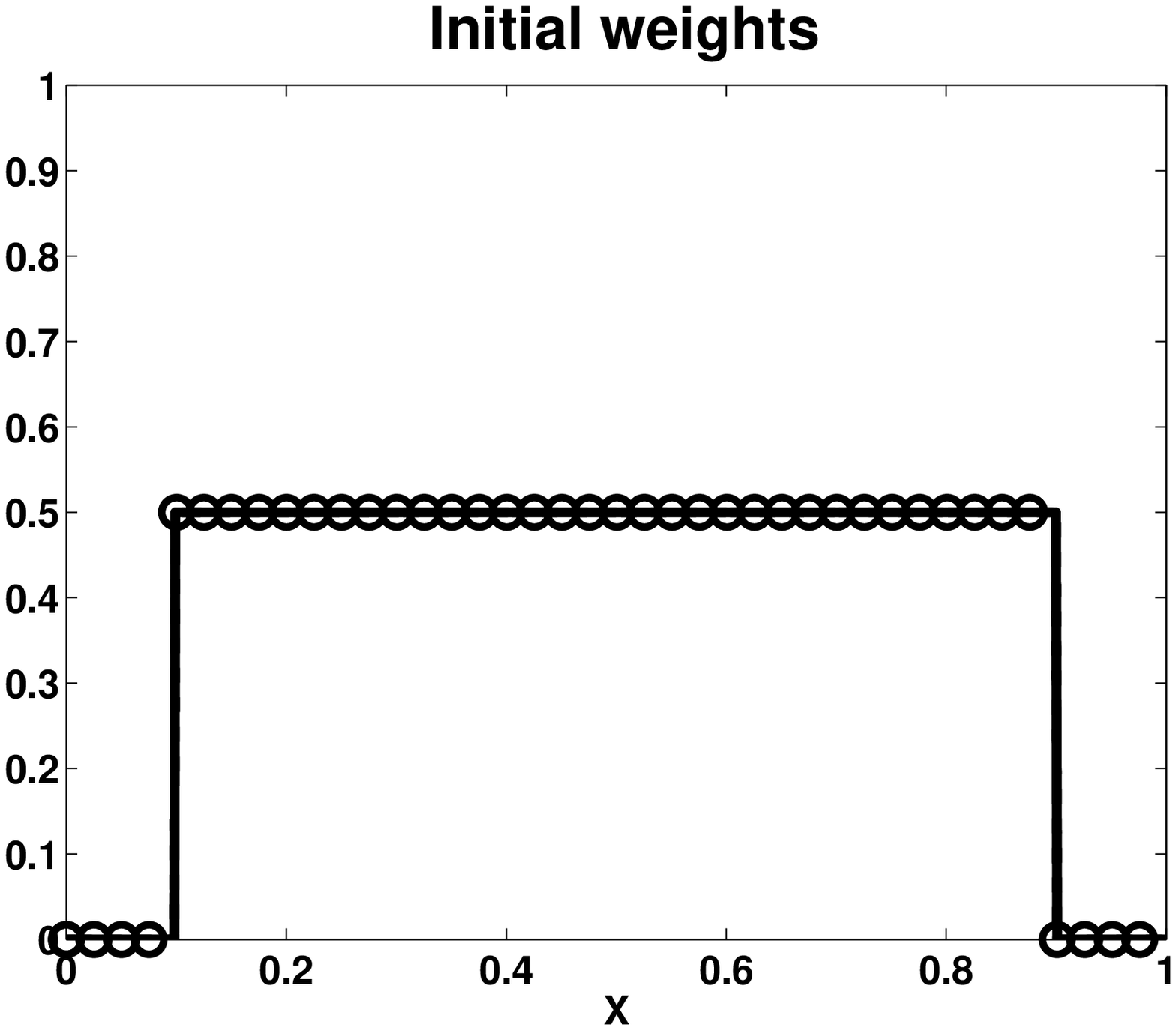}
	\includegraphics[width=0.27\textwidth]{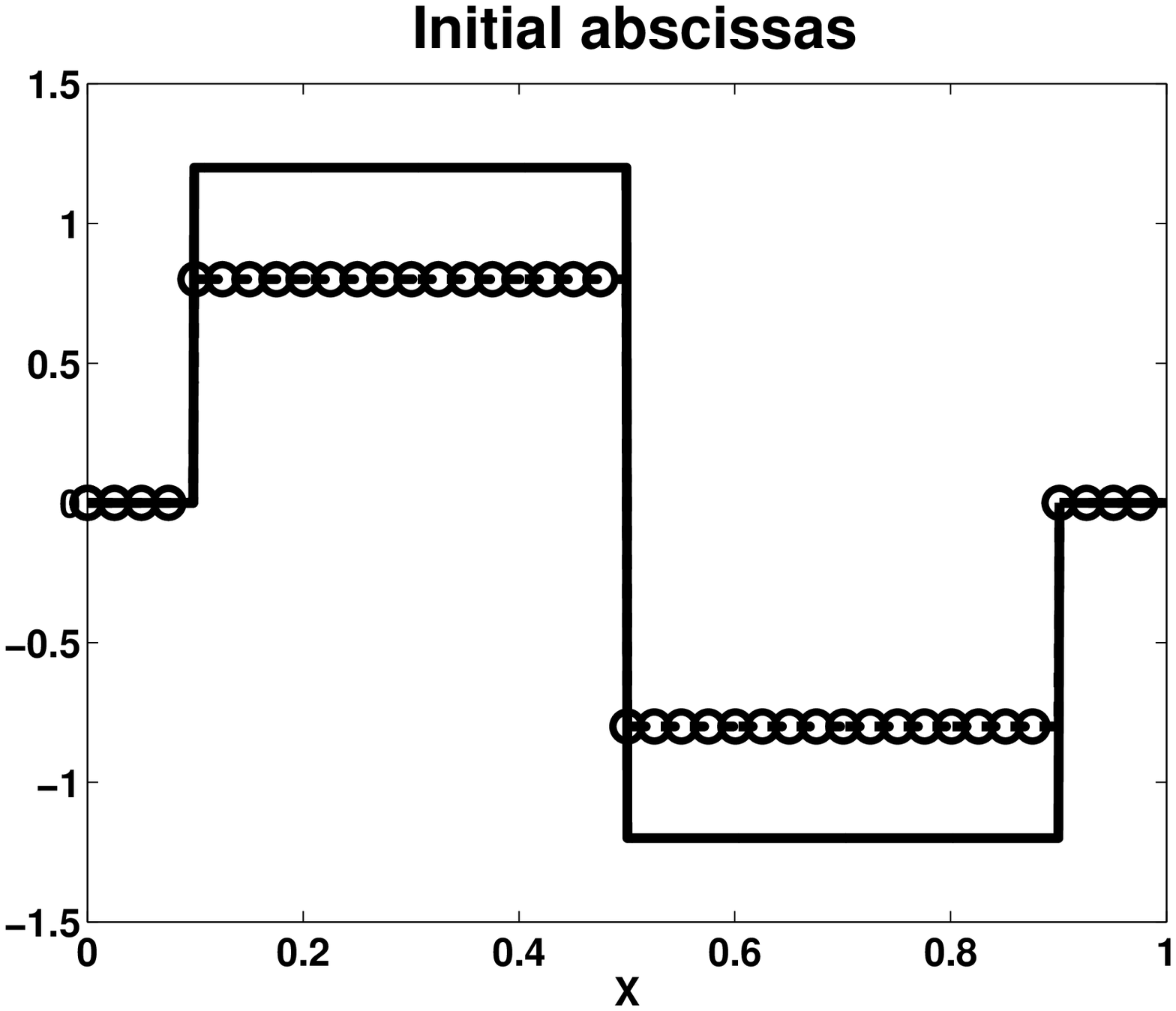}  
	
	\caption{Four packet case with $\rho_L=\rho_R=1$, and $v_1=1.2$, $v_2=0.8$. Initial conditions.
	Top:  $M_0$ (left)  and $M_1$ (right).
	Bottom:  weights (left) and abscissas (right).
	The solid line corresponds to the set ($\rho_1,v_1$), and the dashed 
line with circles to the set ($\rho_2,v_2$). }

		\label{fig:four_moments} 
  \end{center}
\end{figure}

 \begin{figure}[htbp]
  \begin{center}
        \includegraphics[width=0.33\textwidth]{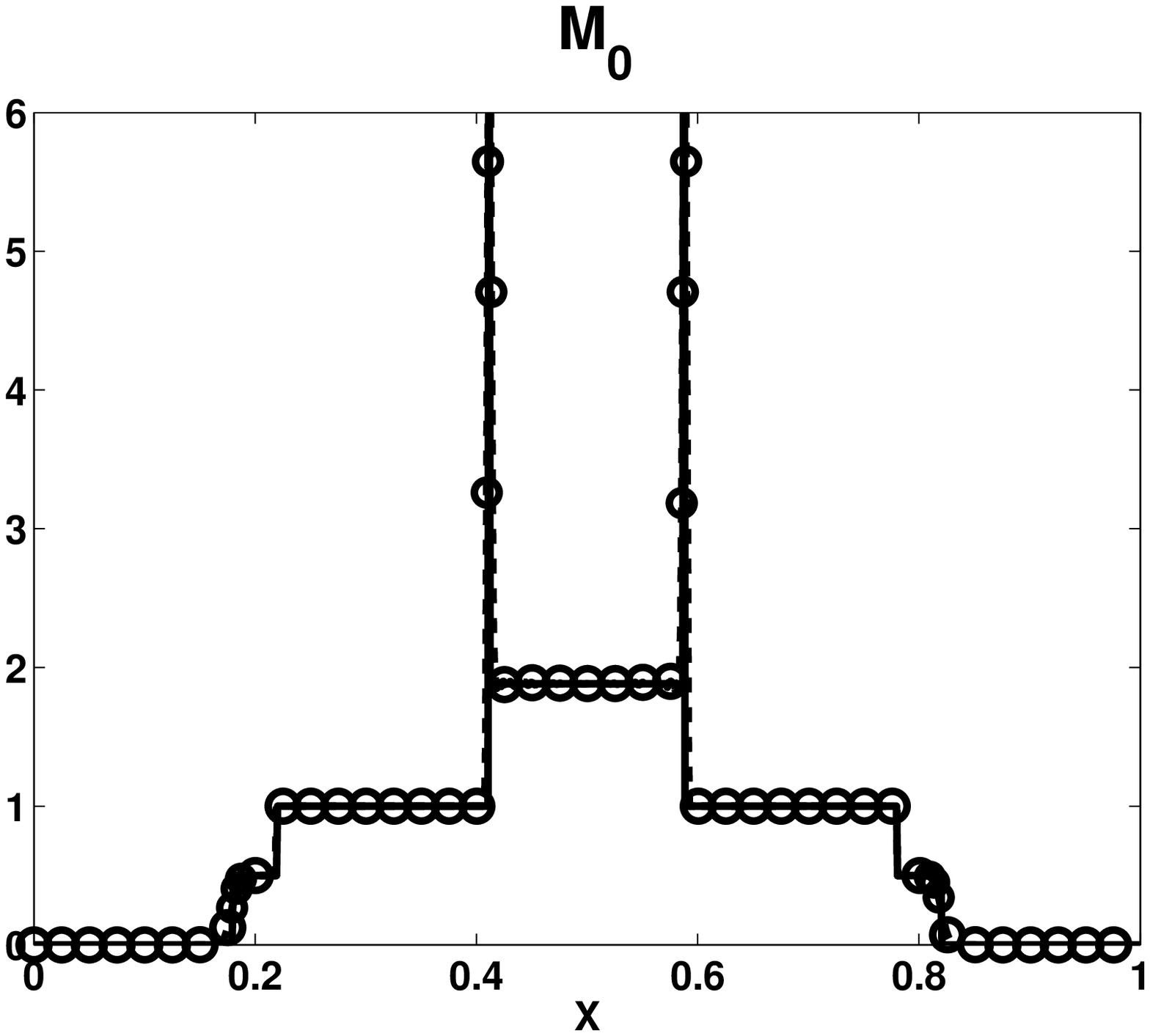}
	\includegraphics[width=0.33\textwidth]{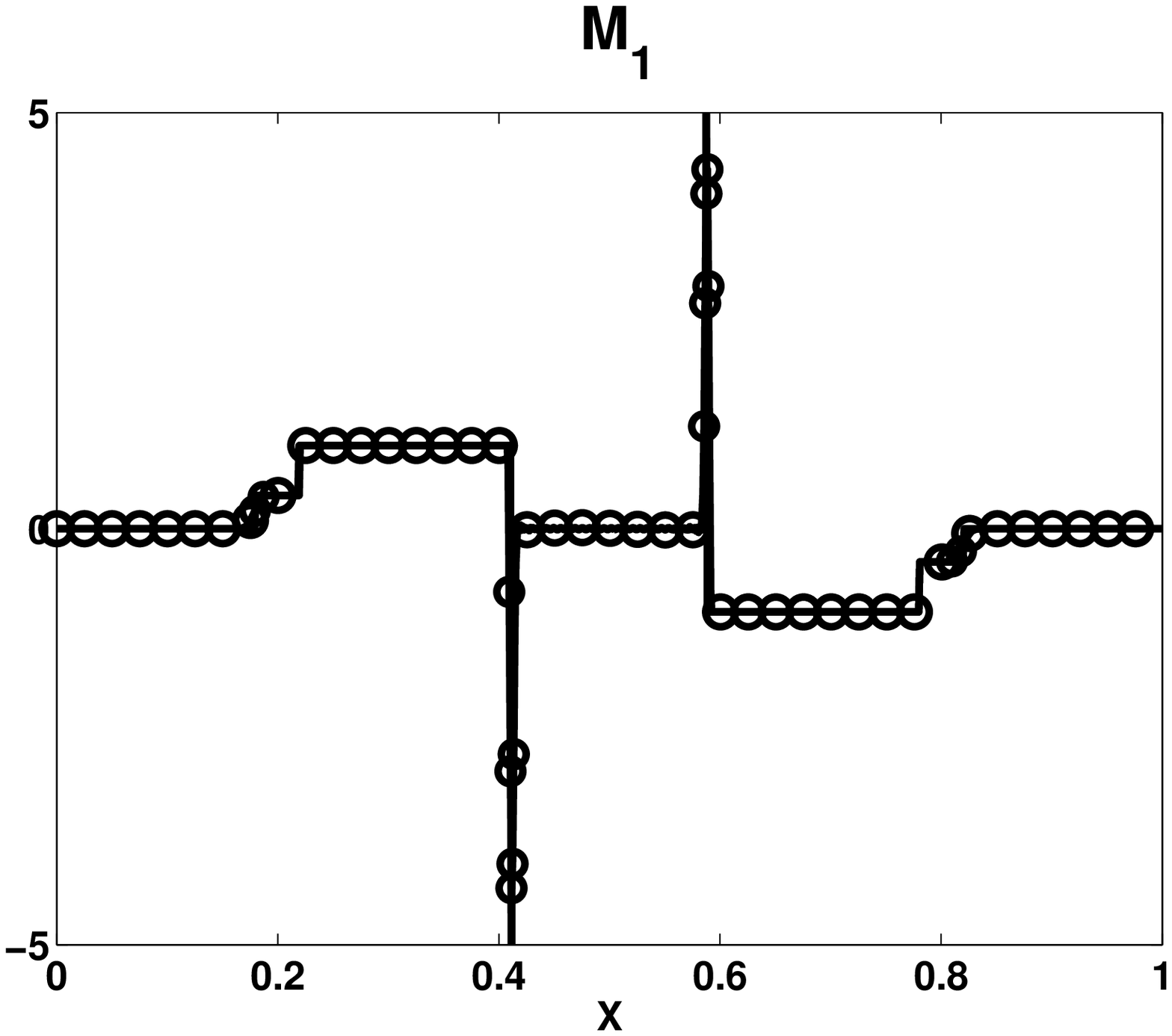}
	\includegraphics[width=0.33\textwidth]{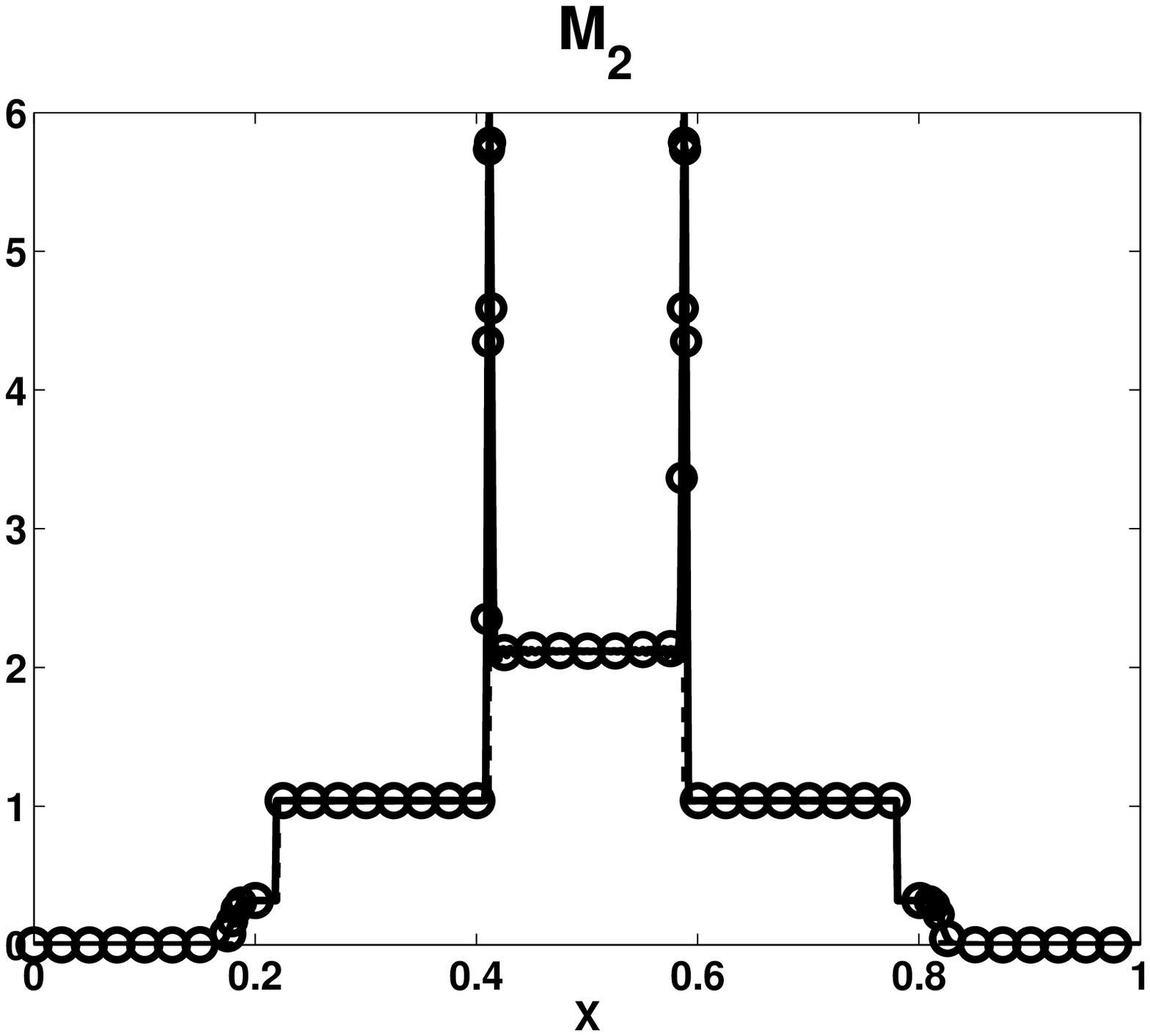}
	\includegraphics[width=0.33\textwidth]{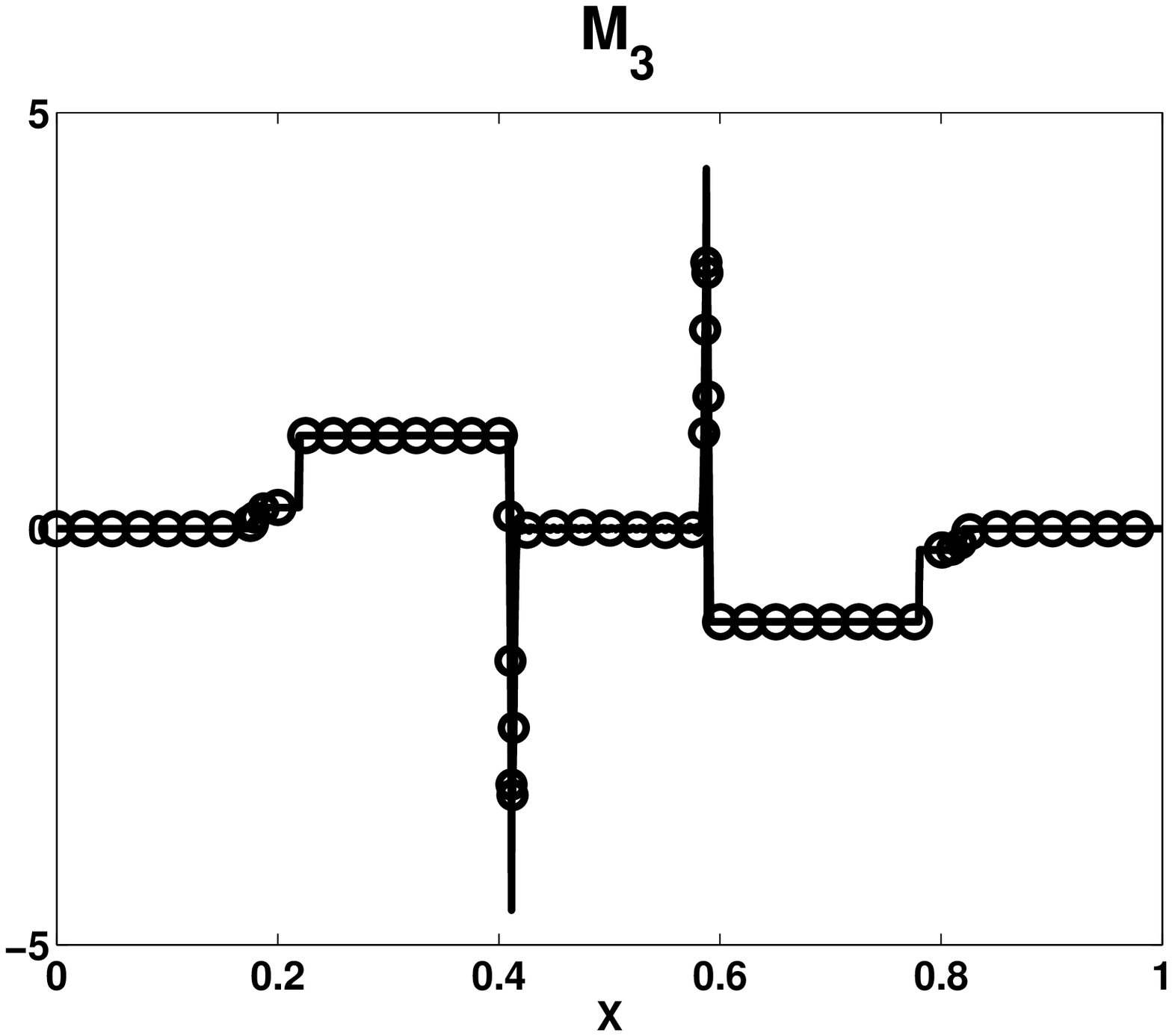}
	\caption{Four packet case. Results at $t=0.1$. Analytical  (solid line) and numerical (dashed line with circles) solutions.
	Top:  $M_0$ (left)  and $M_1$ (right).
	Bottom:  $M_2$ (left) and $M_3$ (right).}
		\label{fig:four_clouds} 
  \end{center}
\end{figure}

 \begin{figure}[htbp]
  \begin{center}
         \includegraphics[width=0.33\textwidth]{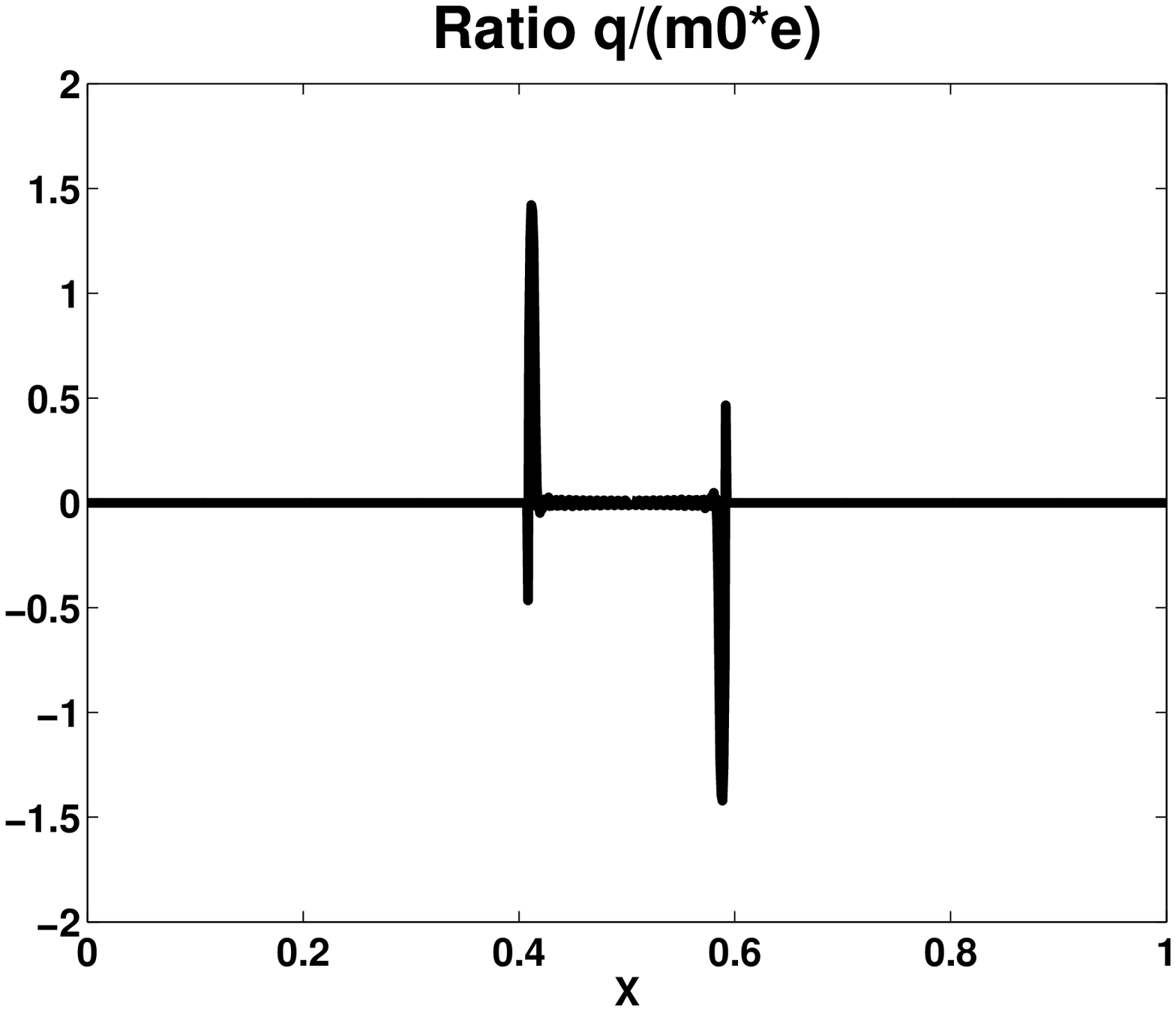}
          \includegraphics[width=0.33\textwidth]{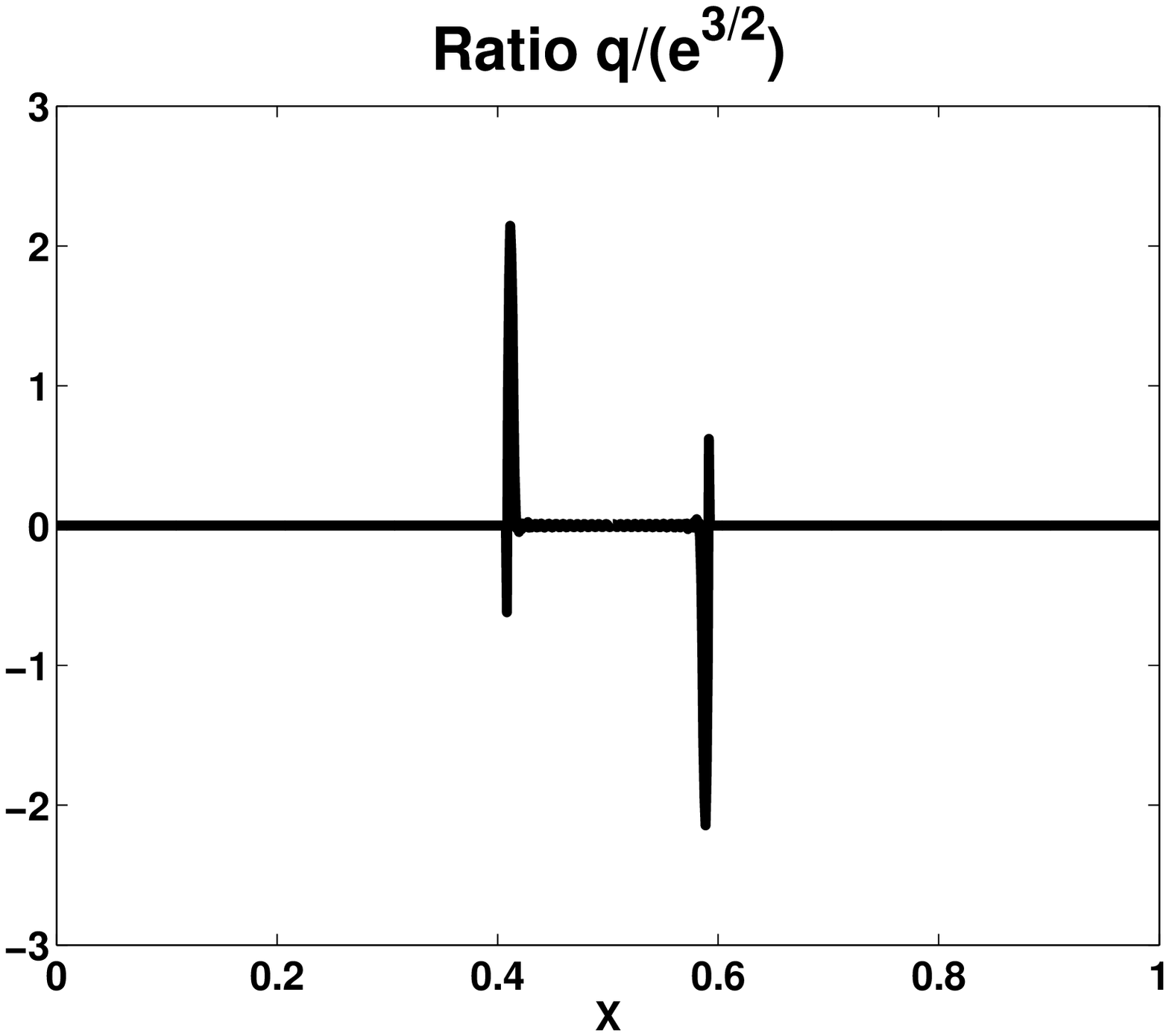}
	\caption{Four particle packet case, at time $t=0.1$.  Left: ratio $q/(M_0e)$. Right: ratio $q/e^{3/2}$. }
		\label{fig:four_clouds_2} 
  \end{center}
\end{figure}

 \begin{figure}[htbp]
  \begin{center}
	\includegraphics[width=0.33\textwidth]{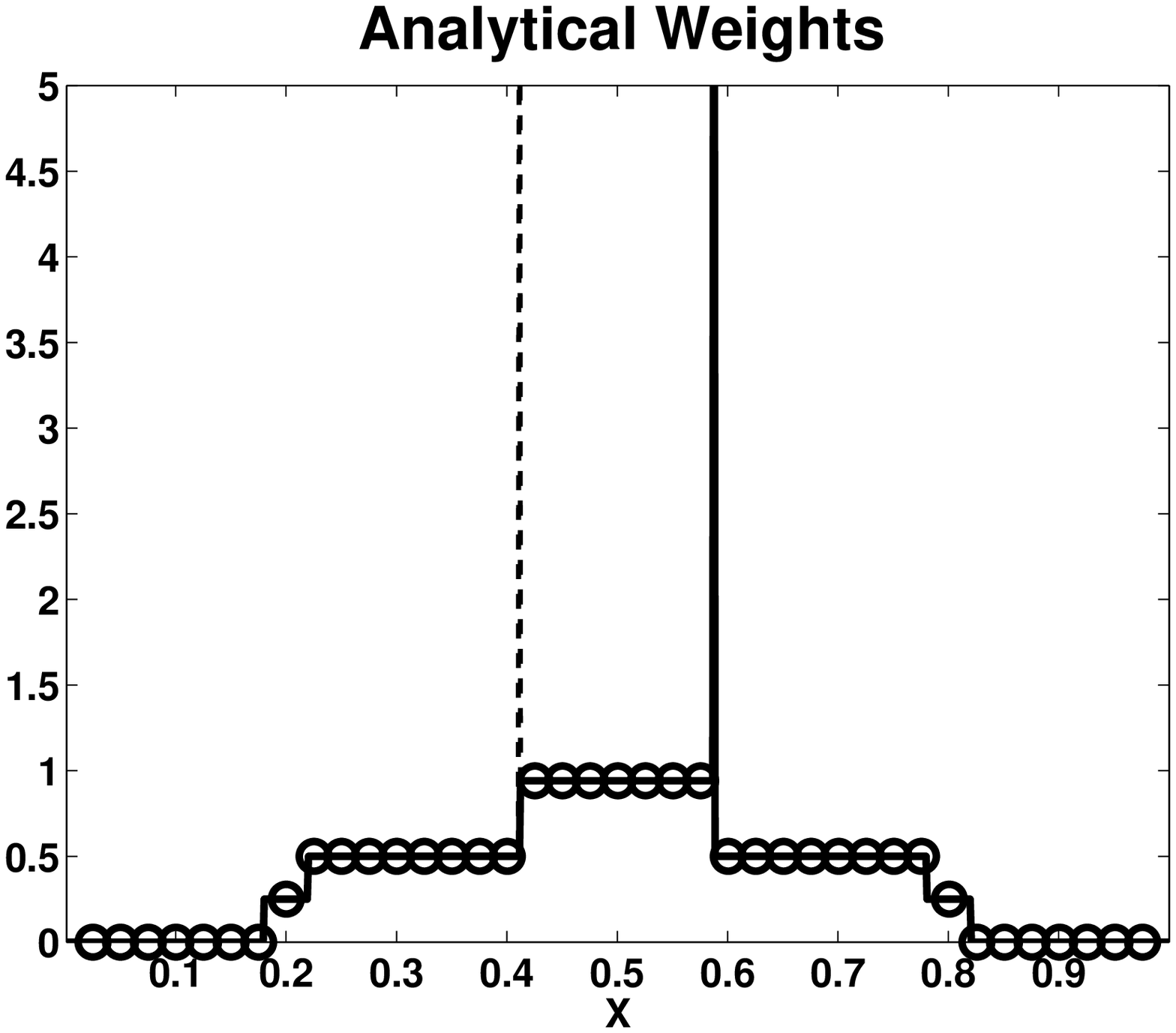}
	\includegraphics[width=0.33\textwidth]{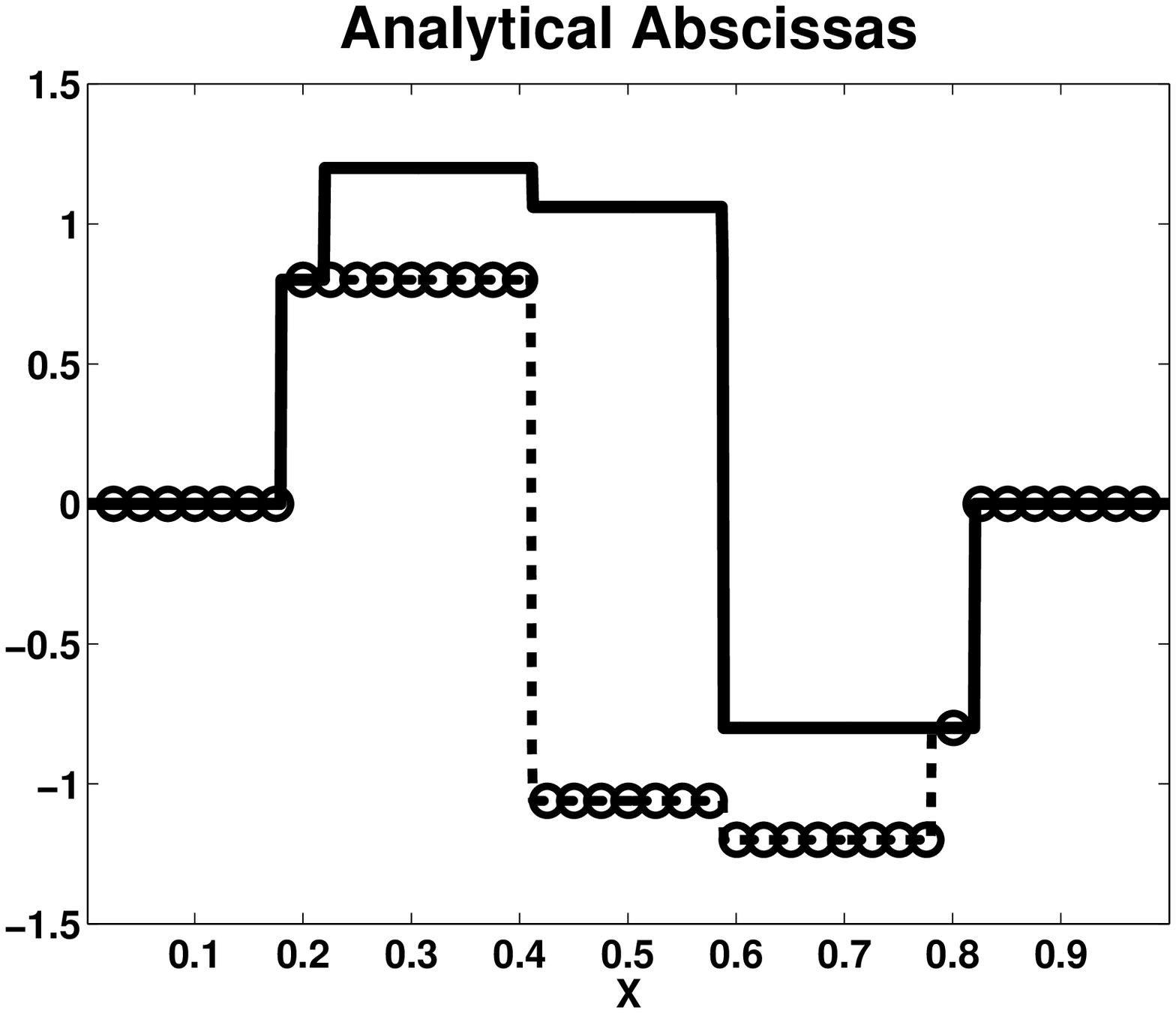}
	\includegraphics[width=0.33\textwidth]{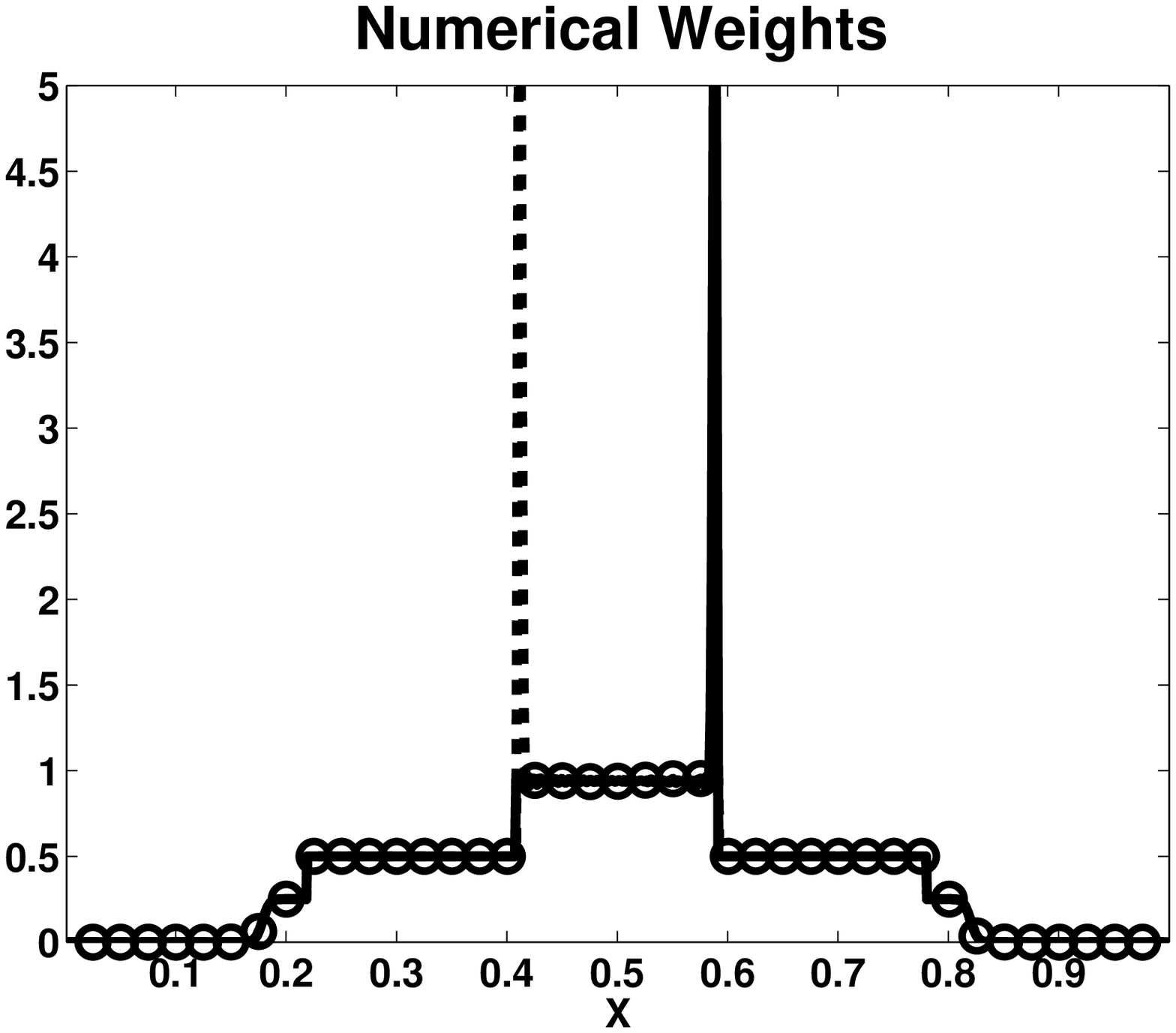}
	\includegraphics[width=0.33\textwidth]{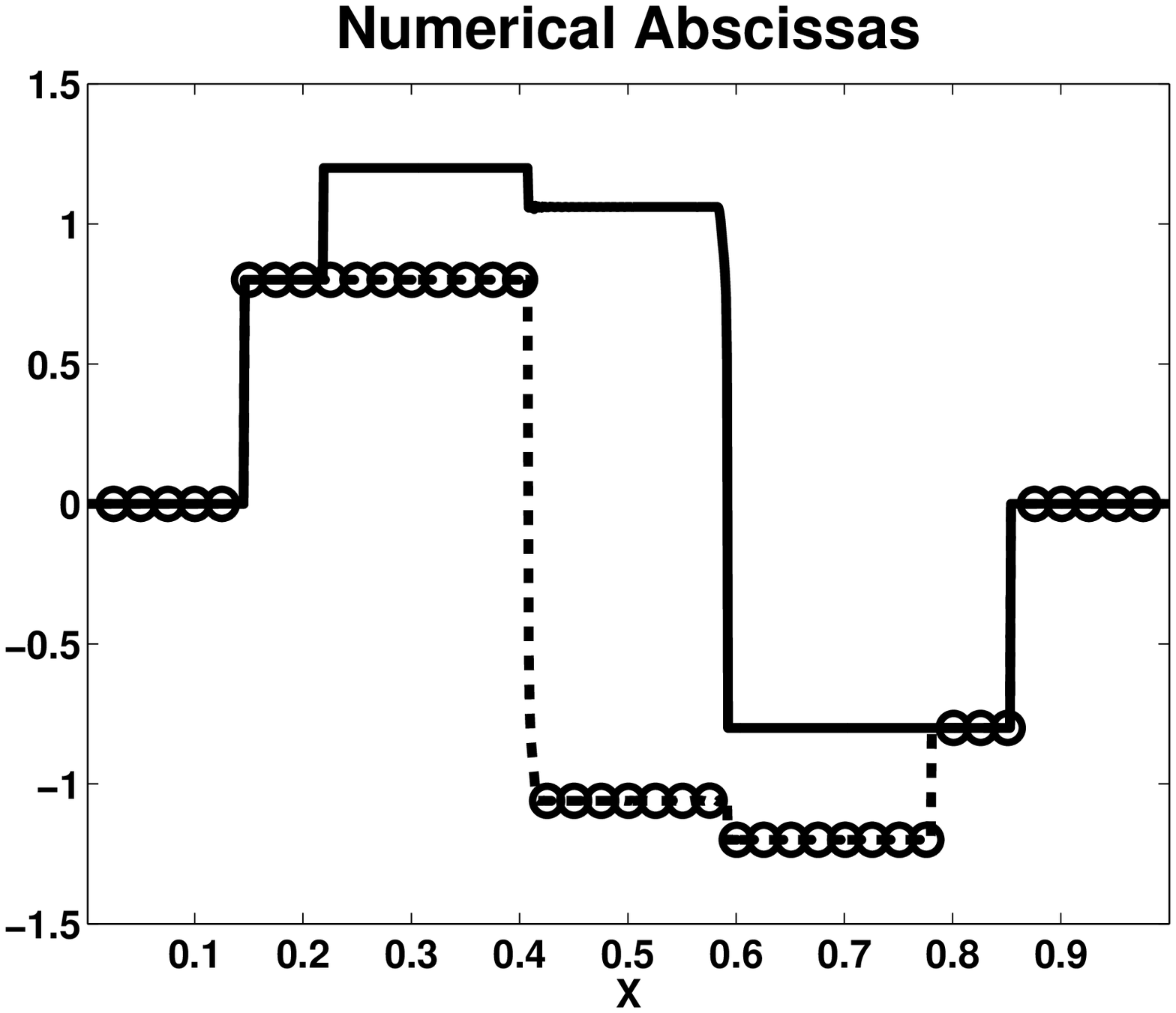}
		\caption{Four packet case. Results at $t=0.1$. 
	Top:  Analytical weights (left) and abscissas (right).
	Bottom:  Numerical weights (left) and abscissas (right).
	The solid line corresponds to the set ($\rho_1,v_1$), and the dashed 
line with circles to the set ($\rho_2,v_2$).} 
		\label{fig:four_clouds_1bis} 
  \end{center}
\end{figure}

We have thus provided numerical simulations in the two cases for which we have at our disposal an analytical entropic solution, either in the piecewise contant case, or in the singular case where $\delta$-shock mesure solutions are present. In the first case, the crossing of the two droplets monokinetic packets is very properly reproduced without numerical diffusion since we work at CFL one, even if this is not symptomatic of the numerical diffusion such methods will encounter with a first order method in realistic configurations \cite{kah10}.
In the second case, the numerical method is able to capture the creation of the measure singular solutions associated to the fact the we have limited the number of quadrature node to two. With this node number limitation, the proper physical solution, in the infinite Knudsen number limit, where the various droplet packets cross without interacting, differs from the entropic solution of the system of partial differential equations (\ref{eq:modele_bipic}) obtained through the quadrature-based closure. 
This is the same type of behavior as seen in the case of pressureless gas dynamics at a lower level. \\

\FloatBarrier


\subsubsection*{Free boundary case}

The last test case, explained in subsection\, \ref{trans},  assesses the ability of the method to solve free boundary cases 
connecting in a continuous manner states lying in the interior of the moment space and at its frontier.
%
The chosen grid contains $400$ cells, CFL number is set to $0.98$. This value of the CFL number is chosen
is order to prevent high frequency instabilities to occur. The computation is run until $t=0.2$.
The analytical solution has been provided in subsection\, \ref{trans}.

Results are displayed in Fig.\, \ref{fig:res_transition}.
 Let us first focus on the fronts present at $x=0.2$ and $x=0.9$ for the analytical solution.
  Since the CFL number is based on the highest velocity value ($2$ in this case) and is taken as $0.98$, the corresponding wave is less 
 diffused at $x=0.9$, contrary to the front wave located at $x=0.2$ moving at velocity $v_2=1$.
  Note that in these areas, $\rho_1=\rho_2$, since $e =0$ or $e < \epsilon_1$. The borders of these areas are clearly seen at 
 $x \approx 0.28$ and $x=0.8$.
 The  constant profiles for $\rho_2$ and $\rho_1$  observed respectively between $x=0.5$ and $x=0.7$ and between $x=0.5$ and $x=0.8$
 correspond to the one observed in the analytical solution.
In the area between $x=0.28$ and $x=0.5$, the density profiles would be expected to be constant, with the same value as before. Instead of that, 
one observes a peak value for $\rho_2$ and a low value for $\rho_1$, the sum $\rho_1+\rho_2$ being constant. 
This results from a coupling
between the behaviour of the quadrature method when $e$ tends to zero and the numerical diffusion. We are here 
in the situation $q > 0$ and $e$ small but $e>\epsilon_1$. See Lemma \ref{lemma23}.
Further away from the discontinuities the density values tend to their analytical value.


 \begin{figure}[h!]
  \begin{center}
         \includegraphics[width=0.33\textwidth]{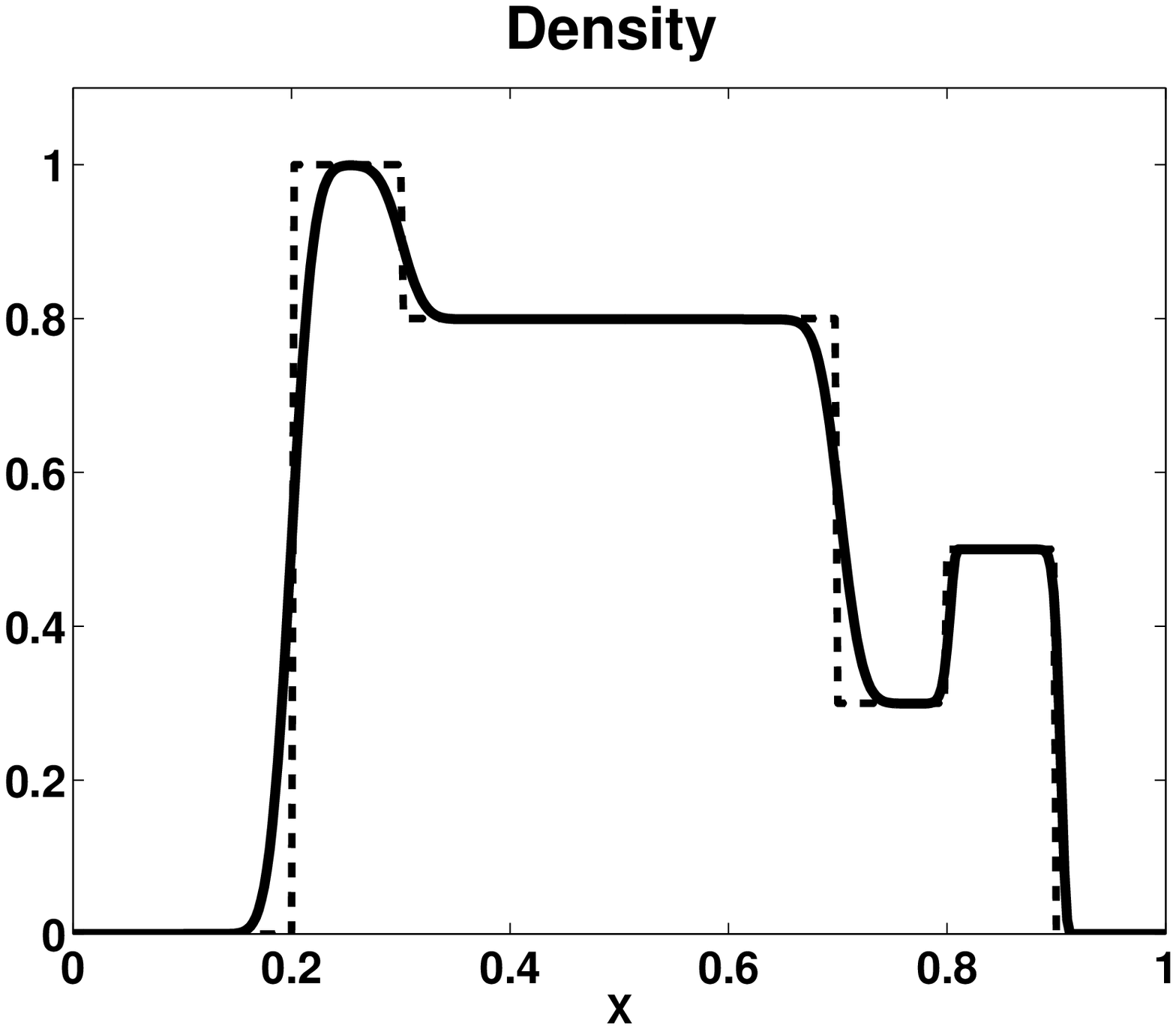}
          \includegraphics[width=0.33\textwidth]{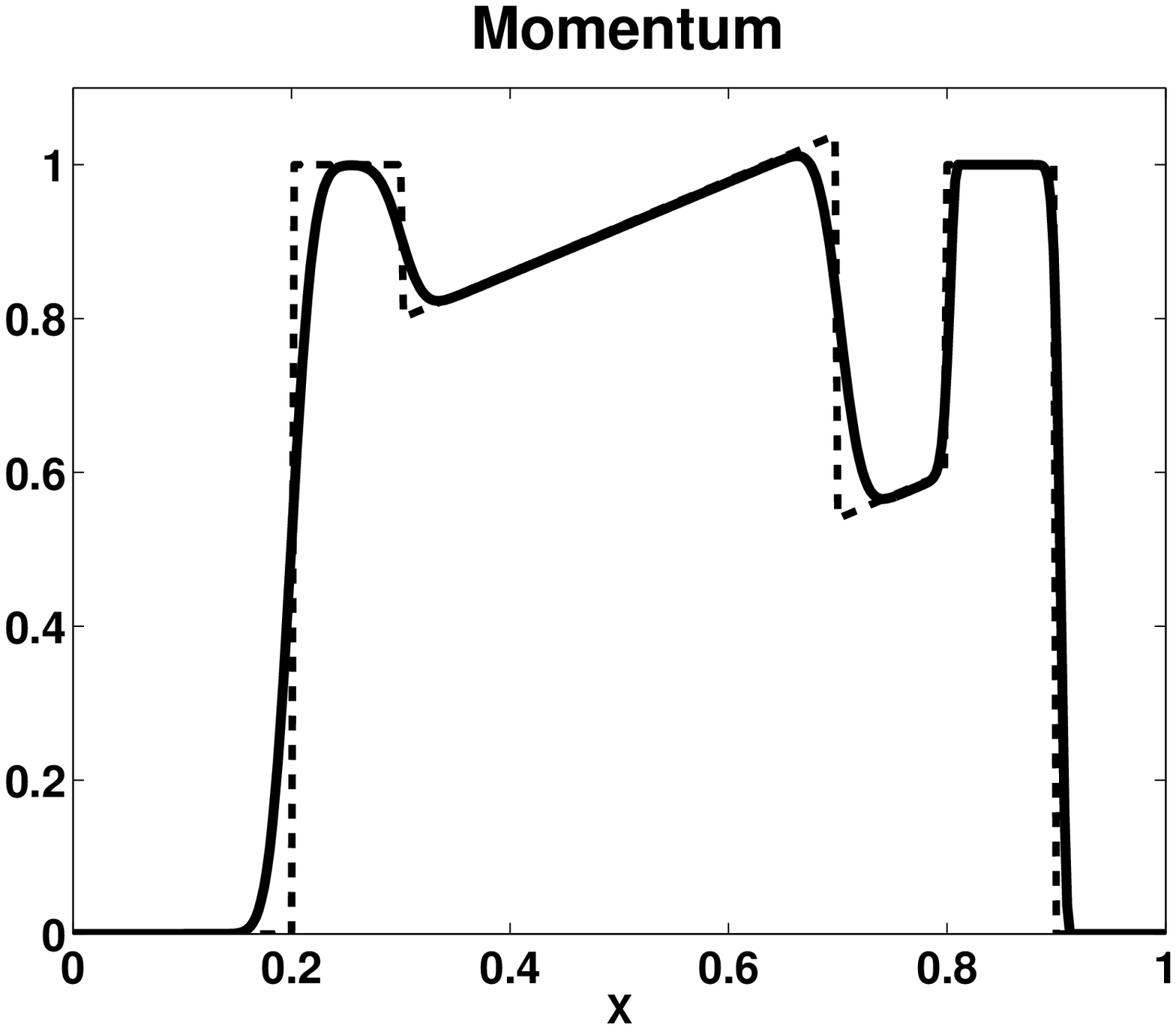}
              \includegraphics[width=0.33\textwidth]{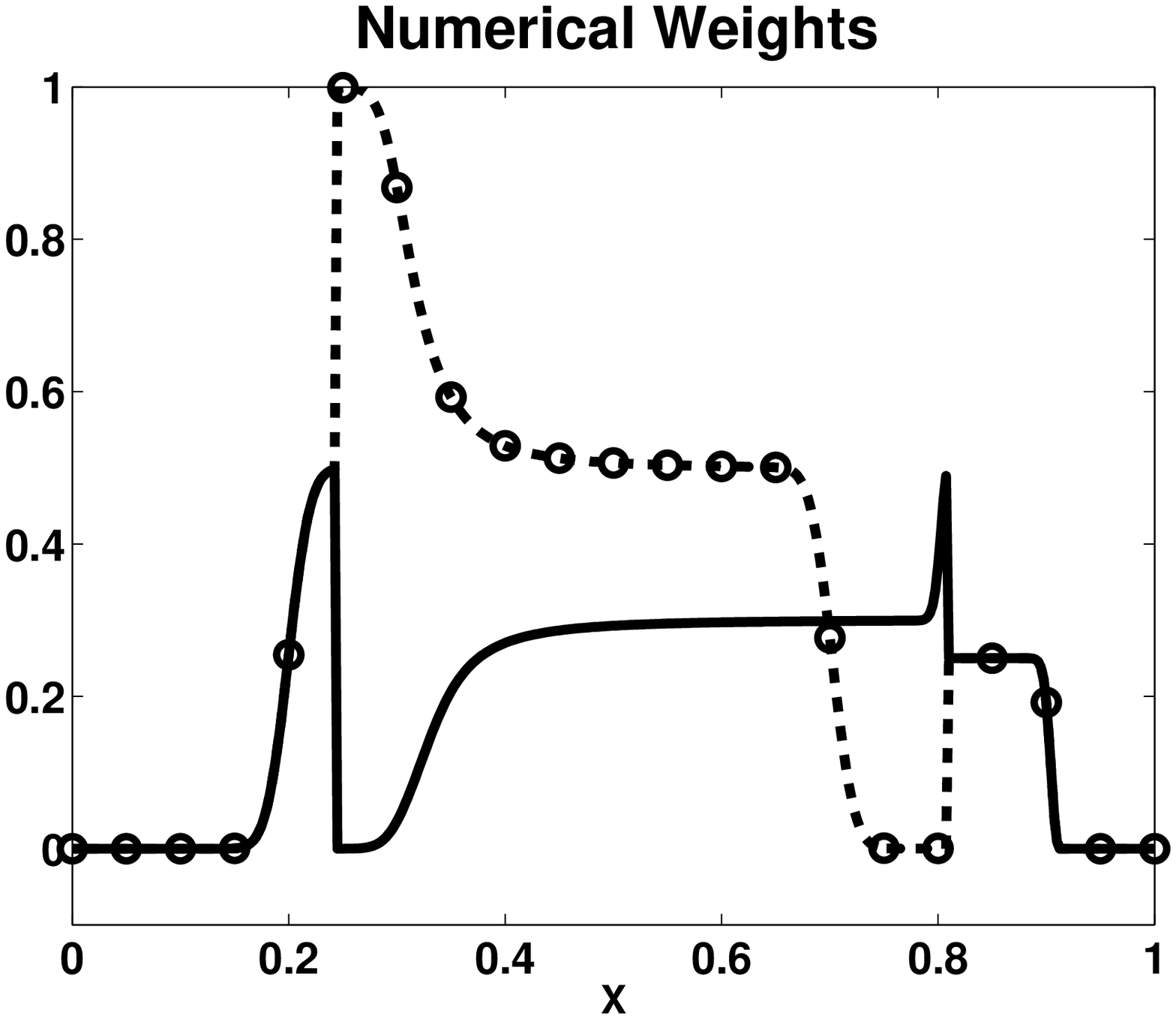}
            \includegraphics[width=0.33\textwidth]{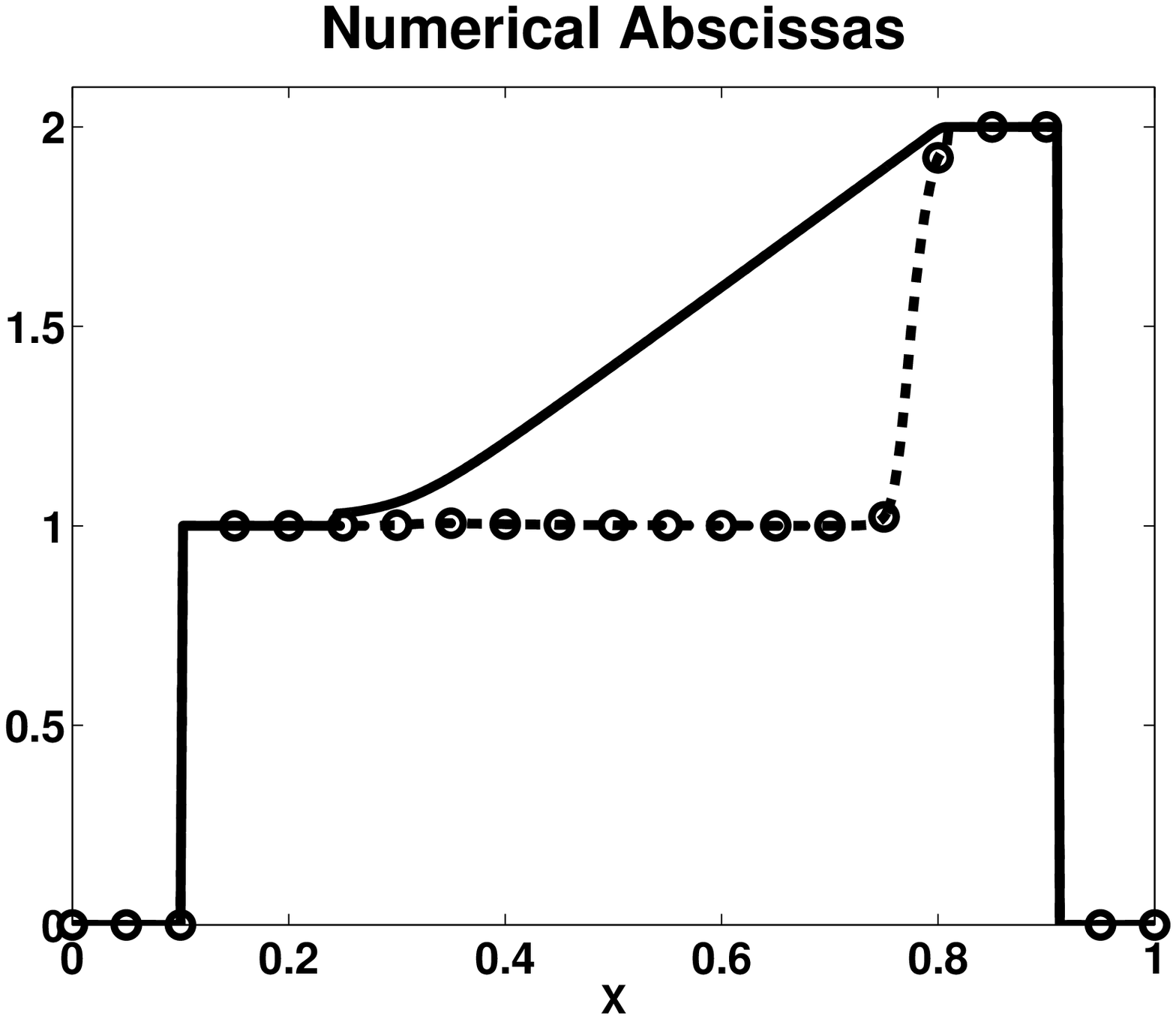}
          \caption{Moment dynamics for a free boundary connecting areas where $e=0$ and $e>0$. Results at $t=0.2$. 
          Top: Analytical (dashed line) and numerical (solid line) density (left) and momentum (right).
          Bottom: weights (left) and abscissas (right). The solid line corresponds to the set ($\rho_1,v_1$), and the dashed line with circles to the set ($\rho_2,v_2$). }
	
		\label{fig:res_transition} 
  \end{center}
\end{figure}

 \begin{figure}[h!]
  \begin{center}
         \includegraphics[width=0.33\textwidth]{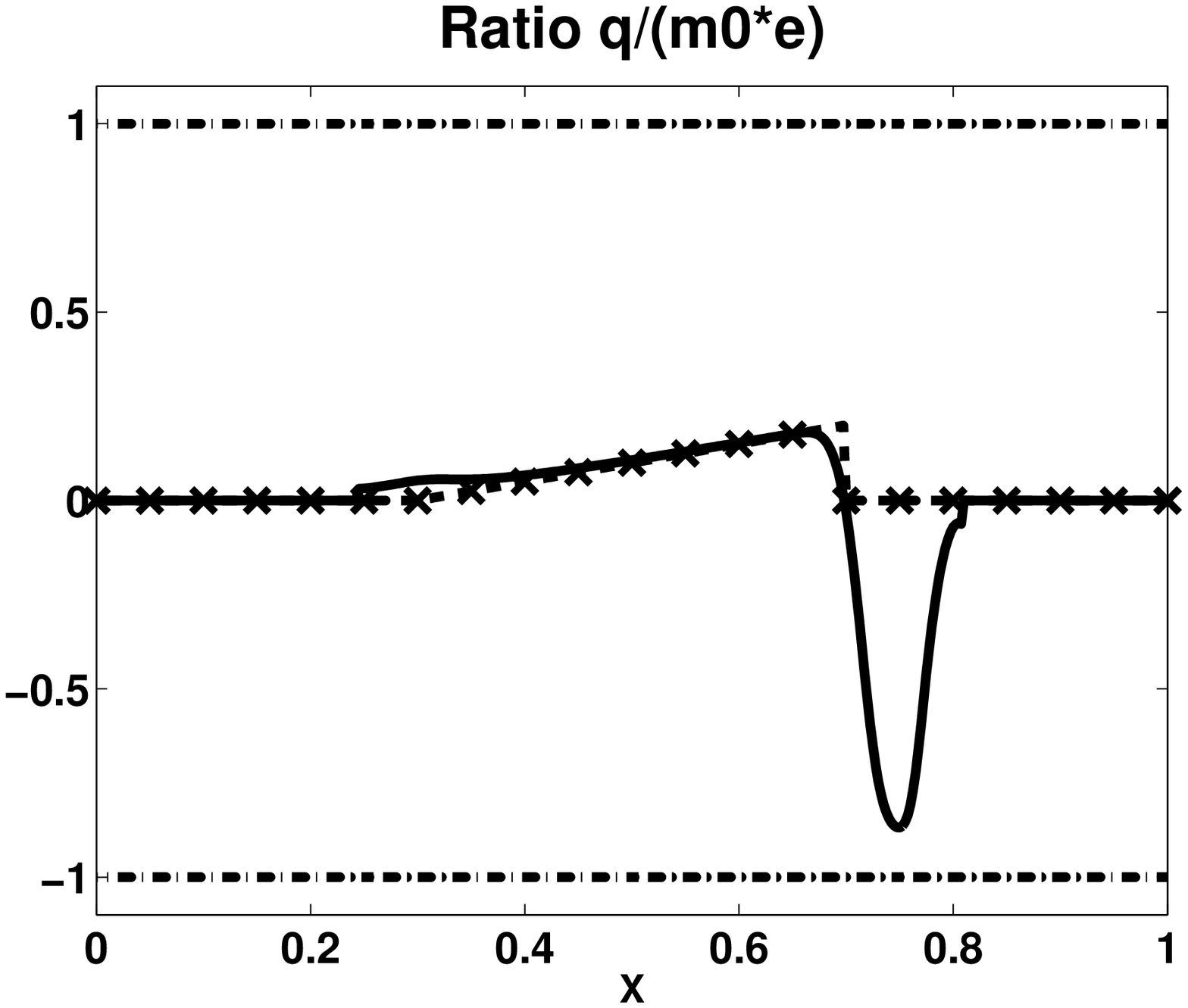}
          \includegraphics[width=0.33\textwidth]{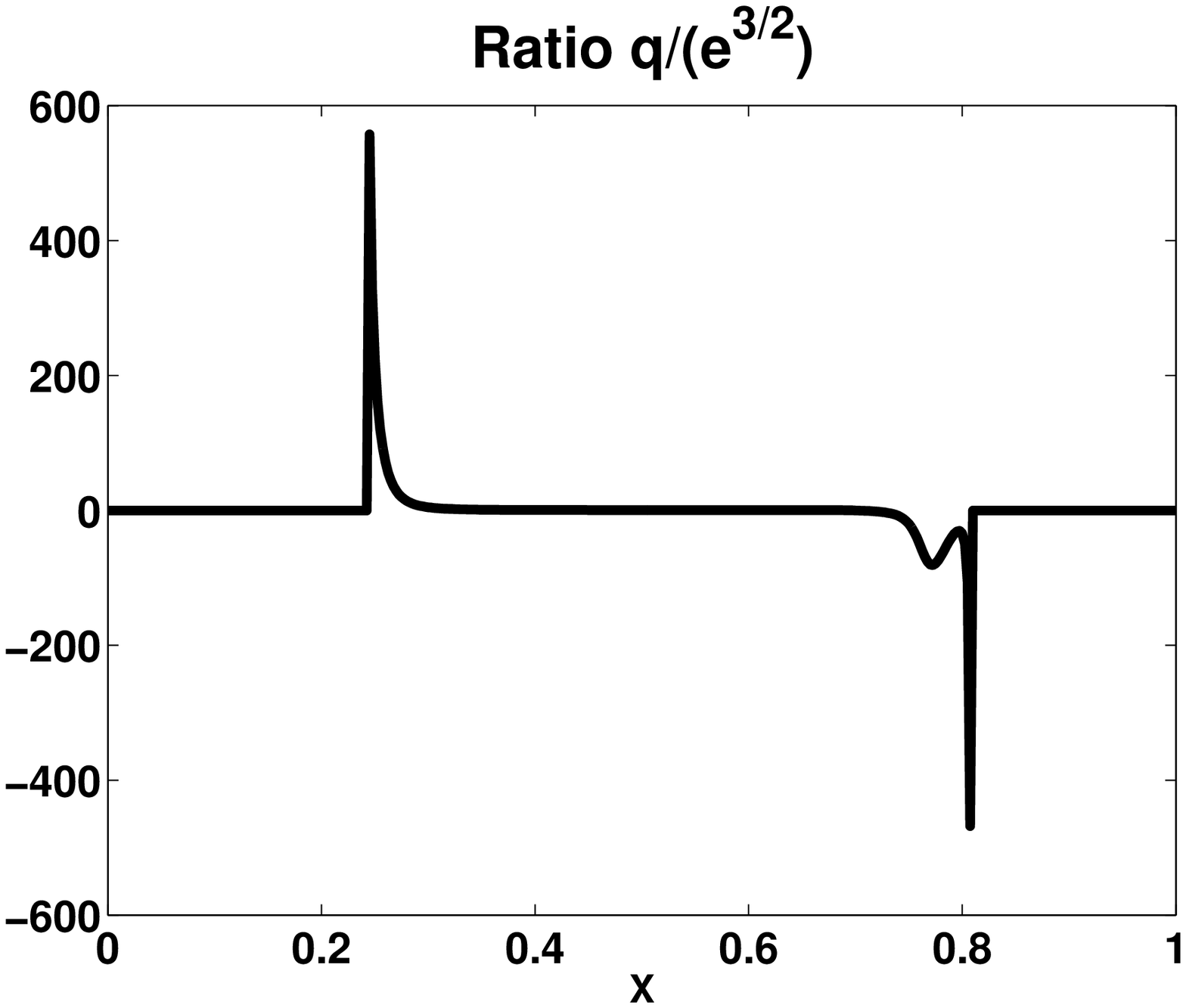}
           \caption{Moment dynamics for a free boundary connecting areas where $e=0$ and $e>0$. Results at $t=0.2$. Left: ratio $q/(M_0e)$.  The dashed line with crosses represents the analytical solution, whereas the solid line represents the numerical solution, bounded by its extremal initial values (dotted-dashed curve). Right: ratio $q/e^{3/2}$. }
		\label{fig:res_transition_2} 
  \end{center}
\end{figure}

Figure \ref{fig:res_transition_2} displays the profile of quantities $\frac{q}{M_0e}$ and $\frac{q}{e^{3/2}}$ at time $t=0.2$. 
First, it is interesting to note that   $\frac{q}{M_0e}$ 
 is naturally bounded by $1$ in the numerical simulation, as shown in Fig\, \ref{fig:res_transition_2} such that the bound imposed with $\eta = 2$ in never effective. Let us underline that in the present case, the numerical scheme allows to preserve the proposed cone associated with a maximal distance between the abscissas which is also invariant. 
Comparing the analytical value of $\vert q/(M_0\,e)\vert$ to the numerical resolution in Fig.  \ref{fig:res_transition_2}, it is instructive to observe that the numerical diffusion is creating zones where it can be far from zero, whereas it should be zero in the analytical solution. However, this is done is such a way as to preserve the maximal value foreseen as the maximal distance between the abscissa at time $t=0$ which is one. 
Let us also underline that we have rerun this case with various values of $\epsilon_1$ varying from $10^{-9}$ up to  $10^{-7}$ without any  effect on the solution. It can also be noticed that the quantity $q/e^{3/2}$ has large values in the regions of connection between the interior and the frontier of the moment space, but has no reason to be naturally limited as opposed to the previous quantity as it can be seen in  Fig.\,\ref{fig:res_transition}-bottom left. 


Concerning the convergence behavior of the numerical solution, we display the error in $L_1$ norm relative to the analytical solution to assess convergence quantitatively, Tab~\ref{errorL1}, for $400$, $800$, $1600$, and $3200$ cell grids. As expected, we get an experimental order of convergence of $0.5$. These data clearly show the convergence towards the analytical solution for each moment. 

As a consequence, it can be seen that we have designed the proper theoretical setting for the transition from the interior of the moment space towards its frontier since the cone we have defined seems to be automatically preserved by the kinetic scheme we have used. Such a point would be worth a detailed study which is beyond the scope of the present paper.

\begin{table}[h!]
\centering
\begin{tabular} {|c@{\hspace{1.5cm}}|c@{\hspace{1.5cm}}c@{\hspace{1.5cm}}c@{\hspace{1.5cm}}c|}
\hline
			Grid size			&	$400$      &     $800$     &    $1600$ &   $3200$   \\ \hline
			$m_0$                         &        0.049 &     0.0344  &   0.0244 &   0.0172     \\ \hline
			$m_1$   			&        0.0442 &     0.0307 &   0.0219 &   0.0153  	   \\ \hline
			$m_2$			&       0.0396  &     0.0271 &    0.0195 &  0.0136     \\ \hline
			$m_3$			&     0.0362    &	    0.0244	& 0.0177	&  0.0122      \\ \hline
							
\end{tabular}
\caption{$L_1$ error on moments relative to the analytical solution.}
\label{errorL1}
\end{table}




\FloatBarrier

\section{Conclusion}

In this paper, we have extended the notion of entropic measure solution of quadrature-based moment method for kinetic equations. Such kinetic equations are frequently encountered in many application fields where a complex dynamics in phase space is involved. 
Following the contribution of \cite{Bouchut03} for the pressureless gas dynamics which is the one-node quadrature version of a more general system of conservation laws for quadrature-based moment models, we have been able to provide a few problem test-cases showing that the numerical solution of the resulting system of conservation laws through kinetic schemes reproduces the defined entropic solution as well as the proper theoretical setting for the transition from the interior to the frontier of the moment space. It is an important point for the case of PTC where the solution remains smooth and where the scheme allows to describe the phase space dynamics properly as well as for cases where the complexity of the dynamics in phase space leads to generalized $\delta$-shocks, as observed for pressureles gas dynamics due to the weakly hyperbolic structure of the system of conservation laws. Two stumbling blocks still remains to be treated. First, we would need a uniqueness theory and a convergence analysis in a general framework in order to fully justify the use of the kinetic schemes for the simulation of such models. However, as explained already in \cite{Bouchut94}, the framework of entropic solution is not sufficient in order to provide uniqueness since one can exhibit multiple entropic solutions for measure solutions. 
Let us underline that it is easy to construct the same type of measure solutions for system (\ref{eq:modele_bipic_short}), which is the exact same collision case used by Bouchut, but with a motionless Dirac delta function in density localized at the collision point of the other two incoming ``particles".  An infinite set of entropic solution can then be exhibited depending on the nature of the collision. As a consequence, it would be first useful to investigate such a point on the pressureless gas dynamics and then to extend it to the present system of higher order quadrature-based moment models. 
Besides, the construction of fully high order methods is still an open question and requires further developments. \\
At last, let us mention that most of the results of the present paper do naturally extend to higher order moment systems, but at the 
price of algebra complications. In fact, the key point 
lies in the extension of the proposed study of the behavior of the quadrature at the frontier of the moment 
space, namely when the two velocities $v_1$ and $v_2$ become equal. If we consider for instance the $6$-moment model 
and assume that one of the three velocities $v_1$, $v_2$ and $v_3$ is smooth while the other two become equal, 
we are in the same framework as in the present paper. But the case when the three velocities tend to be equal 
needs to be precised in a future work.









%





\medskip

\noindent {\bf Acknowledgement.}
We  would like to thank the referees for  a very careful reading of the paper and for useful hints  to improve and make complete the present work.

\end{document}